\documentclass[12pt]{amsart} 
\usepackage[mathscr]{eucal} 
\usepackage{amsmath,amsfonts} 

\parskip=\smallskipamount 
 
\newtheorem{theorem}{Theorem}[section] 
\newtheorem{proposition}[theorem]{Proposition} 
\newtheorem{corollary}[theorem]{Corollary} 
\newtheorem{lemma}[theorem]{Lemma} 
 
\theoremstyle{definition} 
\newtheorem{definition}[theorem]{Definition} 
\newtheorem{example}[theorem]{Example}

\newtheorem{remark}[theorem]{Remark}

\makeatletter 
\@addtoreset{equation}{section} 
\makeatother

\newcommand{\CC}{{\mathbb C}} 
\newcommand{\NN}{{\mathbb N}} 
\newcommand{\FF}{{\mathbb F}}

\newcommand{\RR}{{\mathbb R}} 
\newcommand{\KK}{{\mathbb K}}
 
\newcommand{\cA}{{\mathcal A}} 
\newcommand{\cB}{{\mathcal B}} 
\newcommand{\cC}{{\mathcal C}} 
\newcommand{\cD}{{\mathcal D}} 
\newcommand{\cE}{{\mathcal E}} 
\newcommand{\cF}{{\mathcal F}} 
\newcommand{\cG}{{\mathcal G}} 
\newcommand{\cH}{{\mathcal H}} 
\newcommand{\cJ}{{\mathcal J}} 
\newcommand{\cK}{{\mathcal K}} 
\newcommand{\cL}{{\mathcal L}}

\newcommand{\cT}{{\mathcal T}}

 \newcommand{\bD}{\mathbf{D}}
 \newcommand{\bE}{\mathbf{E}}
 \newcommand{\bG}{\mathbf{G}}
 \newcommand{\bH}{\mathbf{H}}
 \newcommand{\bK}{\mathbf{K}}
 
 \newcommand{\bP}{\mathbf{P}}
 \newcommand{\bS}{\mathbf{S}}
 \newcommand{\bT}{\mathbf{T}}
 \newcommand{\bU}{\mathbf{U}}
 \newcommand{\bV}{\mathbf{V}}
 \newcommand{\bah}{\mathbf{h}}
 \newcommand{\bs}{\mathbf{s}}
 \newcommand{\bt}{\mathbf{t}}
 
 \newcommand{\bv}{\mathbf{v}}
 
 \newcommand{\fH}{\mathfrak{H}}
 \newcommand{\fP}{\mathfrak{P}}
\newcommand{\dom}{\operatorname{Dom}} 
 
\newcommand{\Ra}{\Rightarrow} 
\newcommand{\ran}{\operatorname{Ran}} 
 
\newcommand{\ra}{\rightarrow} 
\newcommand{\ol}{\overline}

\let\phi=\varphi 
 
\renewcommand{\ker}{\operatorname{Ker}}

\newcommand{\clos}{\operatorname{Clos}}

\newcommand{\supp}{\operatorname{supp}}
\newcommand{\diag}{\operatorname{diag}}

\newcommand{\lin}{\operatorname{Lin}}

\newcommand{\nr}[1]{\vspace{0.1ex}\noindent\hspace*{12mm}\llap{\textup{(#1)}}} 
 
\begin{document} 
\title[Positive Semidefinite Maps on $*$-Semigroupoids]{Partially Positive Semidefinite 
Maps on $*$-Semigroupoids and Linearisations} 
 
\keywords{$*$-semigroupoid, $*$-algebroid, positive semidefinite, completely positive, dilation, $*$-representation}

\subjclass[2010]{Primary 47L75; Secondary 43A35, 47A20, 47L60, 46L99}
 
\author[A. Gheondea]{Aurelian Gheondea}

\address{Department of Mathematics, Bilkent University, 06800 Bilkent, Ankara, 
  Turkey, \emph{and} Institute of Mathematics of the Romanian Academy,
  Calea Grivi\c tei 21, 010702 Bucure\c sti, Rom\^ania} 
\email{aurelian@fen.bilkent.edu.tr \textrm{and} A.Gheondea@imar.ro} 

\author[B. Udrea]{Bogdan Udrea}

\address{Institute of Mathematics of the Romanian Academy, 
  Calea Grivi\c tei 21, 010702 Bucure\c sti, Rom\^ania}
  \email{bogdanudrea75@yahoo.com}

\begin{abstract} Motivated by Cuntz-Krieger-Toeplitz systems associated to undirected 
graphs and representations of groupoids,
we obtain a generalisation of the Sz-Nagy's Dilation Theorem for operator valued partially 
positive semidefinite maps on $*$-semigroupoids with unit, with varying degrees of aggregation, firstly by 
$*$-representations with unbounded operators and then 
we characterise the existence of the corresponding $*$-representations by bounded operators.
By linearisation of these constructions, 
we obtain similar results for operator valued partially positive semidefinite maps on $*$-algebroids 
with unit and then, for the special case of $B^*$-algebroids with unit, we obtain a generalisation of the Stinespring's 
Dilation Theorem. As an application of the generalisation of the Stinespring's Dilation Theorem, we show that some
natural questions on $C^*$-algebroids are equivalent.
\end{abstract} 
\maketitle 
 
\section{Introduction}\label{s:i}

The modern theory of $C^*$-algebras was heavily influenced by the seminal paper of J.~Cuntz and 
W.~Krieger \cite{CuntzKrieger} that, in particular, shows how to associate $C^*$-algebras to certain 
undirected graphs; see, for example,
the monograph of I.~Raeburn \cite{Raeburn} and the bibliography cited there. 
This association has an important
combinatorial trait and it is usually performed by a so-called Cuntz-Krieger system 
associated to an undirected 
graph $G$, and its generalisation, the Cuntz-Krieger-Toeplitz system associated to $G$, 
see Definition~\ref{ex:ckt}. Roughly speaking, 
these are systems of partial isometries satisfying certain conditions on a Hilbert space 
$\cH$, that should be infinite dimensional in most of the relevant cases. Our main 
question, concerning these systems, asks whether they can be viewed as special cases 
of representations of certain mathematical objects onto some concrete forms. In this 
respect, one of the main motivation for this article is to show that indeed, these 
mathematical objects can be taken as $*$-semigroupoids and that 
the concrete forms can be 
taken as $*$-algebroids made up by linear operators, with a special concern on 
$C^*$-algebroids. In order to put these words into meaningful ideas, in the following we 
firstly recall the main sources and connections that lead us to our answer.

Another very important influence on the modern theory of $C^*$-algebras was done
by the work of J.~Renault \cite{Renault} that associated $C^*$-algebras to certain groupoids. This can be seen, for
example, in the monographs of A.L.T.~Paterson \cite{Paterson} and of D.P.~Williams \cite{Williams} and the 
rich bibliography cited there. In this case,
one usually considers locally compact groupoids and the main tool for the construction of the associated 
$C^*$-algebras is the so-called Haar system of measures. Haar systems of measures are necessary
in order to produce certain $L^2$ spaces on which the left regular representation acts. 
But the existence of Haar systems of measures on
locally compact groupoids does not hold in general and this imposes heavy restrictions in the theory. From this point
of view, questions referring to existence of representations of groupoids on Hilbert spaces are natural to ask. There 
are many relations between groupoids and graph algebras, for example see A.~Kumjian, D.~Pask, I.~Raeburn, 
and J.~Renault \cite{KumjianPaskRaeburnRenault}.

Dilation theory is a domain that shows up in both operator theory and operator algebras, see the survey article of 
W.B.~Arveson \cite{Arveson1}, for example. Two of the most important landmarks in the domain of noncommutative 
dilation theory are that of
B.~Sz.-Nagy \cite{BSzNagy}, that generalises, on the one hand, 
the dilation theorem for positive semidefinite maps on commutative groups of M.A.~Naimark \cite{Naimark2} to 
operator valued maps on $*$-semigroups with unit and, on the other hand, the B.~Sz.Nagy's unitary dilation 
theorem \cite{SzNagy1}, and that of W.F.~Stinespring \cite{Stinespring}, 
that generalises the other commutative dilation theorem for semispectral measures of 
M.A.~Naimark \cite{Naimark1} 
to operator valued completely positive maps on $C^*$-algebras. A big difference between these two theorems is that 
in the Sz-Nagy's Dilation Theorem, in order to obtain representations by bounded operators, a boundedness 
condition is needed, while in the Stinespring's Dilation Theorem there is no 
boundedness condition at all.

The fact that Sz-Nagy's Dilation Theorem implies 
Stinespring's Dilation Theorem is rather simple and, essentially, it is based, on the 
one hand, on the fact that the dilation
$*$-representation corresponding to a $*$-semigroup has an inherent linearity property and, on the other hand, on 
the existence of square roots for positive elements in $C^*$-algebras. Although considered as rather distinct results 
referring to the nonlinear dilation theory and the linear dilation theory, respectively, these two dilation theorems 
have been eventually proven to be logically equivalent by H.F.~Szafraniec \cite{Szafraniec}; this equivalence 
holds even in a larger generality, see A.~Gheondea and B.E.~U\u gurcan \cite{GheondeaUgurcan}. The dificult
implication, from Stinespring's Dilation Theorem to Sz-Nagy's Dilation Theorem, passes through a step for 
constructing a weight function (here is where the boundedness condition is used), then a second step of
linearisation on a weighted $\ell^1$ type space that can be organised as $B^*$-algebra with unit, 
and a final step of passing to the enveloping $C^*$-algebra. This last step can be shortcut, 
as shown in \cite{GheondeaUgurcan}, by using an idea of W.B.~Arveson \cite{Arveson} on replacing 
$C^*$-algebras with $B^*$-algebras. 

On the one hand, the classical Sz-Nagy's Dilation
Theorem triggered a whole domain of investigations in operator theory by the monumental work of 
B.~Sz-Nagy and C.~Foia\c s \cite{SzNagyFoias} on contractions on Hilbert spaces but, explicitly, it was not pursued
for further investigations: on Mathscinet 
the original paper of B.~Sz.-Nagy has only $12$ citations by papers. However, 
multidimensional generalisations of Sz.-Nagy and Foia\c s investigations, such as G.~Popescu~\cite{Popescu}, were 
followed by other investigations on multidimensional dilation theory. 
From the point of view that we use in this article, we mention here 
the articles of M.T.~Jury and D.W.~Kribs \cite{JuryKribs},
E.~Katsoulis and D.W.~Kribs \cite{KatsoulisKribs},
A.~Dor-On and G.~Salomon \cite{DorOnSalomon}, to cite a few.
On the other hand,
the Stinespring's Dilation Theorem made a remarkably successful career 
(Mathscinet reports $402$ citations only in mathematical papers, but there are many more in quantum physics) 
especially because the concept of 
completely positive map turned out to be extraordinarily useful in operator algebras, operator systems, and in
modelling quantum operations, for example, see V.R.~Paulsen \cite{Paulsen},  M.~Hayashi \cite{Hayashi}. 

In this article we consider some aspects of dilation theory that may establish yet one more 
bridge between graph algebras and
groupoid algebras and, in the same time, shed some light on Cuntz-Krieger-Toeplitz systems. Our aim is to obtain 
analogous results to Sz-Nagy's and Stinespring's dilation theorems and their interplay.
The fundamental concepts that we found useful in this enterprise are that of a $*$-semigroupoid with unit and that 
of operator valued partially positive semidefinite maps on $*$-semigroupoids. This is because, on the one hand, 
$*$-semigroupoids clearly generalise
groupoids while, on the other hand, we observe, see Example~\ref{ex:ckt}, 
that each Cuntz-Krieger-Toeplitz system associated to an 
undirected graph $G$ is actually a fully aggregated $*$-representation of the free $*$-semigroupoid $\FF^+_*(G)$.
In the following we explain the concept of aggregation that plays a major role in this enterprise.

A semigroupoid, as introduced by B.~Tilson \cite{Tilson}, 
is sometimes called either a semicategory, or a naked category, or a precategory,  
because it is very close to a 
small (that is, both the class of objects and morphisms are sets) category, except the assumption that there is 
an identity morphism to each object. However, a small category is a semigroupoid with unit, and this is basically the 
object that we are interested in. In this article, we use semigroupoid 
in the sense of Tilson, although there is a different definition introduced by R.~Exel \cite{Exel} which is not of 
categorial character, but catches better the original approach of J.~Cuntz and W.~Krieger \cite{CuntzKrieger}.
The idea of semigroupoid is heavily related to that of partial action, in the sense that only 
some pairs of elements can be multiplied, while a
Cuntz-Krieger-Toeplitz system consists of a system of operators acting on the same Hilbert space and hence any two
of these operators can be multiplied. It is in this sense that we use the concept of full aggregation. 
To be more precise, we allow a 
certain freedom of aggregation to representations of semigroupoids by introducing an aggregation map,
ranging from full aggregation, corresponding to a single Hilbert space, when the range of the aggregation map 
is a singleton, to no aggregation at all, corresponding to the widest bundle of Hilbert spaces, 
when the aggregation map is injective.

In the following we briefly explain the organisation of this article and point out the main results.
In Section~\ref{s:s} we review the basic concepts related to semigroupoids with an emphasise on $*$-semigroupoids,
present two operator models, one by unbounded operators and another one by bounded operators, and many 
relevant examples, and then 
define operator valued $*$-representations of $*$-semigroupoids by unbounded operators 
and bounded operators, respectively. A special case of a $*$-semigroupoid that turns out to be very useful for our 
investigations is that of an inverse semigroupoid, see V.~Liu \cite{Liu}, for example.
We also show that Cuntz-Krieger-Toeplitz systems associated to undirected 
graphs can be regarded as fully aggregated $*$-representations of free $*$-semigroupoids generated by 
those graphs. This observation provides one of the motivations for this research.

The main results of this article are obtained in Theorem~\ref{t:nagyu} and Theorem~\ref{t:nagy} that refer to
dilation of a partially positive semidefinite map on a $*$-semigroupoid with unit to a
$*$-representation of the $*$-semigroupoid, 
generalising the classical dilation theorem of B.~Sz.-Nagy. The reason that there are two dilation 
theorems is that we separated the case of representation by unbounded operators, see for the example the 
monograph of K.~Schm\"udgen \cite{Schmudgen} and the rich bibliography cited there, from the representation by
bounded operators, which requires an additional boundedness condition. This boundedness condition shows up 
since the classical paper of B.~Sz-Nagy \cite{BSzNagy} article and is quite natural, if $*$-representations with 
bounded operators is what we want. At this level of generality, the dilations that 
are obtained have a certain orthogonality property that is closely related to the characteristics of a 
Cuntz-Krieger-Toeplitz system. In general, the fashion in which different pieces of the representation 
are aggregated within a bundle of Hilbert spaces makes technical obstructions. 
This is clearly seen by the interpretation of the Cuntz-Krieger-Toeplitz systems, see Example~\ref{ex:ckt}, 
as a fully aggregated representation of the $*$-semigroupoid $\FF^+_*(G)$. 
This orthogonality property shows its importance because, when combined with a minimal property, it
implies uniqueness up to unitary equivalence. Also, it is shown that for the case of inverse 
semigroupoids with unit, in particular for groupoids, positive semidefiniteness is equivalent to existence of 
dilations with bounded operators and, in addition, the dilation is a $*$-representation made up by mutually orthogonal
partial isometries, see Corollary~\ref{c:invsem}, that resembles very much a Cuntz-Krieger-Toeplitz system.

Algebroids have been considered since J.~Pradines \cite{Pradines}, for the special case of a
Lie algebroid, and G.H.~Mosa \cite{Mosa}, for the purely algebraic counterpart. In view of the interest for
graph algebras generated through free semigroupoids, e.g.\ see A.~Kumjian, D.~Pask, and I.~Raeburn 
\cite{KumjianPaskRaeburn},  here we want to connect these 
investigations to graph algebras via the free $*$-semigroupoid of an undirected graph. Our interest is also related to
connecting semigroupoid algebras with dilation theory, and here we have to mention the pioneering work of 
D.W.~Kribs and S.C.~Power \cite{KribsPower}, and the connection of semigroups algebras with interpolation 
problems, cf.\ M.A.~Dritschel, S.~Marcantognini, and S.~McCullogh \cite{DritschelMarcantogniniMcCullogh}.
In the last section we examine linearisations of partially positive semidefinite maps and this leads to considering 
$*$-algebroids and $B^*$-algebroids, the latter being 
Banach algebroids with isometric involutions. For the general case 
of $*$-algebroids with unit we obtain an analogue of the Stinespring's dilation with unbounded operators in 
Theorem~\ref{t:stinespringu}. For the special case of partially 
positive semidefinite maps of $B^*$-algebroids with unit we 
obtain Stinesprig's dilations with bounded operators in Theorem~\ref{t:stinespring}. For this, we need
generalisations of the concept of amplification of a $*$-algebroid, the concept of amplification of a certain linear map
between $*$-algebroids, and the concept of operator valued completely positive map on $*$-algebroids, 
in the spirit of the classical result of Stinespring \cite{Stinespring}. The point here is that, 
in a $*$-algebroid, positivity is defined by means of a convex cone and that, even for a $B^*$-algebroid, this cone 
may be larger than the cone of positive elements in the isotropy algebras. 

We end by a subsection where we firstly single out the concept of a $C^*$-algebroid which, 
roughly speaking, is a $B^*$-algebroid for which the norms satisfy a certain $C^*$-algebra condition. The model is
the $B^*$-algebroid defined by a bundle of Hilbert spaces and their bounded operators but there are some 
fundamental
questions that one has to answer on this issue. Firstly, on each isotropy $C^*$-algebra there are two cones of 
positive elements, one intrinsic to the $C^*$-algebra structure and a second one induced by the concept of partially 
positive semidefiniteness, and the first question is whether the two cones coincide. The second question refers to
the amplifications of the fibers and asks to which extent are they $C^*$-algebroids. Lastly, one has to find to which 
extent an analogue of the Gelfand-Naimark Theorem does hold for $C^*$-algebroids, that is, whether any 
$C^*$-algebroid can be embedded in a $C^*$-algebroid defined by a bundle of Hilbert spaces.
As an application of Theorem~\ref{t:stinespring}, these questions are shown to be equivalent. We show in Proposition~\ref{p:rich} that a positive result to these questions might depend on how rich the fibres are, but the general problems remain open for now.

Finally, an observation on terminology. 
In this article we use the word \emph{bundle} or \emph{fibration} to denote simply a family of 
objects, called fibres, indexed on some nonempty
set, without any other structure, topological property, or measurability property, required. In general, we will use the
notation $\{O_j\}_{j\in J}$, where $J$ is the index set and $O_j$'s are the fibers of the bundle. An equivalent way
of defining a bundle, and this is often seen in the literature, is by a surjective function $s\colon O\ra J$, for some 
nonempty sets $O$ and $J$. Obviously, letting the fibre $O_j:=s^{-1}(\{j\})$ this is equivalent to the notation 
$\{O_j\}_{j\in J}$. The fibres may be sets, spaces, operators, and so on.

\section{Preliminaries on Semigroupoids}\label{s:s}
 
In this section we review the basic definitions, relevant examples, and some simple facts on semigroupoids and, 
especially, $*$-semigroupoids with unit. We also introduce two types of operator models for $*$-semigroupoids
with unit, by unbounded and bounded operators, respectively, that will play a major role in this article.

\subsection{Semigroupoids and Examples.}\label{ss:se} We firstly recall some basic definitions and 
examples on semigroupoids, with an emphasise on $*$-semigroupoids.

\begin{definition} A \emph{semigroupoid} \cite{Tilson} is a quintuple $(\Gamma;S;d;c;\cdot)$ subject to the following conditions.
\begin{itemize}
\item[(SG1)] $\Gamma$ and $S$ are nonempty sets.
\item[(SG2)] $d\colon \Gamma\ra S$ and $c\colon \Gamma\ra S$ are maps. 
\item[(SG3)] For every $\alpha,\beta\in\Gamma$ such that $d(\alpha)=c(\beta)$ there exists a unique
element $\alpha\cdot\beta\in\Gamma$, 
with $d(\beta)=d(\alpha\cdot\beta)$ and $c(\alpha\cdot\beta)=c(\alpha)$.
\item[(SG4)] For any $\alpha,\beta,\gamma\in\Gamma$ such that $d(\beta)=c(\gamma)$ 
and $d(\alpha)=c(\beta)$ we have 
\begin{equation*}\alpha\cdot (\beta\cdot\gamma)=(\alpha\cdot\beta)\cdot\gamma.\end{equation*}
\end{itemize}
\end{definition}

The map $d$ is called the \emph{domain map} or the \emph{source map} and $c$ is called the 
\emph{codomain map} or the \emph{range map}.

\begin{remark} For simplicity, we will say that $\Gamma$ is a semigroupoid and, by this, we mean that 
there exist the set $S_\Gamma$, the maps $d_\Gamma$ and $c_\Gamma$, 
and the operation $\cdot$ 
such that the quintuple $(\Gamma;S_\Gamma;d_\Gamma;c_\Gamma;\cdot)$ 
satisfies the properties (SG1) through (SG4). When there is no danger of confusion, 
we will drop the lower index $\Gamma$.

Also, for simplicity we will write $\alpha\beta$ instead of $\alpha\cdot\beta$, 
whenever the operation is possible. In this respect, we denote
\begin{equation*}
\Gamma^{(2)}:=\{(\alpha,\beta)\mid \alpha,\beta\in\Gamma,\ c(\beta)=d(\alpha)\},
\end{equation*}
the set of composable pairs in $\Gamma$. 

In addition, for each $s,t\in S$ we consider the fibres
\begin{equation*}
\Gamma^t:=\{\alpha\in \Gamma\mid c(\alpha)=t\},\quad \Gamma_s:=\{\alpha\in \Gamma
\mid d(\alpha)=s\},\quad \Gamma_s^t:=\Gamma_s\cap \Gamma^t.
\end{equation*}
In general, any of these fibres can be empty. Let $S_0:=\{ s\in S\mid \Gamma^s=\emptyset=\Gamma_s\}$, the set
of isolated points in $S$. By replacing $S$ with $S\setminus S_0$ we remove those points in $S$ that are neither the 
domain nor the codomain of any element $\gamma\in \Gamma$. Without loss of generality, we can thus assume that
for any $s\in S$ either $\Gamma_s$ or $\Gamma^s$ is nonempty.

If $\Gamma_s\neq\emptyset$ for all $s\in S$ then 
$\{\Gamma_s\mid s\in S\}$ is a partition of $\Gamma$ and hence it defines an 
equivalence relation on $\Gamma$: $\alpha\sim_d\beta$ if $d(\alpha)=d(\beta)$. 
Similarly,  if $\Gamma^s\neq \emptyset$ for all $s\in S$ then
$\{\Gamma^s\mid s\in S\}$ is a partition of $\Gamma$ and hence it defines an
equivalence relation on $\Gamma$: $\alpha\sim_c\beta$ if $c(\alpha)=c(\beta)$.
 
 Also, $\Gamma_s^t$ may be empty for some $s,t\in S$. 
 If $\Gamma_s^t$ is not empty for each $(s,t)\in S\times S$,
 equivalently, the map $(d,c)\colon \Gamma\ra S\times S$ is surjective,
 the semigroupoid is called \emph{transitive} and, in this case, $\{\Gamma_s^t\mid s,t\in S\}$ 
 is a partition of $\Gamma$ which defines the equivalence relation: $\alpha\sim_{d,c}\beta$ if 
 $d(\alpha)=d(\beta)$ and $c(\alpha)=c(\beta)$.
 If $s\in S$ and $\Gamma_s^s$ is not empty then $\Gamma_s^s$ is a semigroup, 
 called the \emph{isotropy semigroup} at $s$.
\end{remark}

\begin{example} (\emph{Transformation Semigroupoids}) Let $G$ be a semigroup which acts on a nonempty set $X$ 
to the right, with the right action denoted $x\cdot g$, for $x\in X$ and $g\in G$. In order to keep the construction away 
of trivial cases, without loss of generality
we can assume that for any $y\in X$ there exist $g\in G$ and $x\in X$ such that $y=x\cdot g$.
 Let $d\colon X\times G\ra X$  be defined by $d(x,g):=x\cdot g$ 
and $c\colon X\times G\ra X$ defined by $c(x,g):=x$. By definition, $(x,g)$  is 
composable with  
$(y,h)$ if $y=x\cdot g$ and, in this case, we define 
$(x,g)(x\cdot g,h):=(x,gh)$. Then $(X\times G;X;d;c;\cdot)$
is a semigroupoid called the \emph{transformation semigroupoid} induced by the right action of the
semigroup $G$ on the set $X$.
\end{example}

\begin{definition} A semigroupoid $\Gamma$ has a \emph{unit} if there exists an injective map 
$\epsilon \colon S_\Gamma\ra \Gamma$, for which we use the notation $\epsilon=\{\epsilon_s\}_{s\in S_\Gamma}$, 
subject to the following conditions.
\begin{itemize}
\item[(U1)] For every $s\in S_\Gamma$, $d_\Gamma(\epsilon_s)=c_\Gamma(\epsilon_s)=s$.
\item[(U2)] For any $s\in S_\Gamma$ and any $\alpha\in\Gamma^s$ we have 
$\epsilon_s\alpha=\alpha$.
\item[(U3)] For any $s\in S_\Gamma$ and any $\alpha\in \Gamma_s$ we have 
$\alpha \epsilon_s=\alpha$.
\end{itemize}
\end{definition}

\begin{remark}
It is easy to see that, a unit $\epsilon$ of the semigroupoid $\Gamma$, if it exists, is unique. In addition,
since the unit $\epsilon$ is an injective map this yields an embedding of $S_\Gamma$ in $\Gamma$.
In this case, one usually denotes $\Gamma^0:=\{\epsilon_s\mid s\in S_\Gamma\}$, 
and call it the \emph{set of units} of $\Gamma$. Clearly, $\Gamma^0$ is in a bijective correspondence
with $S_\Gamma$ and, because of that, for semigroupoids with unit, one usually identifies them.

If the semigroupoid $\Gamma$ has a unit then, for each $s\in S_\Gamma$ 
the isotropy semigroup $\Gamma_s^s$ has a unit.

A semigroupoid with unit can be defined, equivalently, as a small category, that is, a category in which
the classes of objects and arrows are sets.
\end{remark}

\begin{remark} (\emph{Semigroupoids as Directed Graphs})\label{ex:sdg} 
A directed graph is, by definition, a quadruple
$(V;E;s;r)$, where $V$ is the set of vertices, $E$ is the set of edges, and $s,r\colon E\ra V$ are the source map
and the range map, respectively. We allow that between two vertices there may be several edges.
Any semigroupoid $\Gamma$ can be viewed as a directed graph 
$(V;E;s;r)$ in which the vertex set $V=S_\Gamma$, the edge set $E=\Gamma$, the source map $s=d$ and the 
range map $r=c$, and such that the edge set is endowed with a partial 
binary operation denoted by juxtaposition $(e,f)\mapsto ef$ satisfying the following conditions.
\begin{itemize}
\item[(i)] If $e,f\in E$ then $ef$ exists if $(e,f)\in\Gamma^{(2)}$, 
equivalently, $r(f)=s(e)$.
\item[(ii)] If $(e,f)\in \Gamma^{(2)}$ then $s(ef)=s(f)$ and $r(ef)=r(e)$.
\item[(iii)] If $(e,f)\in\Gamma^{(2)}$ and $(f,g)\in \Gamma^{(2)}$ then $e(fg)=(ef)g$.
\end{itemize}
\end{remark}

\begin{definition}\label{d:frees} A semigroupoid $\Gamma$ is \emph{free} if there exists a set $A\subset \Gamma$, 
called the set of \emph{symbols},
such that
for every $\gamma\in \Gamma\setminus A$ that is not a unit, 
there exist unique $\alpha_1,\ldots,\alpha_n\in A$ such that 
$\gamma=\alpha_1\cdots \alpha_n$.
The elements $\gamma\in\Gamma$ are then called \emph{words} and, if
$\gamma=\alpha_1\cdots \alpha_n$, with $\alpha_j\in A$ for all $j=1,\ldots,n$, then $n$ is the \emph{length} of 
$\gamma$. If $\gamma\in A$ then its length is $1$. If the free semigroupoid $\Gamma$ has a unit 
$\epsilon=\{\epsilon_s\}_{s\in S}$ then for any $s\in S$, the length of $\epsilon_s$ is $0$.
\end{definition}

\begin{example} (\emph{Free Graph Semigroupoids})\label{ex:fs} 
Let $G=(V,E,s,r)$ be a directed graph.
 Without loss of generality, we
assume that any vertex is either the source or the range of some edges, so there are no isolated vertices.

Let $\FF^+(G)$ \cite{KribsPower}
denote the set of all finite paths of the directed graph $G$. More precisely, if $w_1, w_2,
\ldots,w_n$ is a sequence of edges such that $s(w_1)=r(w_2)$, $s(w_2)=r(w_3)$, through $s(w_{n-1})=r(w_n)$,
then $\omega=w_1\cdots w_n$ denotes the corresponding path. Then $r(w_1)=r(\omega)$ 
and $s(w_n)=s(\omega)$. The \emph{length} of $\omega$ is $n$.
Two finite paths 
$\alpha$ and $\beta$ can be concatenated if and only if $s(\alpha)=r(\beta)$ and then the concatenation
$\alpha\beta$ is defined in an obvious fashion, so the length of $\alpha\beta$ is $n+m$. Here we have to observe
that concatenation is following the same rule as juxtaposition as defined in Example~\ref{ex:sdg}.

It is easy to see that, letting $S=V$, the domain map $d$ be the source map $s$, 
the codomain map $c$ be the range map $r$, and  the operation of 
concatenation, we have a semigroupoid structure on $\FF^+(G)$, called the \emph{free semigroupoid} 
associated to the graph $G$. The set of edges $E$, identified with the set of paths of length $1$, is the set of 
symbols of $\FF^+(G)$. If we allow
any vertex $x$ to have a special loop $e_x$ such that $e_x f=f$ for any edge $f$ with $r(f)=x$ and
$ge_x=g$ for all edge $g$ with $s(g)=x$, then $e=\{e_x\mid x\in V_G\}$ is the unit of $\FF^+(G)$. In this case,
we actually identify the set $V$ of vertices with the unit $e$.
\end{example}

\begin{definition}
Given a semigroupoid $\Gamma$, an \emph{involution} on $\Gamma$ is a map 
$\Gamma\ni \alpha\mapsto \alpha^*\in\Gamma$ subject to the following conditions.
\begin{itemize}
\item[(I1)] For any $\alpha\in\Gamma$ we have $d(\alpha^*)=c(\alpha)$ and $c(\alpha^*)=d(\alpha)$.
\item[(I2)] For any $(\alpha,\beta)\in \Gamma^{(2)}$ we have $(\alpha\beta)^*=\beta^*\alpha^*$.
\item[(I3)] For any $\alpha\in\Gamma$ we have $(\alpha^*)^*=\alpha$.
\end{itemize}
A semigroupoid with a specified involution $*$ will be called a \emph{$*$-semigroupoid}.
\end{definition}

\begin{remark} By property (I3), involutions are bijective maps.  Also, if the
$*$-semigroupoid $\Gamma$ has a unit $\epsilon$ then 
\begin{equation*}
\epsilon_s^*=\epsilon_s,\quad s\in S_\Gamma.
\end{equation*}
Also, for any $s\in S_\Gamma$ 
such that $\Gamma_s^s\neq\emptyset$, the isotropy semigroup $\Gamma_s^s$ has an involution.
\end{remark}

\begin{example} (\emph{Free Graph $*$-Semigroupoids}) \label{ex:fsug}
With notation and assumptions as in Example~\ref{ex:fs}, assume that the graph $G=(V,E,s,r)$ is directed and we 
extend it to a graph $\tilde G=(V;\tilde E;s;r)$ as follows: to any 
edge $f\in E$ which is not a unit 
we add a unique companion edge $f^*$ with $s(f^*)=r(f)$ and $r(f^*)=s(f)$. For
the units $e_v$ we let $e_v^*=e_v$. Also, we define $(f^*)^*=f$ for any companion edge $f^*$. Let $E^*$ denote 
the collection of all edges added in this way from the edges in $E$ and let $\tilde E=E\cup E^*$. We let 
$\FF^+_*(G)$ be the free semigroupoid with unit induced by $\tilde G$.
Then, for any finite path $\omega=w_1\cdots w_n\in\FF^+_*(G)$ we can uniquely define the finite path 
$\omega^*=w_n^*\cdots w_1^*$ and, in this way, the free  $*$-semigroupoid with unit $\FF^+_*(G)$ is 
obtained.
\end{example}

\begin{definition} A semigroupoid $\Gamma$ is called an \emph{inverse semigroupoid}, e.g.\ see \cite{Liu}, if for any 
$\alpha\in\Gamma$ there exists a unique $\alpha^\prime\in\Gamma$ such that 
$\alpha\alpha^\prime\alpha=\alpha$ 
and $\alpha^\prime\alpha\alpha^\prime=\alpha^\prime$. Note that, in particular, this means that
$d_\Gamma(\alpha^\prime)=c_\Gamma(\alpha)$ and 
$c_\Gamma(\alpha^\prime)=d_\Gamma(\alpha)$, for all $\alpha\in \Gamma$.
\end{definition}

\begin{remark} \label{r:isi} 
It is easy to see that if $\Gamma$ is an inverse semigroupoid then for any $\alpha\in\Gamma$ we have
$(\alpha^\prime)^\prime=\alpha$ and for any $(\alpha,\beta)\in\Gamma^{(2)}$ we have 
$(\alpha\beta)^\prime=\beta^\prime\alpha^\prime$. In particular, any inverse semigroupoid is a 
$*$-semigroupoid with $\alpha^*:=\alpha^\prime$.
\end{remark}

\begin{definition}\label{d:action} 
Let $\Gamma$ be a semigroupoid and $X$ a nonempty set. A \emph{left action} of
$\Gamma$ on $X$ is a pair $(a;\cdot)$ subject to the following conditions.
\begin{itemize}
\item[(A1)] $a\colon X\ra S_\Gamma$ is a map.
\item[(A2)] For any $x\in X$ and any $\alpha\in\Gamma_{a(x)}$
there exists a unique element $\alpha\cdot x\in X$ such that $a(\alpha\cdot x)=c(\alpha)$.
\item[(A3)] For any $x\in X$ and any $(\alpha,\beta)\in\Gamma^{(2)}$ such that 
$\beta\in\Gamma_{a(x)}$ we have
\begin{equation*}
(\alpha\beta)\cdot x=\alpha\cdot(\beta\cdot x).
\end{equation*}
\end{itemize}
If, in addition, $\Gamma$ has a unit $\epsilon$ then the left action $(a;\cdot)$ is called \emph{unital} if
\begin{equation*}
\epsilon_{a(x)}\cdot x=x,\quad x\in X.
\end{equation*}
\end{definition}

\begin{remark} (1) If we replace $X$ with $a(S)$, in this way we remove those points in $X$ on which no action
of $\Gamma$ occurs and hence, without loss of generality, we can assume that the map $a$ is surjective. The map 
$a$ is called \emph{anchor}. 

(2) Any semigroupoid $\Gamma$ has a natural left action on itself. To see this, 
let $a\colon \Gamma\ra S_\Gamma$ be defined by $a(\beta):=c(\beta)$. In addition,
for any $\alpha\in\Gamma$ and any $\beta\in \Gamma_{a(\alpha)}$ the action is defined by 
$\beta\cdot \alpha:=\beta \alpha$.

(3) The definition of a left action of a semigroupoid $G$ over a set $X$ is natural. For the case of 
topological groupoids, more conditions are assumed, see e.g.\ \cite{MuhlyRenaultWilliams}.
\end{remark}

\begin{example}\label{ex:segah} (\emph{Semigroupoids Modelled by Bounded Operators on Hilbert Spaces})
Let $S$ be a nonempty set and $\bH=\{\cH_s\}_{s\in S}$ a \emph{bundle of Hilbert spaces}
over $\KK$ (either $\RR$ or $\CC$) for some nonempty set $S$. An element 
$\bah=\{h_s\}_{s\in S}$ such that $h_s\in\cH_s$ for all $s\in S$ will be called a
\emph{cross-section} of the bundle $\bH$.
Let 
\begin{equation}\label{e:gamabeh}
\Gamma_\bH=\bigsqcup_{s,t\in S} \cB(\cH_s,\cH_t),
\end{equation}
where $\cB(\cH_s,\cH_t)$ denotes the vector space of all bounded linear operators 
$T\colon \cH_s\ra\cH_t$.

For every $T\in \Gamma_\bH$, there exists uniquely $s,t\in S$ such that 
$T\in\cB(\cH_s,\cH_t)$ and then
we define $d(T)=s$ and $c(T)=t$. In this way, $d\colon\Gamma_\bH\ra S$ and 
$c\colon \Gamma_\bH\ra S$ are surjective maps.

The operation of composition in $\Gamma_\bH$ 
is defined by operator composition: if $A,B\in\Gamma_\bH$
are such that $d(A)=c(B)=u\in S$ then $A\in\cB(\cH_u,\cH_t)$ for some $t\in S$ and 
$B\in\cB(\cH_s,\cH_u)$ for some $s\in S$. Then $AB\in \cB(\cH_s,\cH_t)$
is the composition of the 
operators $A$ and $B$, in this order. Thus, $(\Gamma_\bH,S,d_\bH,c_\cH,\cdot)$ is a semigroupoid.

The semigroupoid $\Gamma_\bH$ has the unit $\epsilon\colon S\ra\Gamma_\bH$ defined by 
$\epsilon_s:=I_{\cH_s}$, where $I_{\cH_s}$ denotes the identity operator on $\cH_s$, for all $s\in S$.
It also has a natural involution: $\Gamma_\bH\ni T\mapsto T^*\in\Gamma_\bH$ where, 
if $T\in\cB(\cH_s,\cH_t)$ by $T^*\in\cB(\cH_t,\cH_s)$ we denote the Hilbert space adjoint operator.

In addition, the semigroupoid $\Gamma_\bH$ has a natural left action on $S$ defined by 
$T\cdot s:=t$, where $T\in\cB(\cH_s,\cH_t)$.
\end{example}

\begin{example} (\emph{Semigroupoids Modelled on Vector Spaces}) 
\label{ex:semes} Let $S$ be a nonnempty 
set and $\bD=\{\cD_s\}_{s\in S}$ be a bundle of vector spaces over $\KK$. We let
\begin{equation*}
\Gamma_\bD=\bigsqcup_{s,t\in S}\cL(\cD_s,\cD_t),
\end{equation*}
where by $\cL(\cD_s,\cD_t)$ we understand the vector space of all linear operators 
$T\colon \cD_s\ra\cD_t$. As in the previous example, for any $T\in \Gamma_\bD$ there exists
uniquely $s,t\in S$ such that $T\in\cL(\cD_s,\cD_t)$ and we let $d(T):=s$ and $c(T):=t$, hence
$d\colon \Gamma_\bD\ra S$ and $c\colon \Gamma_\bD\ra S$ are surjections. The partial
composition is defined as in the previous example and, in this way, $\Gamma_\bD$ becomes
a semigroupoid with unit $\epsilon\colon S\ra\Gamma_\bD$, $\epsilon(s):=I_{\cD_s}$. As in the 
previous example, there is a natural left action of $\Gamma_\bD$ onto $S$.
\end{example}

\begin{example}\label{ex:semuh}
 (\emph{Semigroupoids Modelled by Unbounded Operators in Hilbert Spaces.})
\label{ex:semeseh} 
Let us assume, with notation as in the previous example,
that for each $s\in S$ the vector space 
$\cD_s$ is a dense subspace of a Hilbert space $\cH_s$ and let
\begin{equation}\label{e:gabed}
\Gamma_{\bH;\bD}:=\bigsqcup_{s,t\in S} \cL^*(\cD_s,\cD_t),
\end{equation}
where, for each $s,t\in S$, we let $\cL^*(\cD_s,\cD_t)$ denote the vector space of all
linear operators $T\colon \cD_s(\subseteq \cH_s)\ra\cH_t$ subject to the following assumptions.
\begin{itemize}
\item[(i)] $T\cD_s\subseteq \cD_t$.
\item[(ii)] $ \cD_t\subseteq \dom(T^*)$ and $T^*\cD_t\subseteq\cD_s$.
\end{itemize}
Here, by $T^*\colon\dom(T^*)(\subseteq \cH_t)\ra\cH_s$ we understand 
the adjoint operator (possibly unbounded) defined in the usual sense,
\begin{equation*}
\dom(T^*):=\{k\in \cH_t\mid \cD_s\ni h\mapsto \langle Th,k\rangle_{\cH_t}\mbox{ is bounded}\},
\end{equation*}
and
\begin{equation*}
\langle Th,k\rangle_{\cH_t}=\langle h,T^*k\rangle_{\cH_s},\quad h\in\cD_s,\quad k\in\dom(T^*).
\end{equation*}
Note that, in this way, any operator $T$ in $\cL^*(\cD_s,\cD_t)$ is closable. The involution on $\Gamma_{\bH,\bD}$ is
defined by $\cL^*(\cD_s,\cD_t)\ni T\mapsto T^*|_{\cD_t}$,  for all $s,t\in S$.

Then, it is easy to see that $\Gamma_{\bH;\bD}\ni T\mapsto T^*\in\Gamma_{\bH;\bD}$ is an involution
which makes $\Gamma_{\bH;\bD}$ a $*$-semigroupoid with unit. As in the previous example, 
there is a natural left action of $\Gamma_{\bH;\bD}$ onto $S$.
\end{example}

\begin{definition}
A semigroupoid $\Gamma$ is called a \emph{groupoid} \cite{Brandt} if it has a unit $\epsilon$ and there exists a 
map $\Gamma\ni \alpha\mapsto \alpha^{-1}\in \Gamma$ such that, for every $\alpha\in \Gamma$,
we have $d(\alpha^{-1})=c(\alpha)$, $c(\alpha^{-1})=d(\alpha)$, and
\begin{equation}
\alpha\alpha^{-1}=\epsilon_{c(\alpha)},\quad \alpha^{-1}\alpha=\epsilon_{d(\alpha)}.
\end{equation}
\end{definition}

If $\Gamma$ is a groupoid then its unit set 
$\Gamma^0=\{\alpha^{-1}\alpha\mid \alpha\in\Gamma\}$ is naturally identified with the set 
$S_\Gamma$, and there is no need to specify $S_\Gamma$ beforehand, which is done by 
some authors. 

Any groupoid is an inverse semigroupoid with unit, when defining $\alpha^\prime=\alpha^{-1}$, in 
particular a  $*$-semigroupoid with unit, when defining the involution $\alpha^*:=\alpha^{-1}$.

\begin{example} (\emph{Transitive Relations as Semigroupoids})\label{ex:releg} 
Let $S$ be a nonempty set and consider $R\subseteq S\times S$ a relation on $S$.
Let $\dom(R) :=\{s\in S\mid (t,s)\in R\mbox{ for some }t\in S\}$,
$\ran(R) :=\{t\in S\mid (t,s)\in R\mbox{ for some }s\in S\}$.
Define the domain map $d\colon S\times S\ra S$ by $d(s,t):=t$ and, similarly, the codomain map
$c\colon S\times S\ra S$ by $c(s,t):=s$, for any $(s,t)\in S\times S$.
A pair $((s,t),(u,v))\in (S\times S)\times (S\times S)$ is composable, by definition, if $t=u$ and, in this case,
let $(s,t)\cdot (u,v):=(s,v)$.

With respect to these maps and composition, the relation $R\subseteq S\times S$
is a semigroupoid if and only if $R$ is transitive.
In addition, assuming that the relation $R$ is a semigroupoid as above, in particular $R$ is transitive, 
then:
\begin{itemize}
\item $R$ has a unit if and only if it is reflexive.
\item $R$ is a groupoid if and only if it is reflexive and symmetric, equivalently, it is an equivalence 
relation.
\end{itemize}
\end{example}

\begin{example}\label{ex:fg} (\emph{Free Graph Groupoids}) Let $G=(V,E)$ be an undirected graph. This means
that any edge $f$ can be considered in a double manner, and we denote the companion of any 
edge $f$ by $f^{-1}$, with $s(f^{-1})=r(f)$ and $r(f^{-1})=s(f)$. For any vertex $v$ we have the loop $e_v$ such 
that $e_v=e_v^{-1}$. Let $\FF(G)$ denote the set of all finite paths over the undirected graph $G$, 
as in Example~\ref{ex:fs}, but in such a way that, when considering
a arbitrary edge $f$ as a finite path of length $1$, for the operation of concatenation we have the cancellation rules 
$ff^{-1}=e_{r(f)}$ and $f^{-1}f=e_{s(f)}$, by definition. 
Then, with any finite paths $\omega,\gamma\in\FF(G)$, whenever
the concatenation $\omega\gamma$ is possible, we apply the cancellation rules as before. In this way, $\FF(G)$
is a groupoid, called the \emph{free groupoid} generated by the undirected graph $G$. 
\end{example}

\subsection{Morphisms of Semigroupoids and Representations.}\label{ss:msr}
Since the semigroupoids considered in this article are modelled by small categories, the concept of semigroupoid
morphism is modelled by that of a functor, hence it has two components, the map between objects and the 
map between arrows.

\begin{definition}\label{d:sm} Let $\Gamma$ and $\Lambda$ be two semigroupoids. A pair $(\phi;\Phi)$ is 
a \emph{semigroupoid morphism} from $\Gamma$ to $\Lambda$ if the following conditions hold.
\begin{itemize}
\item[(SM1)] $\phi\colon S_\Gamma\ra S_\Lambda$ and $\Phi\colon \Gamma\ra \Lambda$ are maps.
\item[(SM2)] For any $\alpha\in\Gamma$ we have $d_\Lambda(\Phi(\alpha))=\phi(d_\Gamma(\alpha))$ 
and $c_\Lambda(\Phi(\alpha))=\phi(c_\Gamma(\alpha))$.
\item[(SM3)] For any $(\alpha,\beta)\in \Gamma^{(2)}$ we have 
$(\Phi(\alpha),\Phi(\beta))\in\Lambda^{(2)}$ and
\begin{equation*}
\Phi(\alpha\beta)=\Phi(\alpha)\Phi(\beta).
\end{equation*}
\end{itemize}
The map $\phi$ is called the \emph{aggregation map}.

On the other hand, if both $\Gamma$ and $\Lambda$ are $*$-semigroupoids, the semigroupoid
morphism $(\phi;\Phi)$ is called a \emph{$*$-morphism} if
\begin{equation*}
\Phi(\alpha^*)=\Phi(\alpha)^*,\quad \alpha\in\Gamma.
\end{equation*}

The pair $(\phi;\Phi)$ is called a \emph{monomorphism}, \emph{epimorphism}, \emph{isomorphism}, 
of semigroupoids if both $\phi$ and $\Phi$ are injective, surjective, and bijective, respectively.
\end{definition}

\begin{remark}\label{r:munital} 
With notation as in the previous definition, assuming that both semigroupoids $\Gamma$ and 
$\Lambda$ have units $\epsilon_\Gamma\colon S_\Gamma\ra \Gamma$ and 
$\epsilon_\Lambda\colon S_\Lambda\ra \Lambda$, what would be a correct definition for a unital morphism 
$(\phi;\Phi)$ from $\Gamma$ to $\Lambda$? Let us first observe that, for all $s\in S_\Gamma$ we have 
$\Phi(\epsilon_\Gamma(s))=\Phi(\epsilon_\Gamma(s))^2$, that is, 
$\Phi(\epsilon_\Gamma(s))$ is an idempotent in the semigroupoid $\Lambda$. In the special case when both 
$\Gamma$ and $\Lambda$ are $*$-semigroupoids, then for all 
for all $s\in S_\Gamma$ we have 
$\Phi(\epsilon_\Gamma(s))=\Phi(\epsilon_\Gamma(s))^*=\Phi(\epsilon_\Gamma(s))^2$, that is, 
$\Phi(\epsilon_\Gamma(s))$ is a selfadjoint idempotent in the $*$-semigroupoid $\Lambda$.
But, because the
aggregation map $\phi$ might not be injective, the tentative definition of unitality
\begin{equation*}
\Phi(\epsilon_\Gamma(s))=\epsilon_\Lambda(\phi(s)),\quad s\in S_\Gamma,
\end{equation*}
might not be the correct one, for example, if $\phi$ is not injective. A possible answer to this question can be 
formulated in case the semigroupoid is modelled by operators, see Remark~\ref{r:cohu}.
\end{remark}

A semigroupoid $\Gamma$ with the property that, for any $s,t\in S_\Gamma$, the set $\Gamma_s^t$ 
has at most one element is called \emph{principal}. Equivalently, this means that  the map 
$(d_\Gamma,c_\Gamma)\colon \Gamma\ra S_\Gamma\times S_\Gamma$ is injective.

\begin{remark} A semigroupoid $\Gamma$ is isomorphic to a 
semigroupoid $\Gamma_R$ defined by a relation $R$, see Example~\ref{ex:releg}, if and only 
$\Gamma$ is a principal semigroupoid.

Indeed, if $R$ is a relation on a set $S\neq\emptyset$ and $(\phi;\Phi)$ is an isomorphism of
$\Gamma$ onto $\Gamma_R$, then, for any $s,t\in S_\Gamma$, $\Phi$ maps bijectively
the set $\Gamma_s^t$ to the set ${\Gamma_R}_{\phi(s)}^{\phi(t)}$, which has at most one element.

Conversely, assume that for any $s,t\in S_\Gamma$ the set $\Gamma_s^t$ has at most one element.
Let $S:=S_\Gamma$ and define the relation $R$ on $S$ in the following way: $(s,t)\in R$ if
$\Gamma_t^s\neq\emptyset$. Then $R$ is transitive, hence $R$
defines a semigroupoid. Let 
$\phi\colon S_\Gamma\ra S$ be the identity map and $\Phi\colon \Gamma\ra\Gamma_R$ be 
defined in the following way: for any $\gamma\in\Gamma$ we define 
$\Phi(\gamma):=(c_\Gamma(\gamma),d_\Gamma(\gamma))$. 
Then $(\phi;\Phi)$ is an isomorphism of $\Gamma$ onto $\Gamma_R$.
\end{remark}

\begin{definition} \label{d:serep}
Let $\Gamma$ be a semigroupoid. A \emph{representation} of $\Gamma$ on a
bundle of Hilbert spaces $\bH=\{\cH_s\}_{s\in S}$ is a semigroupoid morphism 
$(\phi;\Phi)$ of $\Gamma$ to the semigroupoid $\Gamma_\bH$, see Example~\ref{ex:segah}.
More precisely:
\begin{itemize}
\item[(R1)] $\phi\colon S_\Gamma\ra S$ and $\Phi\colon \Gamma\ra\Gamma_\bH$ 
are maps.
\item[(R2)] For any $\alpha\in\Gamma$ we have 
$\Phi(\alpha)\in\cB(\cH_{\phi(d(\alpha))},\cH_{\phi(c(\alpha))})$. 
\item[(R3)] For any 
$(\alpha,\beta)\in\Gamma^{(2)}$ we have $\Phi(\alpha\beta)=\Phi(\alpha) \Phi(\beta)$.
\end{itemize}

The representation $(\phi;\Phi)$ is called \emph{orthogonal}
if, in addition, the following property holds:
\begin{itemize}
\item[(R4)] For any $\alpha,\beta\in \Gamma$ such that
$\phi(c(\alpha))=\phi(c(\beta))$ but $c(\alpha)\neq c(\beta)$,
we have that $\Phi(\alpha)\cH_{\phi(d(\alpha))}\perp 
\Phi(\beta)\cH_{\phi(d(\beta))}$.
\end{itemize}

If $\Gamma$ is a $*$-semigroupoid then the morphism $(\phi;\Phi)$ is called a 
\emph{$*$-representation} of $\Gamma$ on a bundle of Hilbert spaces
$\bH$ if it is a $*$-morphism of $*$-semigroupoids from $\Gamma$ to $\Gamma_\bH$.
\end{definition}

\begin{example}\label{ex:lereg} (\emph{Aggregated Left Regular Representations}) 
Let $\Gamma$ be a semigroupoid and let $\tau\colon S_\Gamma\ra X$ be 
an aggregation map. For each $x\in X$ we consider the Hilbert space 
$\ell_\tau^2(\Gamma;x)$ with orthonormal basis $\{\delta_\gamma\mid \gamma\in \Gamma,\ 
\tau(c(\gamma))=x\}$, if the set  $\{\gamma\in\Gamma\mid
\tau(c(\gamma))=x\}$ is not empty, and let $\ell_\tau^2(\Gamma;x)$ 
be the null space, in the opposite case. Thus, in case the set $\{\gamma\in\Gamma\mid
\tau(c(\gamma))=x\}$ is not empty, the vectors in $\ell_\tau^2(\Gamma;x)$ are all complex functions 
$f\colon \{\gamma\in\Gamma\mid \tau(c(\gamma))=x\}\ra \CC$ such that
\begin{equation*}
\sum_{\gamma\in \Gamma^s,\ \tau(s)=x} |f(\gamma)|^2<\infty,
\end{equation*}
where the convergence of the sum is in the sense of summability.
We consider the bundle of Hilbert spaces 
$\ell^2_\tau(\Gamma)=\{\ell^2_\tau(\Gamma;x)\}_{x\in X}$ and 
let $\Gamma_{\ell^2_\tau(\Gamma)}$ be the semigroupoid defined as in 
Example~\ref{ex:segah}, see \eqref{e:gamabeh}.

The left regular representation induced by the aggregation map $\tau$ is the pair $(\tau;L)$ where $L\colon \Gamma
\ra \Gamma_{\ell^2_\tau(\Gamma)}$ is defined in the following way.
For each $\gamma\in\Gamma$ the operator 
$L(\gamma)$, with its domain in $\ell^2_\tau(\Gamma;\tau(d(\gamma)))$ and its range in $ \ell^2_\tau(\Gamma;\tau(c(\gamma)))$, is defined by
\begin{equation}
L(\gamma)\delta_\beta:=\begin{cases} \delta_{\gamma\beta},& c(\beta)=d(\gamma),\\ 0,& \mbox{ otherwise },
\end{cases}
\end{equation}
for all $\beta\in \Gamma$ such that $\tau(c(\beta))=x$, and then extended by linearity to
the subspace of all functions $f\in \ell^2_\tau(\Gamma;x)$ with finite support. 

(a) We assume, in addition, that $\Gamma$ satisfies the following condition.
\begin{equation}\label{e:lin}
\mbox{For each }\gamma\in\Gamma\mbox{ and  each }\beta,\beta^\prime\in\Gamma^{(d(\gamma)},\ 
\beta\neq \beta^\prime\mbox{ we have }
\gamma\beta\neq\gamma\beta^\prime.
\end{equation}
This condition holds, for example, if $\Gamma$ is a free semigroupoid, see Definition~\ref{d:frees}, or if 
it is a groupoid. It is easy to see that, under the assumption \eqref{e:lin}, 
$L(\gamma)$ is a partial isometry on its domain and hence, since its domain is dense in
$\ell^2_\tau(\Gamma;x)$, it is uniquely extended to a partial isometry on the whole space $\ell^2_\tau(\Gamma;x)$.
Then, it follows that the pair $(\tau;L)$ is an orthogonal
representation of $\Gamma$ on the bundle of Hilbert spaces $\ell^2_\tau(\Gamma)$, in the 
sense of Definition~\ref{d:serep}. 
If $\Gamma$ has a unit $\epsilon=\{\epsilon_s\}_{s\in S_\Gamma}$, then $L(\epsilon_s)$ is an
orthogonal projection for each $s\in S_\Gamma$.

In case of full aggregation, that is, $\tau$ has the range a singleton, the construction described before, for the special
case of a free semigroupoid associated to an undirected graph, can be seen in \cite{KribsPower}.

(b) If condition \eqref{e:lin} is not satisfied, the left regular representation may concern unbounded linear operators.
More precisely, for each $x\in X$, let $\cD_x$ be the linear subspace of $\ell^2_\tau(\Gamma;x)$ 
spanned by all functions $f\in \ell^2_\tau(\Gamma;x)$ with finite support, and consider the bundle
$\bD=\{\cD_x\}_{x\in X}$. Recalling the notation as in Example~\ref{ex:semes}, the pair $(\tau;L)$ is a 
semigroupoid morphism of $\Gamma$ to the semigroupoid $\Gamma_\bD$. 
For $\gamma\in\Gamma$, denote by $m_\gamma:\Gamma_{d(\gamma)} \to \Gamma^{c(\gamma)}$ the map given 
by $m_\gamma(\alpha)=\gamma\alpha$, $\alpha\in\Gamma_{d(\gamma)}$. We make the following remarks.

(c) For $\beta\in\Gamma^{c(\gamma)}$ one has the set  $m_\gamma^{-1}(\{\beta\})$ infinite if and only if 
$\delta_\beta \perp$ Dom$(L(\gamma)^*)$. Consequently, $L(\gamma)$ is closable if and only if 
$m_\gamma^{-1}(\{\beta\})$ is finite for all $\beta\in\Gamma^{c(\gamma)}$. 

Indeed, suppose that $m_\gamma^{-1}(\{\beta\})$ is infinite and consider an arbitrary finite subset 
$\mathcal F\subset m_\gamma^{-1}(\{\beta\})$. Then
\begin{equation*}\langle L(\gamma)(\sum_{\alpha\in\mathcal F} \lambda_\alpha\delta_\alpha),  \sum_{\beta'}
\mu_{\beta'}\delta_{\beta'}\rangle = \langle (\sum_{\alpha\in\mathcal F} \lambda_\alpha)\delta_\beta,
\mu_\beta\delta_\beta+\sum_{\beta'\neq\beta}\mu_{\beta'} \delta_{\beta'}\rangle = \bar \mu_\beta 
\sum_{\alpha\in\mathcal F}\lambda_\alpha.\end{equation*} 
For each $n\in\NN$ choose $\mathcal F$ having cardinality $n$ and 
$\lambda_\alpha=1$ for all $\alpha\in\mathcal F$. Then $\|\sum_{\alpha\in\mathcal F}
\lambda_\alpha\delta_\alpha\|_2=\sqrt n$ while $\sum_{\alpha\in\mathcal F}\lambda_\alpha=n$, so there can be 
no constant $C$ such that 
\begin{equation*}\|L(\gamma)(\sum_{\alpha\in\mathcal F}\lambda_\alpha\delta_\alpha)\|_2 \leq C\|
\sum_{\alpha\in\mathcal F}\lambda_\alpha\delta_\alpha\|_2\end{equation*}
for all such subsets $\mathcal F$ and finite linear combinations unless $\mu_\beta=0$. This shows that the functional 
\begin{equation*}{\rm Dom}  (L(\gamma)) \ni \sum_\alpha \lambda_\alpha\delta_\alpha \mapsto \langle L(\gamma)
(\sum_{\alpha} \lambda_\alpha\delta_\alpha),  \sum_{\beta'}\mu_{\beta'}\delta_{\beta'}\rangle \end{equation*}
cannot be bounded unless $\mu_\beta=0$, i.e.\ unless the vector $\sum_{\beta'} \mu_{\beta'}\delta_{\beta'}$ is 
orthogonal to $\delta_\beta$. Conversely, if the cardinality of $m_\gamma^{-1}(\{\beta\})$ is $N\in\NN$, it is
rather easy to see that the functional 
\begin{equation*}{\rm Dom}  (L(\gamma)) \ni \sum_\alpha \lambda_\alpha\delta_\alpha \mapsto \langle L(\gamma)
(\sum_{\alpha} \lambda_\alpha\delta_\alpha), \delta_\beta \rangle\end{equation*}
is bounded with norm $\leq N$, hence $\delta_\beta\in $ Dom$(L(\gamma)^*)$.  

(d) $L(\gamma)$ is bounded with $\|L(\gamma)\|\leq N$ if and only if there exists $N\in\NN$ such that the cardinality 
of $m_\gamma^{-1}(\{\beta\})$ is $\leq N$ for all $\beta\in\Gamma^{c(\gamma)}$.

(e) If $\gamma$ has an inverse $\gamma^{-1}$, then $L(\gamma)$ is a partial isometry with $L(\gamma)^*=L(\gamma^{-1})$.
\end{example}

\begin{example} (\emph{Representations of Free Graph $*$-Semigroupoids}) \label{ex:rfg}
Let $G=(V;E;s;r)$ be a directed graph 
and let $\FF_*^+(G)$ be the associated free graph $*$-semigroupoid with unit defined as 
in Example~\ref{ex:fsug}. Let 
$\bH=\{\cH_x\}_{x\in X}$ be a bundle of Hilbert spaces and 
consider the $*$-semigroupoid with unit $\Gamma_\bH$ defined 
as in Example~\ref{ex:segah}. Let $\tau\colon V\ra X$ be an aggregation map. For each edge $f\in E$ 
that is not a loop, let $\Phi(f):=T_f\in\cB(\cH_{\tau(s(f))},\cH_{\tau(r(f))})$ and 
$\Phi(f^*):=T_f^*\in\cB(\cH_{\tau(r(f))},\cH_{\tau(s(f))})$.
For each vertex $v\in V$, let $\Phi(e_v):=P_v\in \cB(\cH_{\tau(v)})$ be an orthogonal projection such that for any
$f\in \tilde E=E\cup E^*$ with $r(f)=v$ we have $\ran(T_f)\subseteq \ran(P_v)$. We then extend 
$\Phi:\FF_*^+(G)\ra \Gamma_{\bH}$ in the natural fashion: for any finite path
$\gamma=\gamma_1\cdots\gamma_n\in \Gamma_G$, we have $\Phi(\gamma):=
T_{\gamma_1}\cdots T_{\gamma_n}\in \cB(\cH_{\tau(s(\gamma_n))},\cH_{\tau(r(\gamma_1))})$. It is easy to see that
$(\tau;\Phi)$ is a representation of the unital $*$-semigroupoid $\FF_*^+(G)$ on the unital $*$-semigroupoid 
$\Gamma_\bH$.
\end{example}

\begin{example}\label{ex:ckt} 
(\emph{Cuntz-Krieger-Toeplitz $G$-Families as Representations of  Free Graph 
$*$-Semigroupoids}) 
Let $G=(V;E;s;r)$ be a directed graph such that
both the vertex set $V$ and the edge set $E$ are countable. A \emph{Cuntz-Krieger-Toeplitz} (CKT) 
\emph{$G$-family}, e.g.\ see \cite{KumjianPaskRaeburnRenault}, consists
in a pair $(\bP;\bS)$, where $\bP=\{P_v\mid v\in V\}$ is a bundle of mutually orthogonal projections on a Hilbert 
space $\cH$ and  $\bS=\{S_f\mid f\in E\}$ is a bundle of partial isometries on $\cH$ subject to the following conditions.

\begin{itemize}
\item[(I)] $S_f^*S_f=P_{s(f)}$ for all $f\in E$.
\item[(CKT)]  $\sum_{f\in F}S_fS_f^*\leq P_v$ for any $v\in V$ and any finite set $F\subseteq r^{-1}(v)$.
\end{itemize}

From (I) it follows that the right support $\ran(S_f^*)=\ran(P_{s(f)})$ and then from property
(CKT) it follows that $\ran(S_f)\subseteq \ran(P_{r(f)})$ for all $f\in E$ 
and that whenever $f,g\in E$ are such that $f\neq g$ and $r(f)=r(g)$
then $\ran(S_f)\perp \ran(S_g)$. Let $\tau$ be the full aggregation map on $V$, that is, $\tau(V)$ is a singleton. 
Then we observe that the CKT $G$-family $(\bP;\bS)$ 
satisfies the conditions in Example~\ref{ex:rfg} and hence it induces a
representation of the $*$-semigroupoid $\FF^+_*(G)$ on the Hilbert space $\cH$ (the bundle of Hilbert spaces with 
only one element). More precisely, for each $f\in E$ we have $\Phi(f):=S_f\in\cB(\cH)$ and $\Phi(f^*):=S_f^*$ 
and for each $v\in V$ we have $\Phi(e_v)=P_v$, and then extending $\Phi$ on $\FF^+_*(G)$ in the natural way.

In general, the $*$-representation $\Phi$ defined before is not orthogonal. However, if we consider only its restriction
$\Phi|_{\FF^+(G)}$ to the free semigroupoid $\FF^+(G)$, it is an orthogonal representation. 
This is because, if $\alpha,\beta\in
\FF^+(G)$ are such that $c(\alpha)\neq c(\beta)$ then $S_\alpha$ and $S_\beta$ have ranges in 
orthogonal subspaces of $\cH$, as a consequence of the condition (CKT), as explained before.

The CKT $G$-family $(\bP;\bS)$
is called \emph{nondegenerate} if the span of the ranges of the orthogonal projections $P_v$, $v\in V$, 
is dense in $\cH$. This implies that
\begin{equation}\label{e:suv}
\sum_{v\in V} P_v =I_\cH,
\end{equation}
where the sum converges in the strong operator topology (recall that we assumed that $V$ is countable, 
so the sum is actually a series). This fact is important since it gives us an idea of what a unital representation of  a
semigroupoid might be.

If, in addition,
\begin{itemize}
\item[(CK)] $\sum_{f\in r^{-1}(v)} S_fS_f^*=P_v$ for any $v\in V$ such that $0<|r^{-1}(v)|<\infty$,
\end{itemize}
the pair $(\bP;\bS)$ is called a \emph{Cuntz-Krieger} (CK) \emph{$G$-family}. For most of the investigations on CK 
$G$-families, $G$ is assumed to be a
\emph{row-finite graph}, that is, for any $v\in V$ we have $|r^{-1}(v)|<\infty$, and \emph{sourceless}, that is, for any
$v\in V$ we have $r^{-1}(v)\neq\emptyset$.
\end{example}

In this article we will use a more general concept of representation of a $*$-semigroupoid on 
$*$-semigroupoids of type $\Gamma_{\bH,\bD}$, see Example~\ref{ex:semeseh}. For the moment, this concept is 
motivated by the fact, see Example~\ref{ex:lereg}, that, in general, the left regular representation is made by 
unbounded operators. It will show its significance during the next section.

\begin{definition}\label{d:urep} Let $\Gamma$ be a semigroupoid and consider two bundles
of vector spaces $\bD=\{\cD_s\}_{s\in S}$ and $\bH=\{\cH_s\}_{s\in S}$, where $\cD_s$ is a dense
subspace in the Hilbert space $\cH_s$, for all $s\in S$. With notation as in 
Example~\ref{ex:semuh}, a \emph{generalised representation} of $\Gamma$ on the pair of
bundles $(\bD;\bH)$ is a pair of maps $(\phi;\Phi)$ subject to the following conditions.
\begin{itemize}
\item[(UR1)] $\phi\colon S_\Gamma\ra S$ and $\Phi\colon \Gamma\ra\Gamma_{\bH,\bD}$ are maps.
\item[(UR2)] For any $\alpha\in \Gamma$ we have $\Phi(\alpha)$ a linear operator such that
$\cD_{\phi(d(\alpha))}\subseteq \dom(\Phi(\alpha))$ and 
$\Phi(\alpha)\cD_{\phi(d(\alpha))}\subseteq\cD_{\phi(c(\alpha))}$.
\item[(UR3)] For any $(\alpha,\beta)\in \Gamma^{(2)}$ we have 
$\Phi(\alpha\beta)|_{\cD_{\phi(d(\beta))}}=\Phi(\alpha)\Phi(\beta)|_{\cD_{\phi(d(\beta))}}$.
\end{itemize}

We call the generalised representation $(\phi;\Phi)$ \emph{orthogonal} if the following property holds
\begin{itemize}
\item[(UR4)] For any $\alpha,\beta\in \Gamma$ such that
$\phi(c(\alpha))=\phi(c(\beta))$, but $c(\alpha)\neq c(\beta)$,
 it follows that $\Phi(\alpha)\cD_{\phi(d(\alpha))}\perp 
\Phi(\beta)\cD_{\phi(d(\beta))}$.
\end{itemize}

Also, if, in addition, $\Gamma$ is a $*$-semigroupoid,
then we call $(\phi;\Phi)$ a generalised $*$-representation of $\Gamma$ on the bundles $(\bD,\bH)$ if
the following conditions hold.
\begin{itemize}
\item[(UR6)] For any $\alpha\in \Gamma$ we have
$\Phi(\alpha)^*\cD_{\phi(c(\alpha))}\subseteq\cD_{\phi(d(\alpha))}$. 
\item[(UR7)] For any $\alpha\in \Gamma$ we have $\Phi(\alpha^*)|_{\cD_{\phi(c(\alpha))}}
=\Phi(\alpha)^*|_{\cD_{(\phi(c(\alpha))}}$.
\end{itemize}
\end{definition}

Let us firstly observe that in case the representation $(\phi;\Phi)$ is aggregation free, that is,
$\phi$ is injective, it is automatically orthogonal. In general, the fashion in which different pieces of the representation 
are aggregated within a bundle of Hilbert spaces makes technical obstructions. 
This is clearly seen in the interpretation of
the Toeplitz--Cuntz--Krieger systems as in Example~\ref{ex:ckt}, as a fully aggregated representation of the 
$*$-semigroupoid $\FF^+_*(G)$.
The orthogonality condition on generalised $*$-representations is essential in this article and will 
show its importance in the next section. Now, we record a first consequence of this condition.

\begin{lemma}\label{l:orth} If $(\phi;\Phi)$ is an orthogonal 
generalised $*$-representation of the $*$-semigroupoid $\Gamma$ on the pair of bundles $(\bD;\bH)$, with 
$\bD=\{\cD_s\}_{s\in S}$ and $\bH=\{\cH_s\}_{s\in S}$, 
then, for any 
$\alpha,\beta\in\Gamma$ such that $\phi(c(\alpha))=\phi(d(\beta))=s$ but $c(\alpha)\neq d(\beta)$, 
we have $\Phi(\beta)\Phi(\alpha)|_{\cD_{\phi(d(\alpha))}}=0$.
\end{lemma}

\begin{proof} Let us first observe that, since $\Phi(\alpha)\cD_{\phi(d(\alpha))}\subseteq \cD_s$ we have
$\dom(\Phi(\beta)\Phi(\alpha))\supseteq \cD_{\phi(d(\alpha))}$. We consider $\beta^*\in \Gamma$ and observe
that $\phi(c(\alpha))=s=\phi(c(\beta^*))$ but $c(\alpha)\neq d(\beta)=c(\beta^*)$ hence, by the orthogonality condition,
on the one hand we have
\begin{equation*}\Phi(\alpha)\cD_{\phi(d(\alpha))}\perp \Phi(\beta^*)\cD_{\phi(d(\beta^*))},\end{equation*}
and, on the other hand, by (UR7) we have
\begin{equation*}
\Phi(\beta^*)\cD_{\phi(d(\beta^*))}=\Phi(\beta)^*\cD_{\phi(c(\beta))},
\end{equation*}
hence, for any $h\in \cD_{\phi(d(\alpha))}$ and any $k\in \cD_{\phi(c(\beta))}$ we have 
\begin{equation*}
\langle \Phi(\beta)\Phi(\alpha)h,k\rangle_{\cH_{\phi(c(\beta))}}=
\langle \Phi(\alpha)h,\Phi(\beta)^*k\rangle_{\cH_{\phi(c(\alpha))}}=0.
\end{equation*}
Since $\cD_{\phi(c(\beta))}$ is dense in $\cH_{\phi(c(\beta))}$, from here we get 
$\Phi(\beta)\Phi(\alpha)|_{\cD_{\phi(d(\alpha))}}=0$.
\end{proof}

\section{Generalised Dilations of Positive Semidefinite Maps on $*$-Semigroupoids}\label{s:ud}

In this section we will get dilations of operator valued  partially positive semidefinite maps on $*$-semigroupoids by 
unbounded operators, in the spirit of Definition~\ref{d:urep}. The fundamental concept in this enterprise is that of
partially positive semidefiniteness.

Let $(\Gamma;S;d;c;\cdot;*;\epsilon)$ be a $*$-semigroupoid with unit,
consider a bundle of Hilbert spaces 
$\bH=\{\cH_x\}_{x\in X}$ over the field $\KK$, where $\KK$ can be either $\RR$ or $\CC$,
for some $X\neq\emptyset$, and let $\tau\colon S\ra X$ be an aggregation map.

\begin{definition}\label{d:hm}  The class $\fH_{\bH,\tau}(\Gamma)$ of
\emph{Hermitian $\bH$-valued maps on $\Gamma$ and $\tau$-coherent},  consists in
all maps $T\colon \Gamma\ra \Gamma_\bH$
subject to the following conditions.
\begin{itemize}
\item[(HM1)] $T$ is \emph{$\tau$-coherent} in the sense that,
$T(\alpha)\in \cB(\cH_{\tau(d(\alpha))},\cH_{\tau(c(\alpha))})$, for all $\alpha\in\Gamma$. 
\item[(HM2)] $T$ is \emph{Hermitian}, in the sense that $T(\alpha^*)=T(\alpha)^*$ for all $\alpha\in\Gamma$.
\end{itemize} 
\end{definition}

\begin{definition}\label{d:psm} A map $T\in\fH_{\bH,\tau}(\Gamma)$ is called \emph{partially 
$n$-positive semidefinite}, for some $n\in\NN$,
if for any $s\in S $, any $\alpha_1,\ldots,\alpha_n\in \Gamma^s$, and any $h_1,\ldots,h_n$ such 
that $h_j\in \cH_{\tau(d(\alpha_j))}$ for all $j=1,\ldots,n$, we have
\begin{equation}\label{e:psd1}
\sum_{i,j=1}^n \langle T(\alpha_i^*\alpha_j)h_j,h_i\rangle_{\cH_{\tau(d(\alpha_i))}}\geq 0.
\end{equation}

The map $T$ is called \emph{partially positive semidefinite} if it is partially $n$-positive semidefinite 
for all $n\in\NN$, equivalently, if, for any $s\in S $ and
for any finitely supported cross-section $h=\{h_\alpha\}_{\alpha\in\Gamma^s}$, where 
$h_\alpha\in \cH_{\tau(d(\alpha))}$ for all $\alpha\in\Gamma^s$, we have
  \begin{equation}\label{e:psd2}
  \sum_{\alpha,\beta\in\Gamma^s} \langle T(\beta^*\alpha)h_\alpha,h_\beta\rangle_{\cH_{\tau(d(\beta))}}
  \geq 0.
  \end{equation}
We denote by $\fP_{\bH,\tau}(\Gamma)$ the class of all maps $T\in\fH_{\bH,\tau}(\Gamma)$ that are partially
positive semidefinite.
\end{definition}

\begin{remark}\label{r:part} A partially positive semidefinite map $T$ is actually a bundle of positive semidefinite 
maps $\{T_s\}_{s\in S}$, where $T_s=T|_{\Gamma^s}$, for all $s\in S$. Because of this,
in the following we will drop \emph{partially} whenever 
it will be clear from the context and hence, a partially positive 
semidefinite map will be simply called positive semidefinite.
\end{remark}

\begin{remark} 
For a given $n\in\NN$, the map $T$ is $n$-positive semidefinite if and only if
for any 
$s\in S $ and any $\alpha_1,\ldots,\alpha_n\in \Gamma^s$ the matrix operator 
$[T(\alpha_i^*\alpha_j)]_{i,j=1}^n$ is positive semidefinite as an operator in the Hilbert space
$\cH_{\tau(d(\alpha_1))}\oplus \cdots \oplus \cH_{\tau(d(\alpha_n))}$.

In particular, the Hermitian condition (HM2) in Definition~\ref{d:hm} is needed only when $\KK=\RR$. This is
because,
if $\KK=\CC$ then any $2$-positive semidefinite pair $(\tau;T)$ automatically satisfies the Hermitian 
condition (HM2). Indeed, for any $\alpha\in \Gamma$ letting $s=c(\alpha)$, $\alpha_1=\epsilon_s$, and
$\alpha_2=\alpha$, the $2\times 2$ block operator matrix
\begin{equation*}
\begin{bmatrix}
T(\epsilon_s) & T(\alpha) \\ T(\alpha^*) & T(\alpha^*\alpha)
\end{bmatrix}\geq 0,
\end{equation*} 
when viewed as an operator on $\cH_{\tau(s)}\oplus\cH_{\tau(d(\alpha))}$. This implies that this
block operator matrix is Hermitian, hence $T(\alpha^*)=T(\alpha)^*$.

If $\KK=\RR$ then the statement from before is not true, so we have to assume the 
Hermitian condition (HM2) additionally.
\end{remark}

Positive semidefiniteness of the map $T$ is a necessary condition for it 
to admit dilations in a rather general sense, that we make precise in the following.

\begin{definition}\label{d:udilation} Given a  map $T\in\fH_{\bH,\tau}(\Gamma)$,
a quadruple $(\bK;\bD;\Phi;\bV)$ is called a \emph{generalised dilation} of $T$ if it satisfies the following conditions.
\begin{itemize}
\item[(UD1)] $\bK=\{\cK_x\}_{x\in X}$ is a bundle of Hilbert spaces.
\item[(UD2)] $\bD=\{\cD_x\}_{x\in X}$ is a bundle of vector spaces such that, for each $x\in X$,
$\cD_x$ is a dense subspace of $\cK_x$.
\item[(UD3)] $\bV=\{V_s\}_{s\in S }$ is a bundle of operators, where
$V_s\in\cB(\cH_{\tau(s)},\cK_{\tau(s)})$ is such that $\ran(V_s)\subseteq\cD_{\tau(s)}$ for all $s\in S$.
\item[(UD4)] The pair $(\tau;\Phi)$ is an unbounded $*$-representation of $\Gamma$ on the pair 
$(\bD;\bK)$, in the sense of Definition~\ref{d:urep}, such that
\begin{equation}\label{e:talat}
T(\alpha) = V_{c(\alpha)}^* \Phi(\alpha) V_{d(\alpha)},\quad \alpha\in\Gamma.
\end{equation}
\end{itemize}
\end{definition}

\begin{remark}\label{r:tala} An equivalent formulation of \eqref{e:talat} is that, for all 
$\alpha\in\Gamma$, $h\in\cH_{\tau(d(\alpha))}$, and $k\in\cH_{\tau(c(\alpha))}$, we have
\begin{equation}\label{e:tala}
\langle T(\alpha)h,k\rangle_{\cH_{\tau(c(\alpha))}} =
\langle  \Phi(\alpha) V_{d(\alpha)}h,V_{c(\alpha)}k\rangle_{\cK_{\tau(c(\alpha))}}.
\end{equation}
This is because, by axiom (UD3), we have $\ran(V_{d(\alpha)})\subseteq \cD_{\tau(d(\alpha))}$ and,
see Example~\ref{ex:semuh}, $\Phi(\alpha)\in\cL^*(\cD_{\tau(d(\alpha))},\cD_{\tau(c(\alpha))})$,
hence we have $\cD_{\tau(d(\alpha))}\subseteq\dom(\Phi(\alpha))$. In this fashion,
although, in general, the operator $\Phi(\alpha)$ is not everywhere defined on
$\cK_{\tau(d(\alpha))}$, the operator on the right side in \eqref{e:talat} is everywhere defined.
\end{remark}

Positive semidefiniteness of the map $T$ is a necessary condition for it to admit generalised dilations. 

\begin{proposition}\label{p:psmu} If $T\in\fH_{\bH,\tau}(\Gamma)$ has a generalised dilation
$(\bK;\bD;\Phi;\bV)$ then it is positive semidefinite.
\end{proposition}

\begin{proof} Let $(\bK;\bD;\Phi;\bV)$ be a dilation of $T$.
Then, for any $s\in S $ and
  for any finitely supported cross-section $h=\{h_\alpha\}_{\alpha\in\Gamma^s}$, where 
  $h_\alpha\in \cH_{\tau(d(\alpha))}$ for all $\alpha\in\Gamma^s$, we have
  \begin{align*}
  \sum_{\alpha,\beta\in\Gamma^s} \langle 
 T(\beta^*\alpha)h_\alpha,h_\beta\rangle_{\cH_{\tau(d(\beta))}}
  & = \sum_{\alpha,\beta\in\Gamma^s} \langle V_{d(\beta)}^* 
 \Phi(\beta^*\alpha)V_{d(\alpha)}
  h_\alpha,h_\beta\rangle_{\cH_{\tau(d(\beta))}} \\
  & = \sum_{\alpha,\beta\in\Gamma^s} \langle  \Phi(\beta^*)\Phi(\alpha)V_{d(\alpha)}
  h_\alpha,V_{d(\beta)}h_\beta\rangle_{\cK_{\tau(d(\beta))}} \\
  & = \sum_{\alpha,\beta\in\Gamma^s} \langle  \Phi(\alpha)V_{d(\alpha)}
  h_\alpha,\Phi(\beta)V_{d(\beta)}h_\beta\rangle_{\cK_{\tau(s)}} \\
  & = \biggl\langle \sum_{\alpha\in\Gamma^s} \Phi(\alpha)V_{d(\alpha)}h_\alpha,
  \sum_{\alpha\in\Gamma^s} \Phi(\alpha)V_{d(\alpha)}h_\alpha
  \biggr\rangle_{\cK_{\tau(s)}}
  \geq 0,
  \end{align*}
  where we have used the axioms (UR3) and (UR4), see Definition~\ref{d:urep}, as well as the
  axiom (UD3) in Definition~\ref{d:udilation}.
\end{proof}

\begin{definition}\label{d:minimalu}
A generalised dilation $(\bK;\bD;\Phi;\bV)$ of the map $T\in\fP_{\bH,\tau}(\Gamma)$ is called \emph{minimal} if
\begin{equation*}
\cD_{x}= \lin \{\Phi(\alpha)V_{d(\alpha)}\cH_{\tau(d(\alpha))}\mid \alpha\in\Gamma\mbox{ such 
that }\tau(c(\alpha))=x\},\quad x\in X.
\end{equation*}
\end{definition}

\begin{remark}\label{r:epsu}
In order to put the previous definition in perspective, let $T\in\fH_{\bH,\tau}(\Gamma)$ and let
$(\bK;\bD;\Phi;\bV)$ be a generalised dilation
of the pair $(\tau;T)$. It is a simple exercise to see that, for any $s\in S $, the 
operator $\Phi(\epsilon_s)\in\cL^*(\cD_{\tau(s)})$ has a unique extension to 
an orthogonal projection in $\cK_{\tau(s)}$, yet denoted by $\Phi(\epsilon_s)\in\cB(\cK_{\tau(s)})$.
Indeed, on the one hand, 
$\Phi(\epsilon_s)|_{\cD_{\tau(s)}}=\Phi(\epsilon_s\epsilon_s)|_{\cD_{\tau(s)}}
=\Phi(\epsilon_s)\Phi(\epsilon_s)|_{\cD_{\tau(s)}}$ and 
$\Phi(\epsilon_s)|_{\cD_{\tau(s)}}=\Phi(\epsilon_s^*)|_{\cD_{\tau(s)}}
=\Phi(\epsilon_s)^*|_{\cD_{\tau(s)}}$, and then it is easy to see that 
$\Phi(\epsilon_s)$ is the restriction to $\cD_{\tau(s)}$ of an orthogonal projection, 
yet denoted by $\Phi(\epsilon_s)\in\cB(\cK_{\tau(s)})$.

In addition,
\begin{equation}\label{e:clol}
\clos\lin\{\Phi(\alpha)V_{d(\alpha)}\cH_{\tau(d(\alpha))}\mid \alpha\in\Gamma^s\}
\subseteq\ran(\Phi(\epsilon_s)).\end{equation}
 Indeed, for any $\alpha\in\Gamma^s$ we have
\begin{equation*}
\Phi(\epsilon_s)\Phi(\alpha)V_{\tau(d(\alpha))}h=\Phi(\epsilon_s\alpha)V_{\tau(d(\alpha))}h =
\Phi(\alpha)V_{\tau(d(\alpha))}h,\quad h\in \cH_{\tau(d(\alpha))},
\end{equation*}
hence $\Phi(\epsilon_s)$ acts like the identity operator on the space 
$\lin\{\Phi(\alpha)V_{d(\alpha)}\cH_{\tau(d(\alpha))}\mid \alpha\in\Gamma^s\}$ and, consequently,
on its closure. This proves \eqref{e:clol}.
\end{remark}

At this level of generality there is one more special condition 
on the dilation which we have to consider, and that will play a very important role.

\begin{definition}\label{d:orthogonalu} A generalised dilation $(\bK;\bD;\Phi;\bV)$ of the map
$T\in\fP_{\bH,\tau}(\Gamma)$ is called 
\emph{orthogonal} if the representation $(\tau;\Phi)$ is orthogonal, see Definition~\ref{d:urep}, that is,
 whenever $\alpha,\beta\in \Gamma$ are such that $c(\alpha)\neq c(\beta)$ but
$\tau(c(\alpha))=\tau(c(\beta))$, we have 
$\Phi(\alpha)\cD_{\tau(d(\alpha))}\perp \Phi(\beta)\cD_{\tau(d(\beta))}$.
\end{definition}

The following proposition provides a set of assumptions on a generalised dilation that contains a minimal one. 

\begin{proposition}\label{p:mingendil}
Let $(\bK;\bD;\Phi;\bV)$, with $\bK=\{\cK_x\}_{x\in X}$ and $\bD=\{\cD_x\}_{x\in X}$, be an
orthogonal generalised dilation of the map $T\in\fP_{\bH,\tau}(\Gamma)$. 
For arbitrary $x\in X$
let
\begin{equation}\label{e:pede}
\cE_x:= \lin \{\Phi(\alpha)V_{d(\alpha)}\cH_{\tau(d(\alpha))}\mid \alpha\in\Gamma\mbox{ such 
that }\tau(c(\alpha))=x\},
\end{equation}
provided that the set $\{\alpha\in \Gamma \mid \tau(c(\alpha))=x\}$ is not empty, and let $\cE_x$ be the null space, in 
the opposite case, and let $\cG_x:=\clos \cE_x$. We assume that for any $\alpha\in\Gamma$
\begin{equation}\label{e:peale}
P_{\cG_{\tau(c(\alpha))}}\ran(V_{d(\alpha)})\subseteq \cD_{\tau(c(\alpha))}\mbox{ and }
P_{\cG^\perp_{\tau(c(\alpha))}}\ran(V_{d(\alpha)})\subseteq \cD_{\tau(c(\alpha))}.
\end{equation}
Then, letting $\bG=\{\cG_x\}_{x\in X}$, $\bE=\{\cE_x\}_{x\in X}$,
$\widetilde \Phi(\alpha)=P_{\cG_{\tau(c(\alpha))}}\Phi|_{\cE_{\tau(d(\alpha))}}$ for all $\alpha\in \Gamma$, and
$\widetilde V_s=P_{\cG_{\tau(s)}}V_s$ for all $s\in S$,
the quadruple $(\bG;\bE;\widetilde\Phi;\widetilde\bV)$ is a minimal orthogonal generalised dilation  of the pair 
$(\tau;T)$.
\end{proposition}

\begin{proof} 
From Definition~\ref{d:udilation} and Definition~\ref{d:urep} it is easy to see that $\cE_x\subseteq \cD_x$, for all 
$x\in X$. Then observe that, for any $\beta\in \Gamma$ we have that 
\begin{equation}\label{e:pebe}
\Phi(\beta)\cE_{\tau(d(\beta))}\subseteq \cE_{\tau(c(\beta))}\mbox{ and }
\Phi(\beta) (\cE_{\tau(d(\beta))}^\perp \cap
\cD_{\tau(d(\beta))})\subseteq  \cE_{\tau(c(\beta))}^\perp\cap \cD_{\tau(c(\beta))}.\end{equation} 
Indeed, let $\alpha\in\Gamma$ be such that $\tau(c(\alpha))=\tau(d(\beta))$. If $d(\beta)=c(\alpha)$ we have
\begin{equation*}
\Phi(\beta)\Phi(\alpha)V_{d(\alpha)}\cH_{\tau(d(\alpha))}=\Phi(\beta\alpha)V_{d(\alpha)}\cH_{\tau(d(\alpha))}
\subseteq \cE_{\tau(c(\beta))}.
\end{equation*}
If $d(\beta)\neq c(\alpha)$ then, in view of orthogonality, we can apply Lemma~\ref{l:orth} and get
\begin{equation*}
\Phi(\beta)\Phi(\alpha)V_{d(\alpha)}\cH_{\tau(d(\alpha))}=\{0\}\subseteq \cE_{\tau(c(\beta))}.
\end{equation*}
In conclusion, the first inclusion in \eqref{e:pebe} is proven. A similar argument proves the latter inclusion 
in \eqref{e:pebe}.

These show that, on the one hand, letting 
$\widetilde \Phi(\alpha)=P_{\cG_{\tau(c(\alpha))}}\Phi|_{\cE_{\tau(d(\alpha))}}$ for all $\alpha\in \Gamma$, the pair 
$(\tau;\widetilde \Phi)$ is a $*$-representation of $\Gamma$ on the pair of bundles $(\bE;\bG)$. On the other hand, 
letting $\widetilde V_s=P_{\cG_s}V_s$ for all $s\in S$, these show that,
for any $\alpha\in\Gamma$, $h\in\cH_{\tau(d(\alpha))}$, and $k\in\cH_{\tau(c(\alpha))}$, see \eqref{e:tala}, we have
\begin{align*}
\langle T(\alpha)h,k\rangle_{\cH_{\tau(c(\alpha))}} & =
\langle  \Phi(\alpha) V_{d(\alpha)}h,V_{c(\alpha)}k\rangle_{\cK_{\tau(c(\alpha))}} \\
& = \langle P_{\cG_{\tau(c(\alpha))}} \Phi(\alpha)V_{d(\alpha)}h,V_{c(\alpha)}k\rangle_{\cK_{\tau(c(\alpha))}} \\
& = \langle P_{\cG_{\tau(c(\alpha))}} \Phi(\alpha)\bigl(P_{\cG_{\tau(d(\alpha))}} 
+ P_{\cG_{\tau(d(\alpha))}^\perp}\bigr)V_{d(\alpha)}h, V_{c(\alpha)}k\rangle_{\cK_{\tau(c(\alpha))}} \\
\intertext{which, in view of  \eqref{e:peale} and \eqref{e:pebe}, equals}
& = \langle P_{\cG_{\tau(c(\alpha))}} \Phi(\alpha) P_{\cG_{\tau(d(\alpha))}} 
V_{d(\alpha)}h, P_{\cG_{\tau(c(\alpha))}}V_{c(\alpha)}k\rangle_{\cK_{\tau(c(\alpha))}} \\
& = \langle \widetilde \Phi(\alpha)\widetilde V_{d(\alpha)}h,\widetilde V_{c(\alpha)}k\rangle_{\cK_{\tau(c(\alpha))}}.
\end{align*}
Also, for any $x\in X$, a similar reasoning as before shows that
\begin{align*} 
\cE_x & = \lin \{\Phi(\alpha)V_{d(\alpha)}\cH_{\tau(d(\alpha))}\mid \alpha\in\Gamma,\
c(\alpha)\in\tau^{-1}(\{x\}\}\\
& = \lin \{\widetilde\Phi(\alpha)\widetilde V_{d(\alpha)}\cH_{\tau(d(\alpha))}\mid \alpha\in\Gamma,\ 
c(\alpha)\in\tau^{-1}(\{x\}\},
\end{align*}
hence the generalised dilation $(\bG;\bE;\widetilde\Phi;\widetilde\bV)$ is minimal.
\end{proof}

Minimal dilations have special properties and it is preferable to work with them, instead of the general 
dilations, in view of a uniqueness property in the following sense.

\begin{definition} \label{d:ueu} Two generalised dilations 
$(\bK;\bD;\Phi;\bV)$ and $(\bK^\prime;\bD^\prime;\Phi^\prime;\bV^\prime)$
of the map $T$ are called \emph{unitarily equivalent} if
there exists a bundle of operators $\bU=\{U_x\}_{x\in X}$ subject to the following conditions.
\begin{itemize}
\item[(UE1)] For each $x\in X$, $U_x\colon \cK_x\ra\cK^\prime_x$ is unitary.
\item[(UE2)] For each $x\in X$, $U_x\cD_x=\cD^\prime_x$.
\item[(UE3)] The bundle $\bU$ intertwines the representations $\Phi$ and $\Phi^\prime$ in the sense
that
\begin{equation*}
U_{\tau(c(\alpha))} \Phi(\alpha)|_{\cD_{\tau(d(\alpha))}} = 
\Phi^\prime(\alpha)U_{\tau(d(\alpha))}|_{\cD_{\tau(d(\alpha))}},\quad \alpha\in\Gamma.
\end{equation*}
\item[(UE4)] The bundle $\bU$ maps the bundle $\bV$ to the bundle $\bV^\prime$ in the sense that 
$U_{\tau(s)}V_s=V^\prime_s$ for all $s\in S $.
\end{itemize}
\end{definition}

The main theorem of this section says that, in general, in order for a map $T\in\fH_{\bH,\tau}(\Gamma)$ to 
admit an unbounded 
dilation it is sufficient to be positive semidefinite, that is, $T\in\fP_{\bH,\tau}(\Gamma)$, 
hence the converse of Proposition~\ref{p:psmu} holds.
Also, for positive semidefinite maps $T$, one can always find an orthogonal and 
minimal generalised dilation, which is unique 
up to unitary equivalence. In view of Remark~\ref{r:part}, the proof follows by firstly 
constructing the dilations on each fibre $T_s=T|_{\Gamma^s}$, 
for each $s\in S$ and then we have to put all these together in such a way 
that uniqueness holds and, for this, the idea of orthogonality shows its importance.

\begin{theorem}\label{t:nagyu}
Let $T\in\fH_{\bH,\tau}(\Gamma)$. The following assertions are equivalent.
\begin{itemize}
 \item[(1)] $T$ is positive semidefinite, in the sense of Definition~\ref{d:psm}.   
\item[(2)] $T$ has a generalised dilation $(\bK;\bD;\Phi;\bV)$.
\end{itemize}

In addition, if assertion (1) holds then a dilation $(\bK;\bD;\Phi;\bV)$ of $T$
can always be obtained orthogonal in the sense of Definition~\ref{d:orthogonalu}, and
minimal in the sense of Definition~\ref{d:minimalu} and, in this case, it
is unique up to a unitary equivalence, in the sense of Definition~\ref{d:ueu}.
\end{theorem}

\begin{proof} $(2)\Ra(1)$. This implication is the content of Proposition~\ref{p:psmu}.\medskip

$(1)\Ra(2)$. We divide the proof in six steps.\medskip

\textbf{Step 1.} \emph{Construction of the bundles 
$\bK=\{\cK_x\}_{x\in X}$ and $\bD=\{\cD_x\}_{x\in X}$.}\medskip

For each $s\in S $,
let $\cF_s$ denote the vector space of all cross-sections
$f=\{f_\gamma\}_{\gamma\in\Gamma^s}$ subject to 
the condition that $f_\gamma\in\cH_{\tau(d(\gamma))}$, for all $\gamma\in\Gamma^s$, and let
$\cG_{s,0}$ be the subspace of $\cF_s$ of all families 
$g=\{g_\gamma\}_{\gamma\in \Gamma^s}$ of finite support. If $g,g^\prime$ are two vectors in $\cF_s$ such that at least one of them is in $\cG_{s,0}$, we define
\begin{equation}\label{e:gegep}
\langle g,g^\prime\rangle_{\cG_{s,0}}:=\sum_{\beta\in \Gamma^s} 
\langle g_\beta,g^\prime_\beta\rangle_{\cH_{\tau(d(\beta))}},\quad g,g^\prime\in\Gamma^s.
\end{equation}
We observe that the definition in \eqref{e:gegep} 
 yields an inner product $\langle \cdot,\cdot\rangle_{\cG_{s,0}}$
 when restricting to $\cG_{s,0}\times \cG_{s,0}$.

 Let $\bT_s\colon\cG_{s,0}\ra\cF_s$ be the linear operator 
defined by
\begin{equation}\label{e:betes}
(\bT_s g)_\gamma:=\sum_{\beta\in\Gamma^s} T(\gamma^*\beta)g_\beta,\quad g\in\cG_{s,0},\ \gamma\in\Gamma^s.
\end{equation}
We denote by $\cG_s$ the range of $\bT_s$, that is, $f=\{f_\gamma\}_{\gamma\in\Gamma^s}\in\cF_s$
belongs to $\cG_s$ if, by definition, there exists $g\in\cG_{s,0}$ such that
\begin{equation}\label{e:fegabet}
f_\gamma=(\bT_s g)_\gamma=\sum_{\beta\in\Gamma^s} T(\gamma^*\beta)g_\beta,
\quad \gamma\in\Gamma^s.
\end{equation}

For every $g,g^\prime \in\cG_{s,0}$ we have
\begin{align*}
\langle \bT_sg,g^\prime\rangle_{\cG_{s,0}} 
& = \langle \sum_{\gamma\in\Gamma^s} (\bT_sg)_\gamma,
g^\prime_\gamma\rangle_{\cH_{\tau(d(\gamma))}} 
 =  \langle \sum_{\gamma\in\Gamma^s} \sum_{\beta\in\Gamma^s} 
T(\gamma^*\beta)g_\beta,g^\prime_\gamma\rangle_{\cH_{\tau(d(\gamma))}} \\
& =  \sum_{\gamma\in\Gamma^s} \sum_{\beta\in\Gamma^s} 
\langle g_\beta,T(\gamma^*\beta)^*g^\prime_\gamma\rangle_{\cH_{\tau(d(\gamma))}} 
 = \sum_{\gamma\in\Gamma^s} \sum_{\beta\in\Gamma^s} 
 \langle g_\beta,T(\beta^*\gamma)g^\prime_\gamma\rangle_{\cH_{\tau(d(\beta))}} \\
 & = \sum_{\beta\in\Gamma^s} \sum_{\gamma\in\Gamma^s} 
 \langle g_\beta,T(\beta^*\gamma)g^\prime_\gamma\rangle_{\cH_{\tau(d(\beta))}} 
  = \sum_{\beta\in\Gamma^s} 
 \langle g_\beta,\sum_{\gamma\in\Gamma^s} 
 T(\beta^*\gamma)g^\prime_\gamma\rangle_{\cH_{\tau(d(\beta))}} \\
& = \sum_{\beta\in\Gamma^s} 
 \langle g_\beta,(\bT_s g^\prime)_\beta\rangle_{\cH_{\tau(d(\beta))}} 
  = \langle g,\bT_sg^\prime\rangle_{\cG_{s,0}},
\end{align*}
where we took into account that the sums have only a finite number of nonzero terms and hence we
can change the summation order and that $T$ is Hermitian, in the sense of (HM2) in Definition~\ref{d:hm}.
In conclusion, $\bT_s$ has the following symmetry property
\begin{equation}\label{e:tes}
\langle \bT_sg,g^\prime\rangle_{\cG_{s,0}}=
\langle g,\bT_sg^\prime\rangle_{\cG_{s,0}},\quad g,g^\prime \in\cG_{s,0}.
\end{equation}

On $\cG_{s,0}$ we define the pairing $\langle\cdot,\cdot\rangle_{\bT_s}$ by
\begin{equation}\label{e:gegepa}
\langle g,g^\prime\rangle_{\bT_s}:=\langle \bT_s g,g^\prime\rangle_{\cG_{s,0}}=
\sum_{\beta,\gamma\in\Gamma^s} \langle 
T(\beta^*\gamma)g_\gamma,g^\prime_\beta\rangle_{\cH_{\tau(d(\beta))}},\quad g,g^\prime\in \cG_{s,0},
\end{equation}
which, due to the positive semidefiniteness assumption (a), is positive semidefinite. Also, due to \eqref{e:tes},
 it is (conjugate) 
symmetric, and linear in the first variable. In particular, $\langle\cdot,\cdot,\rangle_{\bT_s}$
satisfies the Schwarz inequality and hence
\begin{equation}\label{e:kerebet}
\ker(\bT_s)  = \{g\in\cG_{s,0}\mid \langle g,g\rangle_{\bT_s}=0\} 
 = \{g\in\cG_{s,0}\mid \langle g,g^\prime\rangle_{\bT_s}=0\mbox{ for all }g^\prime\in\cG_{s,0}\}.
\end{equation}

Let $f,f^\prime\in\cG_s$ be arbitrary, hence there exist $g,g^\prime\in\cG_{s,0}$ such that 
$f=\bT_sg$ and $f^\prime=\bT_sg^\prime$. Then, by \eqref{e:tes} we have
\begin{equation*}
\langle g,g^\prime\rangle_{\bT_s}=\langle \bT_sg,g^\prime\rangle_{\cG_{s,0}}=
\langle g,\bT_sg^\prime\rangle_{\cG_{s,0}}.
\end{equation*}
This observation enables us 
to show that the definition of the pairing $\langle\cdot,\cdot\rangle_{\cG_s}$
given by
\begin{equation}\label{e:fefep}
\langle f,f^\prime\rangle_{\cG_s} := \langle g,g^\prime\rangle_{\bT_s}=
\langle \bT_sg,g^\prime\rangle_{\cG_{s,0}}=\langle f,g^\prime\rangle_{\cG_{s,0}},\end{equation}
where $f=\bT_s g$, $f^\prime=\bT_sg^\prime$, for some  $g,g^\prime\in\cG_{s,0}$,
is correct, that is, it does not depend on the particular representations $f=\bT_s g$  and 
$f^\prime=\bT_sg^\prime$, for $g,g^\prime\in\cG_{s,0}$. 

Clearly, the pairing $\langle\cdot,\cdot\rangle_{\cG_s}$ defined at \eqref{e:fefep} 
is positive semidefinite, (conjugate) symmetric, and linear in 
the first variable. In the following we show that it is positive definite as well, 
hence an inner product.
Indeed, let $f\in\cG_s$ be such that $\langle f,f\rangle_{\cG_s}=0$. By the positive 
semidefiniteness assumption (1), it follows that $\langle f,f^\prime\rangle_{\cG_s}=0$ for all 
$f^\prime\in \cG_s$. For any $\gamma\in\Gamma^s$ and any $h\in\cH_{\tau(d(\gamma))}$, let 
$\delta_\gamma h\in\cG_{s,0}$ be defined by
\begin{equation}\label{e:deltah}
(\delta_\gamma h)_\beta:=\begin{cases} h, & \beta=\gamma,\\ 0, & \beta\neq\gamma,
\end{cases},\quad \beta\in\Gamma^s.
\end{equation} 
Letting $f^\prime=\bT_s(\delta_\gamma f_\gamma)\in\cG_s$, by \eqref{e:tes} 
and \eqref{e:fefep} it follows
\begin{equation*}
0=\langle f,f^\prime\rangle_{\cG_s}=\langle f,\delta_\gamma f_\gamma\rangle_{\cG_{s,0}}
= \sum_{\beta\in\Gamma^s} \langle f_\beta,(\delta_\gamma 
f_\gamma)_\beta\rangle_{\cH_{\tau(d(\beta))}}
=\langle f_\gamma,f_\gamma\rangle_{\cH_{\tau(d(\gamma))}},
\end{equation*}
hence $f_\gamma=0$ for all $\gamma\in\Gamma^s$, that is, $f=0$.

Finally,  let $\cL_s$ be the completion of the inner product space 
$(\cG_s,\langle\cdot,\cdot\rangle_{\cG_s})$ to a Hilbert space and then, for any $x\in X$ define
\begin{equation}\label{e:kax}
\cK_x:=\bigoplus_{s\in\tau^{-1}(\{x\})} \cL_s,
\end{equation}
that is, $\cK_x$ is the orthogonal sum of all Hilbert spaces $\cL_s$, where the sum is taken over all
$s\in S $ such that $\tau(s)=x$. Then, when considering $\cL_s$ as subspaces of $\cK_x$, for all
$s\in S$ such that $\tau(s)=x$, we define
\begin{equation}\label{e:dax}
\cD_x:=\lin\{\cG_s,\mid s\in\tau^{-1}(\{x\})\}.
\end{equation}
Clearly, $\cD_x$ is dense in $\cK_x$ for all $x\in X$. Also, let us observe that the spaces
$\cG_s$, when viewed as subspaces of $\cK_x$, for all $s\in S$ such that $\tau(s)=x$, are 
mutually orthogonal.
\medskip

\textbf{Step 2.} \emph{Construction of the bundle of operators $\bV=\{V_s\}_{s\in S }$.}
\medskip

For each $s\in S $ we define the linear operator $V_s\colon \cH_{\tau(s)}\ra \cK_{\tau(s)}$ 
by
\begin{equation}\label{e:veseh}
V_s h:=\{T(\gamma^*)h\}_{\gamma\in\Gamma^s},\quad h\in\cH_{\tau(s)}.
\end{equation} We show that, by this definition, 
the range of $V_s$ is in $\cG_s$.

Indeed, for any $h\in \cH_{\tau(s)}$
and any $\gamma\in\Gamma^s$, with notation as in \eqref{e:deltah}, we have
\begin{equation}\label{e:tegah}
T(\gamma^*)h=\sum_{\beta\in\Gamma^s}T(\gamma^*\beta)(\delta_{\epsilon_s} h)_\beta,
\end{equation}
in particular, in view of the definition of the operator $\bT_s$ as in \eqref{e:betes}, this implies that
\begin{equation}\label{e:vesehab}
V_sh=\bT_s(\delta_{\epsilon_s} h),\quad h\in\cH_{\tau(s)},
\end{equation}
where $\epsilon$ is the unit of $\Gamma$, hence $V_s h\in\cG_s$.

Since $\cG_s\subseteq \cL_s$, which, modulo the canonical embedding of $\cL_s$ 
in $\cK_{\tau(s)}$, see \eqref{e:kax}, is regarded as 
a subspace of $\cK_{\tau(s)}$, we actually have $V_s\colon \cH_{\tau(s)}\ra \cK_{\tau(s)}$.
We observe that, for each $h\in\cH_{\tau(s)}$, we have
\begin{align*}
\langle V_sh,V_sh\rangle_{\cG_s} & = 
\langle \bT_s(\delta_{\epsilon_s}h),\bT_s(\delta_{\epsilon_s}h)\rangle_{\cG_s} \\
& = \langle \bT_s(\delta_{\epsilon_s}h),\delta_{\epsilon_s}h\rangle_{\cG_{s,0}} \\
& =\sum_{\gamma\in\Gamma^s} \langle T(\gamma^*)h,
(\delta_{\epsilon_s} h)_{\gamma}\rangle_{\cH_{\tau(d(\gamma))}} \\
&  =\langle T(\epsilon_s)h,h\rangle_{\cH_{\tau(s)}}\leq 
\|T(\epsilon_s)\|\langle h,h\rangle_{\cH_{\tau(s)}},
\end{align*}
hence $V_s$ is bounded.

We can also calculate, for later use, the adjoint $V_s^*$. In order to do this, for any $f\in\cG_s$ 
and any $h\in\cL_s\subseteq\cH_{\tau(s)}$, we have
\begin{align*}
\langle f,Vh\rangle_{\cG_s}=\langle f,\{T(\gamma^*)h\}_{\gamma\in\Gamma^s}\rangle_{\cG_s}
=\sum_{\gamma\in\Gamma^s} 
\langle f_\gamma,(\delta_{\epsilon_s}h)_\gamma\rangle_{\cH_{\tau(d(\gamma))} }
=\langle f_{\epsilon_s},h\rangle_{\cH_{\tau(s)}},
\end{align*}
hence, 
\begin{equation}\label{e:vestar}
V_s^*f=f_{\epsilon_s},\quad f\in\cG_s.
\end{equation}
Also, $V^*_s k=0$ for all $k\in\cK_{\tau(s)}\ominus \cL_s$.\medskip

\textbf{Step 3.} \emph{Construction of the $*$-representation 
$\Phi\colon \Gamma\ra\Gamma_{\bK,\bD}$.}\medskip

For the bundle of Hilbert spaces $\bK=\{\cK_x\}_{x\in X}$, see \eqref{e:kax},
and the bundle of subspaces $\bD=\{\cD_x\}_{x\in X}$, see \eqref{e:dax},
recall the definition of the $*$-semigroupoid with unit $\Gamma_{\bK,\bD}$ in Example~\ref{ex:semuh}.
For any $\alpha\in\Gamma$, we define 
$\Phi(\alpha)\colon\cG_{d(\alpha)}\ra\cG_{c(\alpha)}$ by
\begin{equation}\label{e:omaf}
\Phi(\alpha)f:=\{f_{\alpha^*\gamma}\}_{\gamma\in \Gamma^{c(\alpha)}},\quad f\in\cG_{d(\alpha)}.
\end{equation}
We have firstly have to show that the range of $\Phi(\alpha)$ is indeed in $\cG_{c(\alpha)}$, that is, we have to show
that, for any $f\in\cG_{d(\alpha)}$ there exists $e\in\cG_{c(\alpha),0}$ such 
that $\Phi(\alpha)f=\bT_{c(\alpha)} e$. To this end, since $f\in\cG_{d(\alpha)}$, by definition
there exists $g\in\cG_{d(\alpha),0}$ such that $f=\bT_{d(\alpha)}g$, that is, 
\begin{equation*}
f_\eta=(\bT_{d(\alpha)}g)_\eta=\sum_{\zeta\in\Gamma^{d(\alpha)}} T(\eta^*\zeta)g_\zeta,
\quad \eta\in\Gamma^{d(\alpha)}.
\end{equation*} Then, define 
$e=\{e_\lambda\}_{\lambda\in\Gamma^{c(\alpha)}}\in\cG_{c(\alpha),0}$ in the following way: for
each $\lambda\in\Gamma^{c(\alpha)}$ let
\begin{equation}\label{e:ega}
e_\lambda=\begin{cases} 0,& \mbox{ if the equation }\alpha\beta=\lambda\mbox{ has no solution }
\beta\mbox{ in }
\Gamma_{d(\lambda)}^{d(\alpha)}, \\[2ex] \displaystyle\sum_{\beta\in\Gamma_{d(\lambda)}^{d(\alpha)},\ 
\alpha\beta=\lambda}g_\beta, & \mbox{ otherwise.}
\end{cases}
\end{equation}
We observe that, since $g$ has finite support, the sum in the definition of 
$e_\lambda$ has only a finite  
number of nonzero terms, hence always makes sense, and that $e$ itself has finite support, 
more precisely, \begin{equation*}\supp(e)\subseteq \{\alpha\eta\mid \eta\in\supp(g)\}.\end{equation*}
Then, for any $\gamma\in\Gamma^{c(\alpha)}$, we have
\begin{align}\nonumber
\sum_{\lambda\in\Gamma^{c(\alpha)}} \! T(\gamma^*\lambda)e_\lambda  & = \!\!
\sum_{\lambda\in\Gamma^{c(\alpha)}} \!\!
T(\gamma^*\lambda)\!\!
\sum_{\beta\in\Gamma_{d(\lambda)}^{d(\alpha)},\ \alpha\beta=\lambda}\!\!\! g_\beta \\ 
& =\!\! \sum_{\beta\in\Gamma^{d(\alpha)}} \!T(\gamma^*\alpha\beta) g_\beta
  =f_{\alpha^*\gamma} = \bigl(\Phi(\alpha)f\bigr)_\gamma,\label{e:sulam}
\end{align}
hence $\Phi(\alpha)f\in \cG_{c(\alpha)}$.

We now show that $\Phi$ is a $*$-representation of $\Gamma$ on the bundle of vector spaces
$\cG=\{\cG_s\}_{s\in S}$. Indeed, for any $(\alpha,\beta)\in \Gamma^{(2)}$, 
let $f\in\cG_{d(\beta)}$ and
$g=\Phi(\beta)f$ hence, in view of \eqref{e:omaf}, we have $g_\gamma=f_{\beta^*\gamma}$ for all
$\gamma\in\Gamma^{c(\beta)}$ and then, since $c(\beta)=d(\alpha)$, we have 
$\Phi(\alpha)g=\{g_{\alpha^*\eta}\}_{\eta\in\Gamma^{c(\alpha)}}$, hence
\begin{equation*}
g_{\alpha^*\eta}=f_{\beta^*\alpha^*\eta}=
(\Phi(\alpha\beta)f)_{\eta},\quad \eta\in \Gamma^{c(\alpha)},
\end{equation*}
which proves that
\begin{equation}\label{e:omab}
\Phi(\alpha\beta)=\Phi(\alpha)\Phi(\beta),\quad \alpha,\beta\in\Gamma.
\end{equation}

Further on, for any $\alpha\in\Gamma$, let 
$f\in\cG_{d(\alpha)}$ and $f^\prime\in\cG_{c(\alpha)}$ be arbitrary. Then there exist 
$g\in\cG_{d(\alpha),0}$ and $g^\prime\in\cG_{c(\alpha),0}$ such that $f=\bT_{d(\alpha)}g$ and
$f^\prime=\bT_{c(\alpha)}g^\prime$ hence
\begin{align}
\langle \Phi(\alpha)f,f^\prime\rangle_{\cG_{c(\alpha)}} & = 
\sum_{\gamma\in\Gamma^{c(\alpha)} }
\langle f_{\alpha^*\gamma},g_\gamma^\prime\rangle_{\cH_{\tau(d(\gamma))}} \nonumber\\
& = \sum_{\gamma\in\Gamma^{c(\alpha)}} \sum_{\eta\in\Gamma^{d(\alpha)}}
\langle T(\gamma^*\alpha\eta)g_\eta,g^\prime_\gamma\rangle_{\cH_{\tau(d(\gamma))}} \nonumber\\
& =  \sum_{\eta\in\Gamma^{d(\alpha)}} \sum_{\gamma\in\Gamma^{c(\alpha)}} \langle g_\eta,
T(\eta^*\alpha^*\gamma)g^\prime_\gamma\rangle_{\cH_{\tau(d(\eta))}} \label{e:falfas}\\
& = \sum_{\eta\in\Gamma^{d(\alpha)}} \langle g_\eta,f^\prime_{\alpha\eta}
\rangle_{\cH_{\tau(d(\eta))}}\nonumber \\
& = \langle f, \Phi(\alpha^*)f^\prime\rangle_{\cG_{d(\alpha)}}.\nonumber
\end{align}
This shows that $\Phi(\alpha)\colon\cG_{d(\alpha)}\ra\cG_{c(\alpha)}$, 
as a densely defined linear operator from $\cL_{d(\alpha)}$ to
$\cL_{c(\alpha)}$,  has the property
\begin{equation}\label{e:omast}
\Phi(\alpha^*)\subseteq \Phi(\alpha)^*,\quad \alpha\in\Gamma.
\end{equation}

Finally, we lift now the operator $\Phi(\alpha)$ to 
the operator $\Phi(\alpha)\colon \cD_{\tau(d(\alpha))}\ra \cD_{\tau(c(\alpha))}$, 
yet denoted by $\Phi(\alpha)$, in the following way: with respect to the decompositions
\begin{equation}\label{e:cekat}
\cD_{\tau(d(\alpha))}=\bigoplus_{t\in S ,\ \tau(t)=\tau(d(\alpha))}\cG_t,\quad
\cD_{\tau(c(\alpha))}=\bigoplus_{t\in S ,\ \tau(t)=\tau(c(\alpha))}\cG_t,
\end{equation}
letting
\begin{equation}\label{e:phia}
\Phi(\alpha)|_{\cG_t}=0,\mbox{ for all }t\in S \setminus\{d(\alpha)\},\ \tau(t)=\tau(d(\alpha)).
\end{equation} 
Then, it follows that \eqref{e:omab} and \eqref{e:omast} hold for the
lifted operators $\Phi(\alpha)$ and $\Phi(\beta)$ as well, hence $\Phi$ is a 
$*$-representation of $\Gamma$ to $\Gamma_{\bK,\bD}$. 

In addition, it is clear from the construction, see \eqref{e:kax}, that,
whenever $\alpha,\beta\in \Gamma$ are such that $c(\alpha)\neq c(\beta)$ but
$\tau(c(\alpha))=\tau(c(\beta))$, we have 
$\Phi(\alpha)\cD_{\tau(d(\alpha))}\perp \Phi(\beta)\cD_{\tau(d(\beta))}$. Hence the
orthogonality property as in Definition~\ref{d:orthogonalu} holds.
\medskip

\textbf{Step 4.} \emph{$T(\alpha) = V_{c(\alpha)}^* \Phi(\alpha) V_{d(\alpha)}$ for all 
$\alpha\in\Gamma$.}\medskip

Indeed, let $\alpha\in\Gamma$ and $h\in \cH_{\tau(d(\alpha))}$ be arbitrary. Then, in view 
of \eqref{e:veseh}, \eqref{e:omaf}, and \eqref{e:vestar}, we have
\begin{align*}
V_{c(\alpha)}^* \Phi(\alpha)V_{d(\alpha)}h &  = V_{c(\alpha)}^* \Phi(\alpha) 
\{T(\gamma^*)h\}_{\gamma\in\Gamma^{d(\alpha)} } \\
& = 
V_{c(\alpha)}^*\{T(\beta^*\alpha)h\}_{\beta\in\Gamma^{c(\alpha)}} =T(\epsilon_{c(\alpha)}^*\alpha)h
= T(\alpha)h.
\end{align*}

\textbf{Step 5.} \emph{The dilation $(\bK;\bD;\Phi;\bV)$ 
is minimal in the sense of Definition~\ref{d:minimalu}.}
\medskip

Let $x\in X$ be arbitrary. In view of the definition of the linear manifold $\cD_x$ as in \eqref{e:dax}
and the definition of the
inner product space $(\cG_s,\langle\cdot,\cdot\rangle_{\cG_s})$ for arbitrary $s\in S $, see \eqref{e:fegabet},
it is sufficient to show that 
\begin{equation}\label{e:lome}
\lin\{\Phi(\alpha)V_{d(\alpha)}\cH_{\tau(d(\alpha))}\mid \alpha\in\Gamma^s\}= \cG_s.\end{equation}

We first observe that, due to the fact that the 
range of $V_s$ is in $\cG_s$ and $\Phi(\alpha)$ maps $\cG_{d(\alpha)}$ to $\cG_{c(\alpha)}$ 
for any $\alpha\in\Gamma^s$, 
the direct inclusion in \eqref{e:lome} follows.

In order to prove the converse inclusion, 
let $f\in \cG_s$ be arbitrary, hence there exists $g\in\cG_{s,0}$ such that $f=\bT_sg$ and then
\begin{align*}
f & =\{\sum_{\beta\in\Gamma^s} T(\gamma^*\beta)g_\beta\}_{\gamma\in\Gamma^s}
=\sum_{\beta\in\Gamma^s} \{T(\gamma^*\beta)g_\beta\}_{\gamma\in\Gamma^s} \\
 & = \sum_{\beta\in\Gamma^s} \Phi(\beta)V_{d(\beta)} g_\beta\in \lin\{\Phi(\alpha)
V_{d(\alpha)}\cH_{\tau(d(\alpha))}\mid \alpha\in\Gamma^s\},
\end{align*}
where, at the last equality, we used \eqref{e:omaf} and \eqref{e:veseh} and observe that, since $g$
has finite support, the sums have only a finite number of nonzero terms. Thus, \eqref{e:lome}
is proven and the assertion follows.\medskip

\textbf{Step 6.} \emph{The orthogonal minimal dilation $(\bK;\bD;\Phi;\bV)$ is unique, up to a 
unitary equivalence.}\medskip

Let $(\bK^\prime;\bD^\prime;\Phi^\prime;\bV^\prime)$ 
be another orthogonal minimal dilation of $T$, hence,
for any $x\in X$, we have
\begin{equation}\label{e:dexp}
\cD_x^\prime= \lin \{\Phi^\prime(\alpha)V^\prime_{\tau(d(\alpha))}\cH_{\tau(d(\alpha))}
\mid \alpha\in\Gamma^s,\mbox{ such that }
s\in \tau^{-1}(\{x\})\},
\end{equation}
and, whenever $\alpha,\beta\in \Gamma$ are such that $c(\alpha)\neq c(\beta)$ but
$\tau(c(\alpha))=\tau(c(\beta))$ we have 
$\Phi^\prime(\alpha)\cK^\prime_{\tau(d(\alpha))}\perp \Phi^\prime(\beta)\cK^\prime_{\tau(d(\beta))}$.

Then, for any $x\in X$, any $s\in S$ such that
$x=\tau(s)$, and
any finitely supported cross-sections $h=\{h_\alpha\}_{\alpha\in\Gamma^s}$, with 
$h_\alpha\in\cH_{\tau(d(\alpha))}$ for all $\alpha\in\Gamma^s$, and 
$k=\{k_\beta\}_{\beta\in\Gamma^s}$, with $k_\beta\in\cH_{\tau(d(\beta))}$ for all $\beta\in\Gamma^s$, 
in view of the property (UD3) we have
\begin{align*}
\sum_{\alpha,\beta\in\Gamma^s} \langle \Phi(\alpha)V_{d(\alpha)} h_\alpha, 
\Phi(\beta)V_{d(\beta)} k_\beta\rangle_{\cK_x} &  = 
\sum_{\alpha,\beta\in\Gamma^s} 
\langle T(\beta^*\alpha)h_\alpha,k_\beta\rangle_{\cH_{\tau(d(\beta))}} \\
& = \sum_{\alpha,\beta\in\Gamma^s} 
\langle \Phi^\prime(\alpha)V_{d(\alpha)}^\prime h_\alpha, 
\Phi^\prime(\beta)V_{d(\beta)}^\prime k_\beta\rangle_{\cK_x^\prime},
\end{align*}
hence, the operator 
\begin{equation}\label{e:uxes}U_{x,s}\colon \lin\{\Phi(\alpha)V_{d(\alpha)}\cH_{\tau(d(\alpha))}\mid 
\alpha\in \Gamma^s\}\ra \lin\{\Phi^\prime(\alpha)V^\prime_{d(\alpha)}\cH_{\tau(d(\alpha))}\mid 
\alpha\in \Gamma^s\}
\end{equation} defined by
\begin{equation}\label{e:udef}
U_{x,s}\biggl(\sum_{\alpha\in\Gamma^s} \Phi(\alpha)V_s h_\alpha\biggr):=
\sum_{\alpha\in\Gamma^s} \Phi^\prime(\alpha)V_s^\prime h_\alpha,
\end{equation}
for any finitely supported cross-section $h=\{h_\alpha\}_{\alpha\in\Gamma^s}$, with 
$h_\alpha\in\cH_{\tau(d(\alpha))}$ for all $\alpha\in\Gamma^s$, is correctly defined and an isometry. 
Note that $U_{x,s}$ defined as in \eqref{e:uxes} 
is bijective and hence it is uniquely extended to a unitary operator, yet denoted by
\begin{equation*}
U_{x,s}\colon \clos\lin\{\Phi(\alpha)V_{d(\alpha)}\cH_{\tau(d(\alpha))}\mid 
\alpha\in \Gamma^s\}\ra \clos\lin\{\Phi^\prime(\alpha)V^\prime_{d(\alpha)}\cH_{\tau(d(\alpha))}\mid 
\alpha\in \Gamma^s\}.
\end{equation*}

We now use the fact that both unitary dilations are orthogonal and minimal and define the 
unitary operator $U_x$ by
\begin{equation}\label{e:udefex}
U_x=\bigoplus_{s\in\tau^{-1}(\{x\})} U_{x,s}\colon \cK_x\ra\cK_x^\prime.
\end{equation}
Then, from \eqref{e:dax}, \eqref{e:lome}, \eqref{e:uxes}, and \eqref{e:dexp}, 
it follows that $U_x\cD_x=\cD^\prime_x$ for all $x\in X$.
From definitions \eqref{e:udefex} and \eqref{e:udef} 
it follows that the bundle $\bU=\{U_x\}_{x\in X}$ intertwines the
$*$-representations $\Phi$ and $\Phi^\prime$, that is,
\begin{equation*}
U_{\tau(c(\beta))} \Phi(\beta)|_{\cD_{\tau(d(\beta))}}
=\Phi^\prime(\beta)U_{\tau(d(\beta))}|_{\cD_{\tau(d(\beta))}},\quad  
\beta\in\Gamma.
\end{equation*}

By Remark~\ref{r:epsu}, for any $s\in S $, the operator
$\Phi(\epsilon_s)$ acts like the identity operator on the subspace
$\lin \{\Phi(\alpha)V_{d(\alpha)}\cH_{\tau(d(\alpha))}\mid \alpha\in\Gamma^s\}$ and, 
similarly, the operator
$\Phi^\prime(\epsilon_s)$ acts like the identity operator on the subspace
$\lin \{\Phi^\prime(\alpha)V^\prime_{d(\alpha)}\cH_{\tau(d(\alpha))}\mid 
\alpha\in\Gamma^s\}$, hence, by the definitions of $U_x$ and $U_{x,s}$, see \eqref{e:udef}
and \eqref{e:udefex}, it follows that,
for all $x\in X$ such that $s\in \tau^{-1}(\{x\})$ and all $h\in\cH_{x}$, we have
\begin{equation*}
U_x V_sh=U_x\Phi(\epsilon_s)V_sh=\Phi^\prime(\epsilon_s)V_s^\prime h=V_s^\prime h.\qedhere
\end{equation*}
\end{proof}

\begin{remark}\label{r:rkhs} The construction of the spaces $\cK_x$ in the proof of Theorem~\ref{t:nagyu},
see \eqref{e:kax}, is
performed in such a way that it is a space of functions on the fibres $\Gamma^s$, for $s\in S$ such that $x=\tau(s)$,
more precisely a reproducing kernel 
Hilbert space on $\Gamma^s$, and not just an abstract completion. To see this, in the following we show that the 
completion of the inner product space $(\cG_s,\langle\cdot,\cdot\rangle_{\cG_s})$ can always be performed
inside of $\cF_s$, and hence is a function space on $\Gamma^s$.

Indeed, recall that, for each $s\in S$,
$\cF_s$ denotes the vector space of all cross-sections
$f=\{f_\gamma\}_{\gamma\in\Gamma^s}$ subject to 
the condition that $f_\gamma\in\cH_{\tau(d(\gamma))}$ for all $\gamma\in\Gamma^s$. On $\cF_s$ we can
consider $\cT_\mathrm{e}$, 
the weakest locally convex topology that makes continuous all operators of point evaluation 
$E_\beta\colon \cF_s\ra \cH_{\tau(d(\beta))}$, for $\beta\in \Gamma^s$, 
where $E_\beta f:=f_\beta$, for all $f=\{f_\gamma\}_{\gamma\in\Gamma^s}\in\cF_s$. Clearly,  $\cT_\mathrm{e}$
is complete. Thus, in order to see that the completion of the inner product space 
$(\cG_s,\langle\cdot,\cdot\rangle_{\cG_s})$ can always be performed inside of $\cF_s$, it is sufficient to observe
that the strong topology on $\cG_s$ induced by the inner product $\langle\cdot,\cdot\rangle_{\cG_s}$ is stronger
than the topology $\cT_\mathrm{e}$.
\end{remark}

\begin{remark}\label{r:cohu} The orthogonal minimal generalised dilation $(\bK;\bD;\Phi;\bV)$ 
constructed during the proof
of the implication (1)$\Ra$(2) of the proof of Theorem~\ref{t:nagyu} has yet another property. Namely,
for any $s\in S $ the operator $\Phi(\epsilon_s)$ is the orthogonal projection onto $\cL_s$, 
the completion of the inner product space $(\cG_s;\langle\cdot,\cdot\rangle_{\cG_s})$ to 
a Hilbert space. In view of \eqref{e:lome} and Remark~\ref{r:epsu}, it follows that
\begin{equation*}
\clos\lin\{\Phi(\alpha)V_{d(\alpha)}\cH_{\tau(d(\alpha))}\mid \alpha\in\Gamma^s\}
=\ran(\Phi(\epsilon_s)).
\end{equation*}
Consequently,
by the definition of the space $\cK_x$, for arbitrary $x\in X$, see \eqref{e:kax}, the
orthogonal projections $\{\Phi(\epsilon_s)\mid s\in \tau^{-1}(\{x\})\}$ are mutually orthogonal and then,
by the minimality property we have
\begin{equation}
\sum_{s\in\tau^{-1}(\{x\})} \Phi(\epsilon_s) =I_{\cK_x},\quad x\in X.\label{e:sesit}
\end{equation}
If $x\in X$ is such that $\tau^{-1}(\{x\})$ is infinite, the meaning of the sum in \eqref{e:sesit} is in the sense of 
summability with respect to the strong operator topology on $\cK_x$.
\end{remark}

In the next corollary we show that the usage of the term dilation in Definition~\ref{d:udilation} is correct in the sense
that, when the map $T$ is unital then the Hilbert space $\cK_x$ is actually an isometric extension 
of the Hilbert space $\cH_x$, for all $x\in X$, and that $T(\alpha)$ is a compression of $\Phi(\alpha)$.

\begin{corollary} \label{c:embedu} With notation and assumptions 
as in Theorem~\ref{t:nagyu}, if $T\in\fP_{\bH,\tau}(\Gamma)$ is unital
in the sense that, for each $x\in X$ the set
$\{T(\epsilon_s)\mid s\in \tau^{-1}(\{x\})\}$ consists in mutually orthogonal projections in $\cH_x$ such 
that
\begin{equation}\label{e:susintau}
\sum_{s\in\tau^{-1}(\{x\})} T(\epsilon_s)=I_{\cH_x},
\end{equation} 
where the meaning of the sum is in the sense of summability with respect to the strong operator topology on $\cH_x$, 
then the generalised dilation 
$(\bK;\bD;\Phi;\bV)$ constructed during the proof of the implication \emph{(1)$\Ra$(2)} 
has the property that, for each 
$x\in X$, the Hilbert space $\cH_x$ is naturally embedded in $\cK_x$ and, with respect to 
this embedding, we have
\begin{equation*}
T(\alpha)=P_{\cH_{\tau(c(\alpha))}} \Phi(\alpha)|_{\cH_{\tau(d(\alpha))}},\quad \alpha\in\Gamma.
\end{equation*}
\end{corollary}

\begin{proof}
We observed in Remark~\ref{r:cohu}, that, for any $s\in S $ 
we have $\Phi(\epsilon_s)=P_{\cL_s}$. Note that, by construction of the operators $V_s$, see
\eqref{e:veseh}, we have $\ran(V_s)\subseteq \cL_s$ hence,
\begin{equation*}V_s^*V_s = V_s^* P_{\cL_s} V_s= V_s^*\Phi(\epsilon_s)V_s= 
T(\epsilon_s),
\end{equation*}
hence $V_s$ is a partial isometry with initial space $T(\epsilon_s)$. 
Then, on the one hand, we observe that
for arbitrary $x\in X$ the set $\{V_s\mid s\in\tau^{-1}(\{x\})\}$ 
consists in partial isometries on $\cH_x$ with mutually orthogonal initial spaces and then, 
by assumption \eqref{e:susintau}, these initial spaces sum up 
to $\cH_x$. On the other hand, since $\ran(V_s)\subseteq \cL_s$ and taking into account \eqref{e:kax}, it follows
that the final spaces of $V_s$, $s\in\tau^{-1}(\{x\})$, are mutually orthogonal.
From here it follows that the operator
\begin{equation*}
W_x=\sum_{s\in\tau^{-1}(\{x\})} V_s\colon \cH_x\ra\cK_x
\end{equation*}
exists, in the sense of summability with respect to the strong operator topology, if $\tau^{-1}(\{x\})$ is infinite,
and is an isometry.
Identifying $W_x\cH_x$ with $\cH_x$  for any $x\in X$,
we have
\begin{equation*}
T(\alpha)=V_{c(\alpha)}^* \Phi(\alpha)V_{d(\alpha)}= W_{\tau(c(\alpha))}^* \Phi(\alpha)
W_{\tau(d(\alpha))}=
P_{\cH_{\tau(c(\alpha))}} \Phi(\alpha)|_{\cH_{\tau(d(\alpha))}},\quad \alpha\in\Gamma.\qedhere
\end{equation*}
\end{proof}

\section{Dilations with Bounded Operators}\label{s:bd}

We continue with notation as in the previous section and consider a map $T\in\fH_{\bH,\tau}(\Gamma)$ as in
Definition~\ref{d:hm}. In this section we want to get characterisations of those maps $T$ that admit dilations
with bounded operators.

\begin{definition}\label{d:dilation} Given $T\in\fH_{\bH,\tau}(\Gamma)$,
a triple $(\bK;\Phi;\bV)$ is called a \emph{dilation} of $T$ if it satisfies the following conditions.
\begin{itemize}
\item[(D1)] $\bK=\{\cK_x\}_{x\in X}$ is a bundle of Hilbert spaces.
\item[(D2)] $\bV=\{V_s\}_{s\in S }$ is a bundle of operators, where
$V_s\in\cB(\cH_{\tau(s)},\cK_{\tau(s)})$ for all $s\in S $.
\item[(D3)] The pair $(\tau;\Phi)$ is a $*$-representation of $\Gamma$ on $\bK$ such that
\begin{equation*}
T(\alpha) = V_{c(\alpha)}^* \Phi(\alpha) V_{d(\alpha)},\quad \alpha\in\Gamma.
\end{equation*}
\end{itemize}
\end{definition}

An analogue of Remark~\ref{r:epsu} holds as well.

The concept of a minimal dilation can be obtained from
Definition~\ref{d:minimalu} and can be formalised as follows.

\begin{definition}\label{d:minbd} Given $T\in\fH_{\bH,\tau}(\Gamma)$, a dilation $(\bK;\Phi;\bV)$ of $T$ is
\emph{minimal} if
\begin{equation*}
\cK_x=\clos\lin \{\Phi(\alpha)V_{d(\alpha)}\cH_{\tau(d(\alpha))}\mid \alpha\in\Gamma\mbox{ such 
that }\tau(c(\alpha))=x\},\quad x\in X.
\end{equation*}
\end{definition}

The concept of unitarily equivalent dilations is derived from Definition~\ref{d:ueu}. 

\begin{definition} \label{d:ue} Two dilations 
$(\bK;\Phi;\bV)$ and $(\bK^\prime;\Phi^\prime;\bV^\prime)$
of the map $T\in\fP_{\bH,\tau}(\Gamma)$ are called \emph{unitarily equivalent} if
there exists a bundle of operators $\bU=\{U_x\}_{x\in X}$ subject to the following conditions.
\begin{itemize}
\item[(UE1)] For each $x\in X$, $U_x\colon \cK_x\ra\cK^\prime_x$ is unitary.
\item[(UE2)] The bundle $\bU$ intertwines the representations $\Phi$ and $\Phi^\prime$ in the sense
that
\begin{equation*}
U_{\tau(c(\alpha))} \Phi(\alpha) = \Phi^\prime(\alpha)U_{\tau(d(\alpha))},\quad \alpha\in\Gamma.
\end{equation*}
\item[(UE3)] The bundle $\bU$ maps the bundle $\bV$ to the bundle $\bV^\prime$ in the sense that 
$U_{\tau(s)}V_s=V^\prime_s$ for all $s\in S $.
\end{itemize}
\end{definition}

The concept of orthogonal dilation is derived from that in Definition~\ref{d:orthogonalu}.

\begin{definition}\label{d:orthogonal} A dilation $(\bK;\Phi;\bV)$ of the map $T$ is called 
\emph{orthogonal} if the representation $(\tau;\Phi)$ is orthogonal, see Definition~\ref{d:serep}, that is,
 whenever $\alpha,\beta\in \Gamma$ are such that $c(\alpha)\neq c(\beta)$ but
$\tau(c(\alpha))=\tau(c(\beta))$ we have 
$\Phi(\alpha)\cK_{\tau(d(\alpha))}\perp \Phi(\beta)\cK_{\tau(d(\beta))}$.
\end{definition}

In order to connect the concept of dilation with that of generalised dilation we restrict the discussion to orthogonal 
dilations, due to technical obstructions as seen during the previous section. 

\begin{remark}\label{r:bd}
A dilation $(\bK;\Phi;\bV)$ of the map $T$ should be viewed as a special type of
generalised dilation $(\bK;\bD;\Phi;\bV)$ but, in view of
Definition~\ref{d:minimalu}, we cannot simply take $\cD_x=\cK_x$. We assume in addition that the dilation
$(\bK;\Phi;\bV)$ is orthogonal.
Firstly, for arbitrary $x\in X$, recall the notation,
see \eqref{e:pede},
\begin{equation*}\label{e:ex}
\cE_x:= \lin \{\Phi(\alpha)V_{d(\alpha)}\cH_{\tau(d(\alpha))}\mid \alpha\in\Gamma\mbox{ such 
that }\tau(c(\alpha))=x\},
\end{equation*}
which is a linear manifold in $\cK_x$ that may be neither closed nor dense. Then
we take 
\begin{equation*}
\cD_x= \cE_x \oplus \cE_x^\perp,
\end{equation*}
which may not be closed but it is dense in $\cK_x$. We show that
$(\tau;\Phi)$ is a $*$-representation of $\Gamma$ in 
the sense of Definition~\ref{d:urep}, on the pair $(\bD;\bK)$, where $\bD=\{\cD_x\}_{x\in X}$.

Indeed,  inspecting the Definition~\ref{d:urep}, the only thing that we have to prove is that for any $\beta\in\Gamma$
we have $\Phi(\beta)\cD_{\tau(d(\beta))}\subseteq \cD_{\tau(c(\beta))}$. To see this, we observe that due to the 
orthogonality assumption, 
$\Phi(\beta)\cE_{\tau(d(\beta))}\subseteq \cE_{\tau(c(\beta))}$ and $\Phi(\beta) \cE_{\tau(d(\beta))}^\perp\subseteq
\cE_{\tau(c(\beta))}^\perp$, see \eqref{e:pebe}.
Thus, the quadruple $(\bK;\bD;\Phi;\bV)$ is a generalised dilation as in Definition~\ref{d:udilation}. 
\end{remark}

In the context of dilations with bounded operators, as in Definition~\ref{d:dilation}, and comparing with generalised 
dilations that involve unbounded operators, the possibility of extracting a minimal dilation has a better answer, when 
compared with that provided in Proposition~\ref{p:mingendil}.

\begin{proposition}\label{p:mindil}
Let $(\bK;\Phi;\bV)$, with $\bK=\{\cK_x\}_{x\in X}$, be an
orthogonal dilation of the map $T\in\fP_{\bH,\tau}(\Gamma)$. 
For arbitrary $x\in X$
let
\begin{equation}\label{e:pedes}
\cE_x:= \lin \{\Phi(\alpha)V_{d(\alpha)}\cH_{\tau(d(\alpha))}\mid \alpha\in\Gamma\mbox{ such 
that }\tau(c(\alpha))=x\},
\end{equation}
provided that the set $\{\alpha\in \Gamma \mid \tau(c(\alpha))=x\}$ is not empty, and let $\cE_x$ be the null space, in 
the opposite case, and let $\cG_x:=\clos \cE_x$. Then, letting $\bG=\{\cG_x\}_{x\in X}$,
$\widetilde \Phi(\alpha)=P_{\cG_{\tau(c(\alpha))}}\Phi|_{\cE_{\tau(d(\alpha))}}$ for all $\alpha\in \Gamma$, and
$\widetilde V_s=P_{\cG_{\tau(s)}}V_s$ for all $s\in S$,
the triple $(\bG;\widetilde\Phi;\widetilde\bV)$ is a minimal orthogonal dilation of the pair $(\tau;T)$.
\end{proposition}

\begin{proof} 
We observe that, for any $\beta\in \Gamma$ we have that 
\begin{equation}\label{e:pebeb}
\Phi(\beta)\cG_{\tau(d(\beta))}\subseteq \cG_{\tau(c(\beta))}\mbox{ and }
\Phi(\beta) \cG_{\tau(d(\beta))}^\perp \subseteq  \cG_{\tau(c(\beta))}^\perp.\end{equation} 
These follow as in the proof of \eqref{e:pebe} by using orthogonality.

These show that, on the one hand, letting 
$\widetilde \Phi(\alpha)=P_{\cG_{\tau(c(\alpha))}}\Phi|_{\cG_{\tau(d(\alpha))}}$ for all $\alpha\in \Gamma$, the pair 
$(\tau;\widetilde \Phi)$ is a $*$-representation of $\Gamma$ on the pair of bundle $\bG$. On the other hand, 
letting $\widetilde V_s=P_{\cG_s}V_s$ for all $s\in S$, these show that,
for any $\alpha\in\Gamma$, $h\in\cH_{\tau(d(\alpha))}$, and $k\in\cH_{\tau(c(\alpha))}$, we have
\begin{align*}
\langle T(\alpha)h,k\rangle_{\cH_{\tau(c(\alpha))}} & =
\langle  \Phi(\alpha) V_{d(\alpha)}h,V_{c(\alpha)}k\rangle_{\cK_{\tau(c(\alpha))}} \\
& = \langle P_{\cG_{\tau(c(\alpha))}} \Phi(\alpha)V_{d(\alpha)}h,V_{c(\alpha)}k\rangle_{\cK_{\tau(c(\alpha))}} \\
& = \langle P_{\cG_{\tau(c(\alpha))}} \Phi(\alpha)\bigl(P_{\cG_{\tau(d(\alpha))}} 
+ P_{\cG_{\tau(d(\alpha))}^\perp}\bigr)V_{d(\alpha)}h, V_{c(\alpha)}k\rangle_{\cK_{\tau(c(\alpha))}} \\
\intertext{which, in view of  \eqref{e:pebeb}, equals}
& = \langle P_{\cG_{\tau(c(\alpha))}} \Phi(\alpha) P_{\cG_{\tau(d(\alpha))}} 
V_{d(\alpha)}h, P_{\cG_{\tau(c(\alpha))}}V_{c(\alpha)}k\rangle_{\cK_{\tau(c(\alpha))}} \\
& = \langle \widetilde \Phi(\alpha)\widetilde V_{d(\alpha)}h,\widetilde V_{c(\alpha)}k\rangle_{\cK_{\tau(c(\alpha))}}.
\end{align*}
Also, for any $x\in X$, a similar reasoning as before shows that
\begin{align*} 
\cG_x & = \clos \lin \{\Phi(\alpha)V_{d(\alpha)}\cH_{\tau(d(\alpha))}\mid \alpha\in\Gamma,\
c(\alpha)\in\tau^{-1}(\{x\}\}\\
& = \clos\lin \{\widetilde\Phi(\alpha)\widetilde V_{d(\alpha)}\cH_{\tau(d(\alpha))}\mid \alpha\in\Gamma,\ 
c(\alpha)\in\tau^{-1}(\{x\}\},
\end{align*}
hence the dilation $(\bG;\widetilde\Phi;\widetilde\bV)$ is minimal.
\end{proof}

The main theorem of this section says that, in general, in order for a map $T\in \fH_{\bH,\tau}(\Gamma)$
to admit a dilation, in addition to the condition of being positive semidefinite,
it should satisfy one more condition of boundedness type, that is rather natural, 
see \cite{BSzNagy} and \cite{GheondeaUgurcan}.

\begin{theorem}\label{t:nagy}
With notation as before, let $T\in\fH_{\bH,\tau}(\Gamma)$. The following assertions are equivalent.
\begin{itemize}
\item[(1)] $T$ satisfies the following conditions.
  \begin{itemize}
  \item[(a)] $T$ is positive semidefinite, in the sense of Definition~\ref{d:psm}.   
  \item[(b)] For any $\alpha\in\Gamma$ there exists $C(\alpha)\geq 0$ such that for any 
finitely supported cross-section $h=\{h_\gamma\}_{\gamma\in\Gamma^{d(\alpha)}}$, where 
$h_\gamma\in \cH_{\tau(d(\gamma))}$ for all $\gamma\in\Gamma^{d(\alpha)}$, we have
 \begin{equation*}
  \sum_{\beta,\gamma\in\Gamma^{d(\alpha)}} \langle 
  T(\beta^*\alpha^*\alpha\gamma)h_\gamma,h_\beta\rangle_{\cH_{\tau(d(\beta))}} \leq
  C(\alpha)\sum_{\beta,\gamma\in\Gamma^{d(\alpha)}} \langle
   T(\beta^*\gamma)h_\gamma,h_\beta\rangle_{\cH_{\tau(d(\beta))}}.
  \end{equation*}
  \end{itemize}
\item[(2)] $T$ has a dilation $(\bK;\Phi;\bV)$.
\end{itemize}

In addition, if assertion (1) holds then a dilation $(\bK;\Phi;\bV)$ of $(\tau;T)$
can always be obtained to be orthogonal, in the sense of Definition~\ref{d:orthogonal}, and
minimal in the sense of Definition~\ref{d:minimalu} and, in this case, it
is unique up to a unitary equivalence, in the sense of Definition~\ref{d:ue}.
\end{theorem}

\begin{proof} $(1)\Ra(2)$. According to the proof of the implication $(1)\Ra(2)$ of
 Theorem~\ref{t:nagyu} we only have to prove that, 
 for any $\alpha\in\Gamma$, the operator $\Phi(\alpha)$ 
defined at \eqref{e:omaf} is bounded and hence it has a unique extension 
to a bounded linear operator, denoted yet by $\Phi(\alpha)\in\cB(\cL_{d(\alpha)},\cL_{c(\alpha)})$.
Indeed, for a fixed $\alpha\in \Gamma$, let $f\in\cG_{d(\alpha)}$ be arbitrary and 
let $g\in\cG_{d(\alpha),0}$ be such that $f=\bT_{d(\alpha)}g$. Then, in view of \eqref{e:omab},
\eqref{e:omast},  and the condition (b), we have
\begin{align}
\langle \Phi(\alpha)f,\Phi(\alpha)f\rangle_{\cG_{c(\alpha)}} & =
 \langle \Phi(\alpha^*)\Phi(\alpha)f,f\rangle_{\cG_{d(\alpha)}} = 
\langle \Phi(\alpha^*\alpha) f,f\rangle_{\cG_{d(\alpha)}} \nonumber \\
& = \sum_{\beta,\gamma\in\Gamma^{d(\alpha)}} \langle 
  T(\beta^*\alpha^*\alpha\gamma)g_\gamma,g_\beta\rangle_{\cH_{\tau(d(\beta))}} \nonumber\\
  & \leq
  C(\alpha)\sum_{\beta,\gamma\in\Gamma^{d(\alpha)}} \langle
   T(\beta^*\gamma)g_\gamma,g_\beta\rangle_{\cH_{\tau(d(\beta))}}\label{e:calfa} \\
   & = C(\alpha) \langle f,f\rangle_{\cG_{d(\alpha)}},\nonumber
\end{align}
which proves the claim. 

Finally, we lift the operator $\Phi(\alpha)$ to a bounded linear operator from $\cK_{\tau(d(\alpha))}$
to $\cK_{\tau(c(\alpha))}$, yet denoted by $\Phi(\alpha)$, in the following way: with respect to the decompositions
\begin{equation}\label{e:cekata}
\cK_{\tau(d(\alpha))}=\bigoplus_{t\in S ,\ \tau(t)=\tau(d(\alpha))}\cL_t,\quad
\cK_{\tau(c(\alpha))}=\bigoplus_{t\in S ,\ \tau(t)=\tau(c(\alpha))}\cL_t,
\end{equation}
letting
\begin{equation}
\Phi(\alpha)|_{\cL_t}=0,\mbox{ for all }t\in S \setminus\{d(\alpha)\},\ \tau(t)=\tau(d(\alpha)).
\end{equation} 
Then, it follows that \eqref{e:omab} and \eqref{e:omast} hold for the
bounded operators $\Phi(\alpha)$ and $\Phi(\beta)$ as well, hence $\Phi$ is a 
$*$-representation of $\Gamma$ to $\Gamma_\bK$.

$(2)\Ra(1)$. Let $(\bK;\Phi;\bV)$ be a dilation of $(\tau;T)$. It follows from Proposition~\ref{p:psmu}
that $T$ is positive semidefinite.
Also, for any $\alpha\in \Gamma$ and
   for any 
  finitely supported cross-section $h=\{h_\gamma\}_{\gamma\in\Gamma^{d(\alpha)}}$, where 
  $h_\gamma\in \cH_{\tau(d(\gamma))}$ for all $\gamma\in\Gamma^{d(\alpha)}$, we have
  \begin{align*}
  \sum_{\beta,\gamma\in\Gamma^{d(\alpha)}} \langle 
  T(\beta^*\alpha^*\alpha\gamma)h_\gamma,h_\beta\rangle_{\cH_{\tau(d(\beta))}} & =
  \sum_{\beta,\gamma\in\Gamma^{d(\alpha)}} \langle 
  V_{d(\beta)}^*\Phi(\beta^*\alpha^*\alpha\gamma)V_{d(\gamma)}
h_\gamma,h_\beta\rangle_{\cH_{\tau(d(\beta))}} \\
& =
  \sum_{\beta,\gamma\in\Gamma^{d(\alpha)}} \langle 
  \Phi(\alpha)\Phi(\gamma)V_{d(\gamma)}
h_\gamma,\Phi(\alpha)\Phi(\beta)V_{d(\beta)}h_\beta\rangle_{\cK_{\tau(c(\alpha))}} \\
& = \bigl\langle \Phi(\alpha)\!\!\!\!
\sum_{\gamma\in\Gamma^{d(\alpha)}}\!\!\!\! 
\Phi(\gamma)V_{d(\gamma)}h_\gamma,
\Phi(\alpha)\!\!\!\sum_{\beta\in\Gamma^{d(\alpha)}}\!\!\!\! 
\Phi(\beta)V_{d(\beta)}h_\beta \bigr\rangle_{\cK_{\tau(c(\alpha))}} 
\\ & \leq \|\Phi(\alpha)\|^2 
 \bigl\langle\!\! \sum_{\gamma\in\Gamma^{d(\alpha)} }\!\! \Phi(\gamma)V_{d(\gamma)}h_\gamma,\!\!
\sum_{\beta\in\Gamma^{d(\gamma)}}\!\! \Phi(\beta)V_{d(\beta)}h_\beta \bigr\rangle_{\cK_{\tau(d(\alpha))}} \\
& = \|\Phi(\alpha)\|^2 \sum_{\beta,\gamma\in\Gamma^{d(\alpha)}} \langle 
  V_{d(\beta)}^*\Phi(\beta^*\gamma)V_{d(\gamma)}
h_\gamma,h_\beta\rangle_{\cH_{\tau(d(\beta))}} \\
& =   \|\Phi(\alpha)\|^2 \sum_{\beta,\gamma\in\Gamma^{d(\alpha)}} \langle
   T(\beta^*\gamma)h_\gamma,h_\beta\rangle_{\cH_{\tau(d(\beta))}}.
  \end{align*}
This proves the boundedness condition (b).
\end{proof}

Also, Remark~\ref{r:cohu} holds for the
orthogonal minimal dilation $(\bK;\Phi;\bV)$ 
constructed during the proof of Theorem~\ref{t:nagy}, namely that
\begin{equation*}
\sum_{s\in\tau^{-1}(\{x\})} \Phi(\epsilon_s) =I_{\cK_x},\quad x\in X.
\end{equation*}
Similarly to Remark~\ref{r:rkhs}, the construction of the spaces $\cK_x$ in the proof of Theorem~\ref{t:nagy} is
performed in such a way that it is a reproducing kernel Hilbert space on each fibre
$\Gamma^s$, for $s\in S$ such that 
$x=\tau(s)$. Similarly, Corollary~\ref{c:embedu} can be translated word-for-word to the current setting of
dilations.

There are special cases of $*$-semigroupoids $\Gamma$ for which the boundedness condition
(b) in Theorem~\ref{t:nagy} holds automatically. Clearly this is the case when $\Gamma$ 
is a groupoid, but there is a more general case as well, that of an inverse semigroupoid with unit.

\begin{corollary}\label{c:invsem} 
With notation as before, let $T\in\fH_{\bH,\tau}(\Gamma)$ and, assume, in addition, that $\Gamma$ is an inverse 
semigroupoid with unit. Then the following assertions are equivalent.
\begin{itemize}
\item[(1)]  $T$ is positive semidefinite, in the sense of Definition~\ref{d:psm}.   
\item[(2)] $T$ has a dilation $(\bK;\Phi;\bV)$.
\end{itemize}
In addition, if $T$ is positive semidefinite then we can always obtain an orthogonal minimal dilation 
$(\bK;\Phi;\bV)$ such that $\Phi(\alpha)$ is a partial isometry, for all $\alpha\in\Gamma$.
\end{corollary}

\begin{proof} Assume that $T$ is positive semidefinite. With notation as in the proof of 
Theorem~\ref{t:nagyu}, for an arbitrary element $\alpha\in\Gamma$, we consider the operator 
$\Phi(\alpha)\colon \cG_{d(\alpha)}\ra \cG_{c(\alpha)}$  as in \eqref{e:omaf} and we want to prove
that it is bounded.

Indeed, in Step 3 of the proof of Theorem~\ref{t:nagyu}, see \eqref{e:omast},
it is shown that the operator 
$\Phi(\alpha^*)\colon \cG_{c(\alpha)}\ra\cG_{d(\alpha)}$ has the property 
$\Phi(\alpha^*)\subseteq \Phi(\alpha)^*$, when viewing $\Phi(\alpha)$ as a 
densely defined linear operator from the Hilbert space $\cL_{d(\alpha)}$ to the Hilbert space 
$\cL_{c(\alpha)}$. Also, as a consequence of \eqref{e:omab}, \eqref{e:omast}, 
and of the fact that $\Gamma$ is an 
inverse semigroupoid, hence $\alpha^*\alpha\alpha^*=\alpha^*$, 
it follows that, for any $h\in\cG_{c(\alpha)}$ we have
\begin{equation*}
\Phi(\alpha)^*\Phi(\alpha)\Phi(\alpha)^*h=\Phi(\alpha^*)\Phi(\alpha)
\Phi(\alpha^*)h=\Phi(\alpha^*\alpha\alpha^*)h=\Phi(\alpha^*)h,
\end{equation*}
hence, for any $k\in\cG_{d(\alpha)}$, letting $h=\Phi(\alpha)k$, we have
\begin{equation*}
\Phi(\alpha)^*\Phi(\alpha)\Phi(\alpha)^*\Phi(\alpha)k=\Phi(\alpha^*)\Phi(\alpha)k,
\end{equation*}
that is, the operator $\Phi(\alpha)^*\Phi(\alpha)\colon \cG_{d(\alpha)}\ra\cG_{d(\alpha)}$ is a 
projection. This implies that $\cG_{d(\alpha)}$ has a decomposition
\begin{equation*}
\cG_{d(\alpha)}=\ker(\Phi(\alpha)^*\Phi(\alpha))\dot+ \ran(\Phi(\alpha)^*\Phi(\alpha)).
\end{equation*}
In addition, 
\begin{equation*}
\ker(\Phi(\alpha)^*\Phi(\alpha))\perp \ran(\Phi(\alpha)^*\Phi(\alpha)),
\end{equation*}
which follows easily from the fact that $\Phi(\alpha)^*\Phi(\alpha)h=\Phi(\alpha^*\alpha)h$ holds for
all $h\in\cG_{d(\alpha)}$, see \eqref{e:omast} and \eqref{e:omab}. 
Also from here we get that $\ker(\Phi(\alpha)^*\Phi(\alpha))=\ker(\Phi(\alpha))$. Then, letting
$h\in\cG_{d(\alpha)}$ be arbitrary, hence $h=h_0+h_1$, where $h_0\in\ker(\Phi(\alpha))$ and
$h_1\in\ran(\Phi(\alpha)^*\Phi(\alpha))$ hence,
\begin{align*}
\langle \Phi(\alpha)h,\Phi(\alpha)h\rangle_{\cK_{\tau(c(\alpha))}} &  = 
\langle \Phi(\alpha)h_1,\Phi(\alpha)h_1\rangle_{\cK_{\tau(d(\alpha))}} \\
& = \langle \Phi(\alpha)^*\Phi(\alpha)h_1, h_1\rangle_{\cK_{\tau(d(\alpha))}} \\
& =\langle h_1,h_1\rangle_{\cH_{\tau(d(\alpha))}}\leq \langle h,h\rangle_{\cK_{\tau(d(\alpha))}},
\end{align*}
where we have taken into account that $\Phi(\alpha)^*\Phi(\alpha)$ is a projection and acts like
dentity operator on its range. Thus, we have proven that the operator $\Phi(\alpha)$ is bounded
and hence it can be extended to a linear bounded operator $\cL_{d(\alpha)}\ra \cL_{c(\alpha)}$. Then,
proceeding as in \eqref{e:cekat} and \eqref{e:phia}, we lift the operator $\Phi(\alpha)$ to a bounded 
linear operator $\cK_{\tau(d(\alpha))}\ra \cK_{\tau(c(\alpha))}$, and it is easy to see that it is a partial 
isometry.
\end{proof}

\section{Linearisations of Positive Semidefinite Maps}

In this section we pass from semigroupoids to algebroids and hence perform a linearisation of the concepts that
have been studied in the previous sections. Our goal is to obtain dilation theorems of Stinespring type. Here the main 
concept is that of a $*$-algebroid. As in the 
previous sections, we will do this in two steps, firstly by dilation with possibly unbounded operators and then by
bounded operators that involve
the boundedness condition on the representation of the algebroid. We show that in the case of $B^*$-algebroids, the boundedness condition is automatically satisfied and, in addition,
we relate the positive semidefiniteness to complete positivity, in the spirit of the original version of 
Stinespring's Theorem \cite{Stinespring}. Finally, we single out the concept of a $C^*$-algebroid and 
tackle some natural questions on it.
 
\subsection{Positive Semidefinite Maps on $*$-Algebroids} 
 \begin{definition} An \emph{algebroid}, see \cite{Pradines} and \cite{Mosa}, $\cA$ over the field $\CC$ 
 is a semigroupoid with the following additional properties.
 \begin{itemize}
 \item[(a1)] For each $s,t\in S_\cA$, the fibre $\cA_s^t$ is a complex vector space.
 \item[(a2)] For each $s,t\in S_\cA$ the following distributivity properties hold.
 \begin{itemize}
 \item[(i)] $a (\alpha b+\beta c)=\alpha ab+\beta ac$ for all $b,c\in\cA_s^t$,
 $a\in \cA_t$, and $\alpha,\beta\in\CC$.
 \item[(ii)] $(\alpha b+\beta c)d=\alpha bd+\beta cd$ for all $b,c\in\cA_s^t$, $d\in \cA^s$, and $\alpha,\beta\in\CC$.
 \end{itemize}
 \end{itemize}
 The algebroid $\cA$ has a \emph{unit} if the underlying semigroupoid has a unit.
 \end{definition}
 
 \begin{example}
 With notation as in Example~\ref{ex:semes}, letting
 \begin{equation*}
 \cA=\bigsqcup_{s,t\in S} \cL(\cD_s,\cD_t),
 \end{equation*}
 the semigroupoid $\cA$ has a natural structure of algebroid with unit.
 \end{example}
   
Let us observe that, if $\cA$ is an algebroid, then for each $s\in S_\cA$, the fibre $\cA_s^s$ is a complex algebra, 
called the \emph{isotropy algebra} of $\cA$ at $s$. If the algebroid has a unit $e$ then each isotropy algebra 
$\cA_s^s$ has a unit $e_s$. 

 \begin{definition} Given $\cA$ an algebroid, 
 an \emph{involution} on $\cA$ is a map $*\colon \cA\ra\cA$ that turns
 the underlying semigroupoid into a $*$-semigroupoid and subject to the following property
 \begin{itemize}
 \item[(ai)] For any $s,t\in S_\cA$, $a,b\in \cA_s^t$, $\alpha,\beta\in\CC$, we have $(\alpha a+\beta b)^*=\ol \alpha a^*+\ol \beta b^*$.
 \end{itemize}
 \end{definition}
 
 \begin{example}\label{ex:bh} 
 With notation and assumptions as in Example~\ref{ex:semuh}, letting
 \begin{equation}
  \cL_\bD^*(\bH):=\bigsqcup_{s,t\in S} \cL^*(\cD_s,\cD_t),
 \end{equation}
 we have a $*$-algebroid with unit.
 \end{example}

\begin{definition}\label{d:repa} Let $\cA$ and $\cB$ be two algebroids. A  pair $(\phi;\Phi)$ is an \emph{algebroid
morphism} from $\cA$ to $\cB$ if the following conditions hold.
\begin{itemize}
\item[(am1)] The pair $(\phi;\Phi)$ is a semigroupoid morphism from $\cA$ to $\cB$, in the sense of 
Definition~\ref{d:sm}.
\item[(am2)] For every $s,t\in S_\cA$ the map $\Phi|_{\cA_s^t}\colon \cA_s^t\ra \cB_{\phi(s)}^{\phi(t)}$ is linear.
\end{itemize}
If the two algebroids $\cA$ and $\cB$ are $*$-algebroids, the pair $(\phi;\Phi)$ is a \emph{$*$-algebroid morphism},
or  a \emph{$*$-morphism} if, in addition to (am1) and (am2) it is Hermitian
\begin{itemize}
\item[(am3)] $\Phi(a^*)=\Phi(a)^*$ for all $a\in \cA$.
\end{itemize}
\end{definition}

\begin{definition}\label{d:repu} 
Let $\cA$ be a $*$-algebroid, $\bD=\{\cD_x\}_{x\in X}$ a bundle of vector spaces,  and 
$\bH=\{\cH_x\}_{x\in X}$ a bundle of Hilbert spaces, where $\cD_x$ is a dense subspace of the Hilbert space
$\cH_x$ for all $x\in X$. An \emph{unbounded $*$-representation} of $\cA$ on the pair of bundles $(\bD;\bH)$ 
is a pair of maps $(\phi;\Phi)$, where $\phi\colon S_\cA\ra X$ and $\Phi\colon \cA\ra\cL^*_\bD(\bH)$,
subject to the following conditions.
\begin{itemize}
\item[(uar1)] For any $s,t\in S_\cA$  and $a\in \cA_s^t$, $\Phi(a)$ is a linear operator such that
$\cD_{\phi(s)}\subseteq \dom(\Phi(a))$ and 
$\Phi(a)\cD_{\phi(s)}\subseteq\cD_{\phi(t)}$.
\item[(uar2)] For any $s\in S_\cA$, any $a\in \cA_s^u$, and any $b\in\cA^s_v$ we have 
$\Phi(ab)|_{\cD_{\phi(v)}}=\Phi(a)\Phi(b)|_{\cD_{\phi(v)}}$.
\item[(uar3)] For any $s,t\in S_\cA$ and any $a,b\in\cA_s^t$ we have 
$\Phi(a+b)|_{\cD_{\phi(s)}}=\Phi(a)|_{\cD_{\phi(s)}}+\Phi(b)|_{\cD_{\phi(s)}}$.
\item[(uar4)] For any $a\in \cA_s^t$ we have
$\Phi(a)^*\cD_{\phi(t)}\subseteq\cD_{\phi(s)}$ and $\Phi(a^*)|_{\cD_{\phi(t)}}
=\Phi(a)^*|_{\cD_{\phi(t)}}$.
\end{itemize}

We call the unbounded $*$-representation $(\phi;\Phi)$ \emph{orthogonal} if the following property holds
\begin{itemize}
\item[(uar5)] For any $a\in \cA^s_u$ and $b\in \cA^t_v$ with $s\neq t$ and such that
$\phi(s)=\phi(t)$
 it follows that $\Phi(a)\cD_{\phi(u)}\perp 
\Phi(\beta)\cD_{\phi(v)}$.
\end{itemize}
\end{definition}

\begin{definition}\label{d:uda} 
Let $T\colon \cA\ra\cB(\bH)$ be a $\tau$-coherent and Hermitian map, for some bundle of 
Hilbert spaces $\bH=\{\cH_x\}_{x\in X}$ and some aggregation map $\tau\colon S_\cA\ra X$.
A \emph{generalised dilation} of $T$ on $(\bK;\bD)$ is a 
quadruple $(\bK;\bD;\Phi;\bV)$ subject to the following conditions.
\begin{itemize}
\item[(ad1)] $\bK=\{\cK_x\}_{x\in X}$ is a bundle of Hilbert spaces, $\bD=\{\cD_x\}_{x\in X}$ is a bundle of 
vector spaces such that $\cD_x$ is a dense subspace of $\cK_x$ for all $x\in X$.
\item[(ad2)] $\bV=\{V_s\}_{s\in S}$ is a bundle of operators, with $V_s\in\cB(\cH_{\tau(s)},\cK_{\tau(s)})$ such 
that $\ran(V_s)\subseteq \cD_{\tau(s)}$ for each $s\in S_\cA$.
\item[(ad3)] $(\Phi;\tau)$ is an unbounded 
$*$-representation of $\cA$ on $(\bK;\bD)$, in the sense of Definition~\ref{d:repu}, such that
\begin{equation*}
T(a)=V_{\tau(t)}^* \Phi(a) V_{\tau(s)},\quad s,t\in S_\cA,\ a\in \cA_s^t.
\end{equation*}
\end{itemize}

In addition, a generalised dilation 
$(\bK;\bD;\bV;\Phi)$ of $T$ is called \emph{orthogonal} in the sense of Definition~\ref{d:repu},
more precisely, for any $a\in \cA_u^s$ and $b\in \cA_v^t$ with $s\neq t$ and $\tau(s)=\tau(t)$, 
we have $\Phi(a)\cD_{\tau(u)}\perp \Phi(b)\cD_{\tau(v)}$. 

The generalised dilation $(\bK;\bD;\Phi;\bV)$ is \emph{minimal} in the sense of Definition~\ref{d:minimalu},
more precisely, for each $x\in X$ we have
\begin{equation*}
\cK_x=\clos\lin \{\Phi(a)V_s\cH_{\tau(s)}\mid s,t\in S_\cA,\ a\in\cA_s^t,\ \tau(t)=x\}.
\end{equation*}

Two generalised dilations $(\bK;\bD;\bV;\Phi)$ and $(\bK^\prime;\bD^\prime;\bV^\prime;\Phi^\prime)$ 
of $T$ are \emph{unitarily equivalent} if there exists a bundle $\bU=\{U_x\}_{x\in X}$ subject to the following 
conditions.
\begin{itemize}
\item[(au1)] For each $x\in X$ the operator $U_x\colon \cK_x\ra\cK_x^\prime$ is unitary.
\item[(au2)] For each $x\in X$, $U_x\cD_x=\cD_x^\prime$.
\item[(au3)] The bundle $\bU$ intertwines the $*$-representations $\Phi$ and $\Phi^\prime$, that is,
\begin{equation*}
U_{\tau(t)}\Phi(a)|_{\cD_{\tau(s)}}=\Phi^\prime(a) U_{\tau(s)}|_{\cD_{\tau(s)}},\quad s,t\in S_\cA,\ a\in\cA_s^t.
\end{equation*}
\item[(au4)] The bundle $\bU$ maps the bundle $\bV$ to $\bV^\prime$, that is, $U_{\tau(s)}V_s=V_s^\prime$.
\end{itemize}
\end{definition}

\begin{theorem}\label{t:stinespringu} Let $\cA$ be a $*$-algebroid with unit, 
$\bH=\{\cH_x\}_{x\in X}$ a bundle of Hilbert
spaces, $\tau\colon S_\cA\ra X$ an aggregation map, and $T\colon \cA\ra \cB(\bH)$ a  Hermitian and
$\tau$-coherent map. The following assertions are equivalent.
\begin{itemize}
\item[(1)] $T$ is positive semidefinite in the sense of Definition~\ref{d:psm}.
\item[(3)] There exists a generalised dilation $(\bK;\bD;\bV;\Phi)$ of $T$.
\end{itemize}

In addition, if exists, the generalised dilation $(\bK;\bD;\bV;\Phi)$ can be chosen orthogonal and minimal.
Moreover, any two generalised dilations $(\bK;\bD;\bV;\Phi)$ and 
$(\bK^\prime;\bD^\prime;\bV^\prime;\Phi^\prime)$, that are minimal and orthogonal, are unitarily equivalent.
\end{theorem}

\begin{proof} This is a consequence of Theorem~\ref{t:nagyu}, with the observation that, during the construction of
the $*$-representation $\Phi$, from \eqref{e:sulam} it follows that the linearity of $T$ implies the linearity of $\Phi$.
\end{proof}

Given a $*$-algebroid $\cA$, a bundle of Hilbert spaces $\bH=\{\cH_x\}_{x\in X}$, an aggregation map $\tau\colon 
S_\cA\ra X$, and a Hermitian and $\tau$-coherent map $T\colon \cA\ra \cB(\bH)$, from the previous
definitions it is clear what a \emph{dilation} of $T$ should be.
For example, with notation as in Definition~\ref{d:uda}, this 
means that $\cD_x=\cK_x$ for all $x\in X$ and hence $\Phi(a)$ is a bounded linear operator 
from $\cK_{\tau(s)}$ to $\cK_{\tau(t)}$, where $a\in \cA_s^t$. For this reason, a dilation of $T$ is denoted by a triple 
$(\bK;\bV;\Phi)$. Then, as a consequence of Theorem~\ref{t:stinespringu} and of Theorem~\ref{t:nagy}, we
have the following result.

\begin{corollary}\label{c:stinespringa} Let $\cA$ be a $*$-algebroid with unit, 
$\bH=\{\cH_x\}_{x\in X}$ a bundle of Hilbert
spaces, $\tau\colon S_\cA\ra X$ an aggregation map, and $T\colon \cA\ra \cB(\bH)$ a Hermitian and
$\tau$-coherent map. The following assertions are equivalent.
\begin{itemize}
\item[(1)] $T$ satisfies the following conditions.
\begin{itemize}
\item[(a)] $T$ is positive semidefinite in the sense of Definition~\ref{d:psm}.
\item[(b)] For any $a\in\cA_s$, for some $s\in S_\cA$, there exists $C(a)\geq 0$ such that, for any $n\in\NN$, any
$b_j\in \cA_{s_j}^s$, for some $s_1,\ldots,s_n\in S_\cA$, and all $j=1,\ldots,n$, and all cross-sections $h_1,\ldots,h_j$ with $h_j\in \cH_{\tau(s_j)}$ for all $j=1,\ldots,n$, we have
\begin{equation*}
\sum_{i,j=1}^n \langle T(b_j^*a^*ab_i)h_i,h_j\rangle_{\cH_{\tau(s_j)} }\leq
C(a) \sum_{i,j=1}^n \langle T(b_j^*b_i)h_i,h_j\rangle_{\cH_{\tau(s_j)}}.
\end{equation*}
\end{itemize}
\item[(2)] There exists a dilation $(\bK;\bV;\Phi)$ of $T$.
\end{itemize}

In addition, if exists, the generalised dilation $(\bK;\bD;\bV;\Phi)$ can be chosen orthogonal and minimal.
Moreover, any two generalised dilations $(\bK;\bD;\bV;\Phi)$ and 
$(\bK^\prime;\bD^\prime;\bV^\prime;\Phi^\prime)$, that are minimal and orthogonal, are unitarily equivalent.
\end{corollary}

\subsection{Completely Positive Maps on $B^*$-Algebroids}

\begin{definition}\label{d:posa} 
Let $\cA$ be a $*$-algebroid. For each $s\in S_\cA$ we consider the convex cone $\cA_s^{s,+}$ 
in $\cA_s^s$ 
generated by all elements $x^*x$, where $x\in \cA_s$, that is, $x\in \cA_s^t$ for some $t\in S_\cA$. 
An element $a\in \cA_s^s$ is called \emph{positive} if $a\in\cA_s^{s,+}$
and, for this, we simply write $a\geq 0$.
More precisely,
given $a\in\cA$, we have $a\geq 0$ if and only if there exist $t_1,\ldots,t_n\in S_\cA$ 
and $x_1\in\cA_s^{t_1},\ldots,x_n\in\cA_s^{t_n}$ such that $a=x_1^*x_1+\cdots+x_n^*x_n$.
\end{definition}

\begin{remark}\label{r:beha} Let $\bH=\{\cH_s\}_{s\in S}$ be a bundle of Hilbert spaces and let
 \begin{equation*}
\cB(\bH)=\bigsqcup_{s,t\in S} \cB(\cH_s,\cH_t).
\end{equation*}
Then, see Example~\ref{ex:segah}, $\cB(\bH)$ is a
$*$-algebroid with unit. It is well known, e.g. see \cite{Conway} or \cite{Arveson}, that, in this case, 
for any $s\in S$ and any $A\in\cB(\bH)_s^s=\cB(\cH_s,\cH_s)$, the following assertions are equivalent.
\begin{itemize}
\item[(i)] $A$ is positive in the sense of Definition~\ref{d:posa}.
\item[(ii)] $A=X^*X$ for some $X\in \cB(\bH)_s^t=\cB(\cH_s,\cH_t)$.
\item[(iii)] $A=X^*X$ for some $X\in \cB(\bH)_s^s=\cB(\cH_s,\cH_s)$.
 \end{itemize}
\end{remark}

\begin{definition} Given $\cA$ an algebroid, a \emph{system of submultiplicative norms} on $\cA$ is a 
collection $\{\|\cdot\|_{s,t}\mid s,t\in S_\cA\}$ subject to the following conditions.
 \begin{itemize}
 \item[(an1)] For each $s,t\in S_\cA$, $\|\cdot\|_{s,t}$ is a norm on the vector space $\cA_s^t$.
 \item[(an2)] For each $s,t,u\in S_\cA$ we have
 \begin{equation*}\|ab\|_{s,t}\leq \|a\|_{u,t}\, \|b\|_{s,u},\quad a\in\cA_u^t,\ b\in\cA_s^u.
 \end{equation*}
\end{itemize}
An algebroid $\cA$ endowed with a system of submultiplicative norms $\{\|\cdot\|_{s,t}\mid s,t\in S_\cA\}$ 
on it is called a \emph{normed algebroid}. A normed algebroid $\cA$ such that, for all $s,t\in S_\cA$ the norm
$\|\cdot\|_{s,t}$ is complete, is called a \emph{Banach algebroid}. Note that, in this case, 
the isotropy algebra $\cA_s^s$ is a Banach algebra, for all $s\in S$.
\end{definition}

\begin{example} With notation as in Example~\ref{ex:semes}, assume that $\bD=\{\cD_s\}_{s\in S}$ is a bundle of 
normed spaces. Then, letting 
\begin{equation*}\cB(\bD):=\bigsqcup_{s,t\in S}\cB(\cD_s,\cD_t),\end{equation*} 
where $\cB(\cD_s,\cD_t)$ denotes the vector space of all 
bounded linear operators $\cD_s\ra \cD_t$, 
$\cB(\bD)$ becomes in a natural fashion a normed algebroid, where for each $s,t\in S$,
$\|\cdot\|_{s,t}$ is the operator norm on $\cB(\cD_s,\cD_t)$. If $\cD_s$ is a Banach space for all $s\in S$, then
$\cB(\bD)$ is a Banach algebroid with unit.
\end{example}

\begin{definition}\label{d:csa}
If $\cA$ is a normed algebroid with an involution $*$ that is isometric, that is, 
$\|a^*\|_{t,s}=\|a\|_{s,t}$, for all $s,t\in S$ and 
$a\in\cA_s^t$, we call $\cA$ an \emph{involutive algebroid}.
 
If $\cA$ is a Banach algebroid with an isometric involution $*$, 
we call $\cA$ a \emph{$B^*$-algebroid}. Note that, in this case, for each $s\in S_\cA$, the isotropy
algebra $\cA_s^s$ is a $B^*$-algebra, in the sense that it is a Banach algebra with isometric involution. The
definition of positive elements is as in Definition~\ref{d:posa}.
\end{definition}

\begin{example} With notation and assumptions as in Remark~\ref{r:beha}, see also Example~\ref{ex:segah}, 
$\cB(\bH)$ is a $B^*$-algebroid with unit.
\end{example}
 
\begin{definition} \label{d:smorph}
If $\cA$ and $\cB$ are $*$-algebroids, the algebroid morphism $(\Phi;\phi)$ from $\cA$ to $\cB$ is called a
\emph{$*$-morphism} if $\Phi(a^*)=\Phi(a)^*$ for all $a\in \cA$.

An algebroid morphism from an algebroid $\cA$ to the algebroid $\cB(\bH)$, for some bundle $\bH$ of Hilbert 
spaces, see Example~\ref{ex:bh}, is called a \emph{representation} of $\cA$ on $\bH$. 
In case $\cA$ is a $*$-algebroid, an algebroid $*$-morphism from $\cA$ to $\cB(\bH)$ is called a 
\emph{$*$-representation} of $\cA$ on $\bH$. 
\end{definition}

\begin{definition}\label{d:ampl}
If $\cA$ is an algebroid and $n\in \NN$, we define the 
\emph{$n$-fold amplification}  ${}^n\!\!\cA$ in the following way. The 
symbol set $S_{{}^n\!\!\cA}=S_\cA^n$ consists in all possible $n$-tuples $\bs=(s_1,\ldots,s_n)$ with
$s_j\in S_\cA$ for all $j=1,\ldots,n$. For each $\bs,\bt\in S_{{}^n\!\!\cA}$ the fibre 
${}^n\!\!\cA_\bs^\bt$ consists in all
$n\times n$ matrices $A=[a_{i,j}]_{i,j=1}^n$ with entries $a_{i,j}\in \cA_{s_j}^{t_i}$ for all $i,j=1,\ldots,n$. 
Clearly,  ${}^n\!\!\cA$ is a vector space. If 
$A\in {}^n\!\!\cA_\bv^\bt$ and $B\in {}^n\!\!\cA^\bv_\bs$ then the matrix multiplication $AB$ is possible and
$AB\in {}^n\!\!\cA_\bs^\bt$. 

If $\cA$ is a $*$-algebroid then, for any $A\in {}^n\!\!\cA$, by matrix transposition and elementwise 
involution, we define $A^*\in {}^n\!\!\cA$.
\end{definition}

\begin{remark}
If $\cA$ is an algebroid then, for each $n\in\NN$, the \emph{$n$-th amplification} ${}^n\!\!\cA$, with the algebraic 
structure described in the previous definition, is an algebroid.

 If $\cA$ is a $*$-algebroid then, for each $n\in\NN$, ${}^n\!\!\cA$ is a $*$-algebroid.
\end{remark}

\begin{definition}\label{d:np} 
If $\cA$ and $\cB$ are two algebroids over $\CC$, let $\tau\colon S_\cA\ra S_\cB$ be an aggregation map, and let 
$T\colon \cA\ra\cB$ be a map that is $\tau$-coherent. 

For any $n\in\NN$, we define the $n$-th amplification $T^{(n)}\colon {}^n\!\!\cA\ra {}^n\cB$ in the 
following way. For any $\bs,\bt\in S_\cA^n$, let $A=[a_{i,j}]_{i,j=1}^n\in {}^n\!\!\cA_\bs^\bt$, hence 
$a_{i,j}\in \cA_{s_j}^{t_i}$ for all $i,j=1,\ldots,n$, and we define
\begin{equation}\label{e:tena}
T^{(n)}(A)=[T(a_{i,j})]_{i,j=1}^n\in {}^n\cB_{\tau(\bs)}^{\tau(\bt)},
\end{equation}
where the aggregation map $\tau$ is lifted to $S_\cA^n\ra S_\cB^n$ in the natural fashion: 
$\tau(s_1,\ldots,s_n)=(\tau(s_1),\ldots,\tau(s_n))$ for all $\bs=(s_1,\ldots, s_n)\in S_\cA^n$.
\end{definition}

\begin{definition}\label{d:cp} Assume that $\cA$ and $\cB$ are $*$-algebroids and the rest of notation as in the 
previous definition.
The map $T\colon \cA\ra\cB$ is called \emph{$n$-positive} if $T$ is Hermitian, in the sense that $T(a^*)=T(a)^*$
for all $a\in \cA$, and
maps positive elements in ${}^n\!\!\cA$ to positive elements
in ${}^n\cB$, in the sense of Definition~\ref{d:posa}. More precisely, for any 
$\bs\in S_\cA^n$ and any $A\in {}^n\!\!\cA_\bs^\bs$ such that $A\geq 0$, it 
follows that $T(A)$ is positive in ${}^n\cB_{\tau(s)}^{\tau(s)}$. 

The map $T$ is called \emph{completely positive} if it is $n$-positive for all $n\in\NN$.
\end{definition}

\begin{remark} Assume the notation as in Definition~\ref{d:cp}.
If $T$ is $n$-positive for some $n\in\NN$ then it is 
$k$-positive for all $k\in \{1,2,\ldots,n-1\}$.
\end{remark}

The main result of this section is a Stinespring type theorem. Following a generalisation obtained by W.B.~Arveson in
\cite{Arveson1}, the dilation is obtained at the general level of $B^*$-algebroids with unit.

\begin{theorem}\label{t:stinespring} Let $\cA$ be a $B^*$-algebroid with unit, 
$\bH=\{\cH_x\}_{x\in X}$ a bundle of Hilbert
spaces, $\tau\colon S_\cA\ra X$ an aggregation map, and $T\colon \cA\ra \cB(\bH)$ a Hermitian and
$\tau$-coherent map. The following assertions are equivalent.
\begin{itemize}
\item[(1)] $T$ is completely positive in the sense of Definition~\ref{d:cp}.
\item[(2)] $T$ is positive semidefinite in the sense of Definition~\ref{d:psm}.
\item[(3)] There exists a dilation $(\bK;\bV;\Phi)$ of $T$.
\end{itemize}

In addition, if exists, the dilation $(\bK;\bV;\Phi)$ can be chosen orthogonal and minimal.
Moreover, any two dilations $(\bK;\bV;\Phi)$ and $(\bK^\prime;\bV^\prime;\Phi^\prime)$ of $T$, 
that are minimal and orthogonal, are unitarily equivalent.
\end{theorem}

\begin{proof} (1)$\Ra$(2). For arbitrary $n\in \NN$, let $a_1,\ldots, a_n\in \cA$ with $a_i\in \cA_{s_i}^s$, 
for some $s,s_1,\ldots,s_n\in S_\cA$ and all $i=1,\ldots,n$, and let $h_i \in \cH_{\tau(s_i)}$ for $i=1,\ldots,n$. 
We consider the $n\times n$ matrix $A\in {}^n\!\!\cA_\bs^{\ol\bs}$, where $\bs=(s_1,s_2,\ldots,s_n)$ and 
$\ol\bs=(s,s,\ldots,s)$, defined by
\begin{equation*} A=\begin{bmatrix} a_1 & a_2 & \cdots & a_n \\ 0 & 0 & \cdots & 0 \\ 
\vdots & \vdots & \cdots & \vdots \\ 0 & 0 & \cdots & 0
\end{bmatrix}
\end{equation*}
and note that
\begin{equation*}
A^*A = \begin{bmatrix} a_1^*a_1 & a_1^*a_2 & \cdots & a_1^*a_n \\ a_2^*a_1 & a_2^*a_2 & \cdots & a_2^* a_n \\
\vdots & \vdots & \cdots & \vdots \\ a_n^*a_1 & a_n^*a_2 & \cdots & a_n^*a_n
\end{bmatrix} \in {}^n\!\!\cA_\bs^\bs.
\end{equation*}
Since $A^*A\geq 0$ and $T$ is completely positive, it follows that 
$T^{(n)}(A^*A)\geq 0$ in $\cB(\cH_{\tau(s_1)}\oplus \cdots\oplus \cH_{\tau(s_n)})$ hence
\begin{equation*}
\sum_{i,j=1}^n \langle T(a_i^*a_j)h_j,h_i\rangle_{\cH_{\tau(s_i)}} = \langle T^{(n)}(A^*A)\bah,\bah\rangle_{\cH_{\tau(s_1)}
\oplus\cdots\oplus \cH_{\tau(s_n)}} \geq 0,
\end{equation*}
where $\bah=(h_1,\ldots,h_n)$.
We have shown that $T$ is positive semidefinite.

(2)$\Ra$(3). Let $a\in \cA_t^u$. We first observe that, by replacing $a$ with $a/(1+\|a\|_{s,u})$, 
without loss of generality we can assume that $\|a\|_{t,u}<1$. Let us consider the power series expansion of the 
complex function $(1-\lambda)^{1/2}$ by means of the principal branch of the square root,
\begin{equation*}(1-\lambda)^{1/2}=1-\sum_{n\geq 1}c_n\lambda^n,
\end{equation*}
that converges in the open unit disc of the complex plane, where $c_n$ are real for all $n\geq 1$. Let
\begin{equation*}
b=e_t-\sum_{n\geq 1}c_n (a^*a)^n\in \cA_t^t
\end{equation*}
and observe that, since $\|a^*a\|_{t,t}\leq \|a\|_{t,u}^2<1$, the series converges absolutely. Also,
$b=b^*$ because $(a^*a)^n$ are all selfadjoint, the involution is continuous, 
and the coefficients $c_n$, $n\in \NN$, are all real. 
Then, we observe that this definition complies with the requirements of the functional calculus with holomorphic 
functions in Banach algebras, e.g.\ see \cite{Arveson}, and then, by multiplicativity we see 
that $e_t-a^*a=b^2$. Consequently, since $T$ is positive semidefinite, for arbitrary
$n\in\NN$, arbitrary $b_1\in\cA_{s_1}^t,\ldots,b_n\in \cA^t_{s_n}$, 
for some $s_1,\ldots,s_n\in S_\cA$, and arbitrary $h_1\in\cH_{\tau(s_1)},\ldots,h_n\in\cH_{\tau(s_n)}$, we have
\begin{align*}
0 & \leq \sum_{i,j=1}^n \langle T(b_i^*b^*bb_j)h_j,h_i\rangle_{\cH_{\tau(s_i)}} \\
& =  \sum_{i,j=1}^n \langle T(b_i^*(e_t-a^*a)b_j)h_j,h_i\rangle_{\cH_{\tau(s_i)}} \\
& =  \sum_{i,j=1}^n \langle T(b_i^*b_j)h_j,h_i\rangle_{\cH_{\tau(s_i)}}-
\sum_{i,j=1}^n \langle T(b_i^*a^*ab_j)h_j,h_i\rangle_{\cH_{\tau(s_i)}}.
\end{align*}
This shows that the condition (b) in Corollary~\ref{c:stinespringa} holds with $C(a)=1$ 
and hence $T$ has a dilation $(\bK;\Phi;\bV)$.

The last statements on orthogonality, minimality, and uniqueness of the dilation with these properties are
consequences of Theorem~\ref{t:nagy} as well.

(3)$\Ra$(1). If $(\bK;\Phi;\bV)$ is a dilation of $T$ then, for any $n\in \NN$, the triple $(\bK^n;\Phi^{(n)};\bV^n)$ 
is a dilation of the $n$-th amplification $T^{(n)}$, where 
\begin{equation}\label{e:kan}
\bK^n:=\{\cK_{x_1}\oplus\cdots\oplus\cK_{x_n}\}_{(x_1,\ldots,x_n)\in X^n}\end{equation} 
$\Phi^{(n)}$ is the
$n$-fold amplification of $\Phi$ as in \eqref{e:tena}, and 
$\bV^n=\{V_\bs^{(n)}\}_{\bs\in S^n}$, where
\begin{equation*} V_\bs^{(n)}=\diag(V_{s_1},\ldots,V_{s_n})=\begin{bmatrix} V_{s_1} & 0 & \cdots & 0 \\
0 & V_{s_2} & \cdots  & 0 \\ \vdots & \vdots & \cdots & \vdots \\
0 & 0 & \cdots & V_{s_n}
\end{bmatrix},\quad \bs=(s_1,s_1,\ldots,s_n)\in S^n.
\end{equation*}
Then, for any $\bs=(s_1,s_2,\ldots,s_n)\in S_\cA^n$, any  $\bt=(t_1,t_2,\ldots,t_n)\in S_\cA^n$, and any $X\in{}^n\!\!\cA_\bs^\bt$, we have
\begin{equation*}
T^{(n)}(X^*X)=V_{\bs}^{(n)*} \Phi^{(n)}(X^*X)V_{\bs}^{(n)} =
V_{\bs}^{(n)*} \Phi(X)^{(n)*}\Phi^{(n)}(X)V_{\bs}^{(n)}\geq 0,
\end{equation*}
hence $T$ is completely positive.
\end{proof}

\begin{remark} The implications (1)$\Ra$(2) and (3)$\Ra$(1) in Theorem~\ref{t:stinespring} 
hold true for the general case of a $*$-algebroid $\cA$. This is easily observed by an inspection of the proofs. 
Only the implication (2)$\Ra$(3) is problematic 
in the general setting.
\end{remark}

\subsection{$C^*$-Algebroids.}\label{s:fc}
The linearisations to $*$-algebroids become of more interest for the special case of 
$C^*$-algebroids but, in this special case some obstructions show up. In the previous 
subsection, we somehow circumvented the main obstruction for the more general case 
of $B^*$-algebroids.
Once the concept of a $B^*$-algebroid is defined, it is rather natural to consider 
$C^*$-algebroids: the model is taken from Example~\ref{ex:segah} and 
Remark~\ref{r:beha}. In 
this subsection we tackle some natural questions related to $C^*$-algebroids, as 
applications of Theorem~\ref{t:stinespring}.

\begin{definition} 
If $\cA$ is a Banach algebroid, with its system of submultiplicative norms $\{\|\cdot\|_{s,t}\mid s,t\in S_\cA\}$, and with 
an involution $*$ such that $\|a^*a\|_{s,s}=\|a\|_{s,t}^2$ for all $s,t\in S$ and all
$a\in\cA_s^t$, we call $\cA$ a \emph{$C^*$-algebroid}. Note that, in this case, for each $s\in S_\cA$, the isotropy
algebra $\cA_s^s$ is a $C^*$-algebra.
\end{definition}

Clearly any $C^*$-algebroid is a $B^*$-algebroid.

 \begin{example} With notation and assumptions as in Example~\ref{ex:segah}, letting
 \begin{equation*}
\cB(\bH)=\bigsqcup_{s,t\in S} \cB(\cH_s,\cH_t),
\end{equation*}
we have a $C^*$-algebroid with unit.
\end{example}

\begin{remark} Let $\cA$ be a $C^*$-algebroid. If $B$ is a nonempty subset of $\cA$ then the intersection of 
all $C^*$-subalgebroids of $\cA$ that contain $B$ is a $C^*$-subalgebroid, that we can call the $C^*$-subalgebroid 
of $\cA$ generated by $B$. In particular, let $\Gamma$ be a semigroupoid and $\tau\colon S_\Gamma\ra X$ be an 
aggregation map. Recalling the definition of the aggregated left regular representation $(\tau;L)$ as in 
Example~\ref{ex:lereg} and under the additional assumption \eqref{e:lin}, 
we can consider the $C^*$-subalgebroid generated by the set $\{L(\gamma)\mid \gamma\in\Gamma\}$ 
in the $C^*$-algebroid $\cB(\ell^2_\tau(\Gamma))$. 
If we 
consider a directed graph $G$, the free semigroupoid $\FF^+(G)$, see Example~\ref{ex:fs}, and the aggregation map 
$\tau$ is a singleton range map, hence we have full aggregation, 
then the $C^*$-algebroid generated by $\{L(\gamma)\mid \gamma\in \FF^+(G)\}$ is the 
graph $C^*$-algebra of $G$, see \cite{KribsPower}.
\end{remark}

As in Definition~\ref{d:posa}, 
given a $C^*$-algebroid $\cA$ and $s\in S_\cA$, if $a\in \cA_s^s$ then it is called \emph{positive}
if $a=x^*x$ for some $x\in \cA_s^t$, and in this case we write $a\geq 0$. This definition is needed in view of the 
concept of positive semidefiniteness 
as in Definition~\ref{d:psm}. However, the cone of positive elements in the isotropy
$C^*$-algebra, with the classical definition, $a\in \cA_s^s$ is positive if $a=x^*x$ for some $x\in \cA_s^s$, may be 
smaller than what we have here. So, a natural question is: \medskip

\textbf{Question 1.} \emph{Are the two concepts of positivity in a $C^*$-algebroid equivalent?} 
\medskip

A positive answer to this question would imply that, for the case of a $C^*$-algebroid $\cA$  with unit, the 
proof of the implication (2)$\Ra$(3) in Theorem~\ref{t:stinespring} can be obtained by the classical fact that in any 
$C^*$-algebra the existence of the square root for positive elements is guaranteed. 

Also, with notation as in Definition~\ref{d:ampl}, let $\cA$ be a $C^*$-algebroid and for an arbitrary $n\in\NN$, 
let ${}^n\!\!\cA$ be the $n$-fold amplification of $\cA$, which is a $*$-algebroid. Clearly, for any $n\in\NN$ and any
$\bs,\bt\in S_{{}^n\!\!\cA}$, for each $j,k=1,\ldots,n$, there is a canonical embedding of $\cA_{s_j}^{t_k}$ into
${}^n\!\!\cA_\bs^\bt$ and, with respect to this embedding, we have
\begin{equation}\label{e:nasat}
{}^n\!\!\cA_\bs^\bt=\bigoplus_{j=1}^n\bigoplus_{k=1}^n \cA_{s_j}^{t_k}.
\end{equation}

\textbf{Question 2.} \emph{Is there a system of submultiplicative norms on ${}^n\!\!\cA$, that extend the 
submultiplicative norms of $\cA$ with respect to the embedding \eqref{e:nasat}, 
with respect to which ${}^n\!\!\cA$ becomes a 
$C^*$-algebroid?}\medskip

 In view of the Question~1, a positive answer to Question 2 would imply that the cone of 
positive elements is known and hence the concept of complete positivity is better understood. Let us recall that for
the case of a $C^*$-algebra, what we denote by ${}^n\!\!\cA$ is $M_n\otimes \cA=M_n(\cA)$, which 
is a $C^*$-algebra, 
but the proof of this fact either goes through tensor products of $C^*$-algebras or by the Gelfand-Naimark Theorem.
Consequently, a Gelfand-Naimark Theorem for $C^*$-algebroids would be also of interest. \medskip

\textbf{Question 3.} \emph{Is it true that for any
$C^*$-algebroid $\cA$ with unit there exists a bundle of Hilbert spaces $\bH=\{\cH_s\}_{s\in S_\cA}$ and
a faithful $*$-representation of $\cA$ on $\bH$?}\medskip


\begin{definition} 
If $\cA$ and $\cB$ are normed algebroids and $(\phi;\Phi)$ is an algebroid morphism from $\cA$ to $\cB$, 
see Definition~\ref{d:repa}, we call it 
\emph{bounded} if for every $s,t\in S_\cA$ the linear map $\Phi|\cA_s^t\colon \cA_s^t\ra \cB_{\phi(s)}^{\phi(t)}$ is bounded.

Recall that, see Definition~\ref{d:repa}, if $\cA$ and $\cB$ are $*$-algebroids, the algebroid morphism 
$(\phi;\Phi)$ from $\cA$ to $\cB$ is called a
\emph{$*$-morphism} if $\Phi(a^*)=\Phi(a)^*$ for all $a\in \cA$. Also, recall that, see Definition~\ref{d:smorph},
an algebroid morphism from an algebroid $\cA$ to the algebroid $\cB(\bH)$, for some bundle $\bH$ of Hilbert 
spaces, see Example~\ref{ex:bh}, is called a \emph{representation} of $\cA$ on $\bH$. 
In case $\cA$ is a $*$-algebroid, an algebroid $*$-morphism from $\cA$ to $\cB(\bH)$ is called a 
\emph{$*$-representation} of $\cA$ on $\bH$. 
\end{definition}

The rigidity of $C^*$-algebra morphisms holds for $C^*$-algebroids as well.

\begin{proposition}\label{p:rig} If $\cA$ and $\cB$ are $C^*$-algebroids and $(\phi;\Phi)$ is an algebroid 
$*$-morphism from $\cA$ to $\cB$,
then $(\phi;\Phi)$ is contractive, that is, for any $s,t\in S_\cA$ we have
\begin{equation*}
\|\Phi(a)\|_{\phi(s),\phi(t)} \leq \|a\|_{s,t},\quad a\in \cA_s^t,
\end{equation*}
in particular, it is a bounded algebroid morphism.

In addition, if the algebroid $*$-morphism $(\phi;\Phi)$ is injective then it is isometric, that is, 
for any $s,t\in S_\cA$ we have
\begin{equation*}
\|\Phi(a)\|_{\phi(s),\phi(t)} = \|a\|_{s,t},\quad a\in \cA_s^t,
\end{equation*}
\end{proposition}

\begin{proof} Taking into account that any $*$-morphism of $C^*$-algebras is contractive, we have
\begin{align*}
\|\Phi(a)\|_{\phi(s),\phi(t)}^2 & = \|\Phi(a)^*\Phi(a)\|_{\phi(s),\phi(s)} = \|\Phi(a^*a)\|_{\phi(s),\phi(s)}
\leq \|a^* a\|_{s,s}=\|a\|_{s,t}^2.
\end{align*}

The latter statement follows since any injective $*$-morphism of $C^*$-algebras is isometric, hence the inequality
becomes an equality.
\end{proof}

\begin{definition} Let $\cA$ be a $*$-algebroid. A map $\omega\colon \cA\ra \CC$ is called a \emph{linear form}, or
simply, a \emph{form}, if for any $s,t\in S_\cA$ the restriction of $\omega$ to the vector space $\cA_s^t$ is linear.
The form $\omega$ is called \emph{positive} if $\omega(x^*x)\geq 0$ for all $x\in\cA$.
\end{definition}

\begin{proposition}\label{p:form}
If $\cA$ is a $C^*$-algebroid with unit and $\omega\colon \cA\ra \CC$ is a positive form, then the following
assertions hold true.
\begin{itemize}
\item[(i)] $\omega$ is positive semidefinite.
\item[(ii)] $\omega$ is completely positive.
\item[(iii)] There exists a triple $(\cH_\omega;\xi_\omega;\Phi_\omega)$ subject to the following conditions.
\begin{itemize}
\item[(1)] $\cH_\omega$ is a Hilbert space, $\xi_\omega=\{\xi_{\omega,s}\}_{s\in S_\cA}$ is a bundle of vectors in 
$\cH_\omega$, and $\Phi_\omega\colon \cA\ra \cB(\cH_\omega)$ 
is a fully aggregated $*$-representation of $\cA$.
\item[(2)] $\omega(a)=\langle\Phi_\omega(a)\xi_{\omega,d(a)},\xi_{\omega,c(a)}\rangle_{\cH_\omega}$ 
for all $a\in \cA$.
\item[(3)] For any $a,b\in \cA$ such that $c(a)\neq c(b)$ we have 
$\Phi_\omega(a)\cH_\omega\perp \Phi_\omega(b)\cH_\omega$.
\item[(4)] $\cH_\omega$ is the closed linear span of $\{\Phi_\omega(a)\xi_{\omega,d(a)}\mid a\in \cA\}$.
\end{itemize}
\item[(iv)] $\omega$ is bounded in the sense that, for each $s,t\in S_\cA$, the restriction of $\omega$ to the normed 
space $(\cA_s^t,\|\cdot\|_{s,t})$ is bounded.
\end{itemize}
\end{proposition}

\begin{proof} (i) Let $n\in\NN$, $a_1,\ldots,a_n\in \cA$, and $\lambda_1,\ldots,\lambda_n\in\CC$ be arbitrary. Then,
\begin{align*} \sum_{i,j=1}^n \omega(a_i^*a_j)\ol\lambda_i\lambda_j 
& =  \sum_{i,j=1}^n \omega((\lambda_i a_i)^*\lambda_ja_j) \\
& = \omega\Bigl(\sum_{i,j=1}^n (\lambda_j a_j)^* \lambda_ia_i\Bigr)
= \omega\Bigl(\bigl(\sum_{j=1}^n \lambda_j a_j\bigr)^*
\bigl(\sum_{i=1}^n \lambda_ia_i\bigr)\Bigr) \geq 0,
\end{align*}
hence $\omega$ is positive semidefinite. 

(ii) This assertion follows from Theorem~\ref{t:stinespring}.

(iii) This assertion follows from Theorem~\ref{t:stinespring} since the triple $(\cH_\omega,\xi_\omega,\Phi_\omega)$
is simply an orthogonal minimal dilation of $\omega$.

(iv) In view of Proposition~\ref{p:rig}, the $*$-representation $\Phi_\omega$ is bounded and hence, by (4) it follows
that $\omega$ is bounded as well.
\end{proof}

In the following theorem we show that Question 1, Question 2, and Question 3 are equivalent.

\begin{theorem}\label{t:gn} Let $\cA$ be a $C^*$-algebroid with unit. 
The following assertions are equivalent.

\nr{1} For any $s\in S_\cA$ and any $x\in\cA_s$ there exists $a\in \cA_s^s$ such that $x^*x=a^*a$.

\nr{2} There exists a bundle of Hilbert spaces 
$\bH=\{\cH_s\}_{s\in S_\cA}$ and an injective $*$-representation of $\cA$ on $\bH$, with the aggregation map
the identity map on $S_\cA$.

\nr{3} For each $n\in \NN$ there exists 
a system of submultiplicative norms on ${}^n\!\!\cA$, that extends the system of 
submultiplicative norms of $\cA$ with respect to the embedding \eqref{e:nasat}, 
with respect to which ${}^n\!\!\cA$ becomes a 
$C^*$-algebroid.

\nr{3'} There exists 
a system of submultiplicative norms on ${}^2\!\!\cA$, that extends the system of 
submultiplicative norms of $\cA$ with respect to the embedding \eqref{e:nasat}, 
with respect to which ${}^2\!\!\cA$ becomes a 
$C^*$-algebroid.
\end{theorem}

\begin{proof} $(1)\Rightarrow (2)$. 
With notation and assumptions as in Proposition~\ref{p:form}, from the orthogonal minimal dilation
$(\bH_\omega;\xi_\omega;\Phi_\omega)$ one can obtain a minimal dilation that is aggregation free, in the sense 
that the aggregation map is one-to-one. Indeed, letting
  $e=\{e_s\}_{s\in S_\cA}$ be the unit of $\cA$, it was observed in Remark~\ref{r:cohu} that
$\Phi_\omega(e_s)$, for $s\in S_\cA$, is a bundle of mutually orthogonal projections in $\cH_\omega$ that sum up 
to the identity. So, letting $\cH_{\omega,s}$ denote the range of $\Phi_\omega(e_s)$, we have a bundle of Hilbert
spaces $\mathbf{H_\omega}=\{\cH_{\omega,s}\}_{s\in S_\cA}$. Then we observe that,
for any $a\in \cA$ the right support of $\Phi_\omega(a)$ is in $\cH_{\omega,d(a)}$ and its left support is
in $\cH_{\omega,c(a)}$, hence we can consider $\Phi_\omega(a)\in\cB(\cH_{\omega,d(a)},\cH_{\omega,c(a)})$.
Taking the aggregation map the identity map on $S_\cA$, in this way we obtain a minimal dilation 
$(\bH_\omega;\xi_\omega;\Phi_\omega)$ of $\omega$ that is aggregation free.
  
Given $\omega\colon \cA\ra\CC$ a positive form on the $C^*$-algebroid $\cA$, we call the triple 
$(\bH_\omega;\xi_\omega;\Phi_\omega)$, defined in the previous remark, 
the \emph{cyclic $*$-representation} induced by $\omega$, 
to be in accordance with the classical concept for $C^*$-algebras.

So, let $s\in S_\cA$ and $x\in\cA_s$  be nontrivial but arbitrary. Since $\|x^*x\|_{s,s}=\|x\|_{s,t}^2$, it follows 
that the selfadjoint element $x^*x$ in the isotropy $C^*$-algebra $\cA_s^s$ is not trivial and hence its spectrum 
$\sigma(x^*x)\neq \{0\}$. Let $\lambda_x\in \sigma(x^*x)\setminus\{0\}$ and let $\cB$ be the commutative 
$C^*$-algebra generated by $x^*x$ and $e_s$. Note that, by the spectral permanence of spectra in unital 
$C^*$-algebras, the spectrum of $x^*x$ is the same with respect to $\cA_s^s$ as to $\cB$. Also, 
see e.g.\ \cite{Arveson} at page 26,
$\sigma(x^*x)=\{\omega(x^*x)\mid \omega\in\mathrm{Sp}(\cB)\}$, where $\mathrm{Sp}(\cB)$ denotes the set of all 
nontrivial characters $\omega\colon \cB\ra\CC$. From here we conclude that
there exists $\omega_x\in\mathrm{Sp}(\cB)$ such that $\omega_x(x^*x)=\lambda_x\neq 0$. Then, extend $\omega_x$ to 
a state on $\cA_s^s$, denoted again by $\omega_x$,  see e.g.\  \cite{Arveson} at page 128. Since for every $t\in S_{\cA}$ and every $b\in \cA_s^t$ there exists $a\in\cA_s^s$ such that $b^*b=a^*a$, it follows that $\omega_x(b^*b)\geq 0$ for all $t\in S_\cA$ and all $b\in\cA_s^t$. This enables us to extend $\omega_x$ to a positive form on $\cA$ in the sense of Definition 6.6, denoted again 
by $\omega_{x}$, such that, for any $u,v\in S_\cA$ such that at least one of $u$ and $v$ is not $s$, we have $\omega_{x}|\cA_u^v=0$.

We consider a collection $\{\omega_j\}_{j\in\cJ}$ of positive forms on $\cA$
such that for each 
$x\in \cA$ nontrivial, there exists $j\in\cJ$ such that $\omega_j=\omega_{x}$.
One can define the orthogonal sum of these $*$-representations in 
natural fashion. For each $s\in S_\cA$, we consider the Hilbert space 
$\cH_s=\bigoplus_{j\in\cJ} \cH_{\omega_j,s}$ and the bundle of Hilbert spaces $\bH=\{\cH_s\}_{s\in S_\cA}$. Then
let the $*$-representation 
\begin{equation*}\Phi=\bigoplus_{j\in\cJ} \Phi_{\omega_j}\colon \cA\ra \cB(\bH),
\end{equation*}
be defined by
\begin{equation*} \Phi(a)=
\bigr(\bigoplus_{j\in\cJ} \Phi_{\omega_j}\bigl)(a)=\bigoplus_{j\in\cJ} \Phi_{\omega_j}(a),\quad a\in\cA.
\end{equation*}
In the following we show that $(\mathrm{id}_{S_\cA};\Phi)$ is an injective $*$-representation of $\cA$ on $\bH$, 
where the aggregation map $\mathrm{id}_{S_\cA}$ is the identity map on $S_\cA$.

To see this, let $x\in \cA_s^t$, for some $s,t\in S_\cA$, be nontrivial, hence $x^*x$ is nontrivial, as explained before. 
Let $j\in \cJ$ be
such that $\omega_j=\omega_{x}$ and consider the cyclic $*$-representation 
$(\bH_{\omega_j};\xi_{\omega_j};\Phi_{\omega,j})$. Since
\begin{equation*}
0\neq \omega_j(x^*x)
=\langle \Phi_{\omega_j}(x^*x)\xi_{\omega_j},\xi_{\omega_j}\rangle_{\cH_{\omega_j}},
\end{equation*}
if follows that $\Phi_{\omega_j}(x)^*\Phi_{\omega_j}(x)=\Phi_{\omega_j}(x^*x)\neq 0$, 
hence $\Phi_{\omega_j}(x)\neq 0$. In view of the definition of $\Phi$ it follows that $\Phi(x)\neq 0$. Assertion (2) is 
proven.

$(2)\Rightarrow (3)$. Indeed, if (2) holds, without loss of generality we can assume that $\cA$ is a 
$C^*$-subalgebroid of a $C^*$-algebroid $\cB(\bH)$, for some bundle of Hilbert spaces 
$\bH=\{\cH_s\}_{s\in S_\cA}$.
Since the elements of $\cA$ are bounded linear operators between 
appropriate Hilbert spaces and, consequently, for arbitrary $n\in\NN$, the elements of ${}^n\!\!\cA$ are matrices with 
entries operators on appropriate Hilbert spaces. Then, on ${}^n\!\!\cA$ we consider the operator norms inherited 
from the operator norms of ${}^n\cB(\bH)$ which is a system of submultiplicative norms on ${}^n\!\!\cA$, that extends 
the system of submultiplicative norms of $\cA$ with respect to the embedding \eqref{e:nasat}, with respect to which 
${}^n\!\!\cA$ becomes a $C^*$-algebroid, hence assertion (3) follows.

$(3)\Rightarrow (3)'$. Clear.

$(3)'\Rightarrow (1)$. Let $x\in\cA_s^t$, for some $s,t\in  S_\cA$, $s\neq t$. 
Since $x^*x$ is a Hermitian element in the isotropy $C^*$-algebra $\cA_s^s$, it follows that $x^*x=a_+-a_-$, with 
$a_\pm$ positive elements in the isotropy $C^*$-algebra $\cA_s^s$ and such that $a_+a_-=0$. We 
show that $a_-=0$ and hence that $x^*x$ is positive in the isotropy $C^*$-algebra $\cA_s^s$. Indeed,
\begin{equation*}(xa_-)^*(xa_-)=a_-x^*xa_-=a_-(a_+-a_-)a_-=-a_-^3,\end{equation*}
hence, replacing $x$ with $xa_-$, we can therefore without loss of generality assume that $a_+=0$, 
hence $x^*x=-a_-$. We consider the element $A\in {}^2\!\!\cA$ defined by 
$A=\begin{bmatrix}
a_-^{\frac{1}{2}} & 0 \\
x & 0 
\end{bmatrix} $ and note that 
\begin{equation*}A^*A=    \begin{bmatrix}
a_-^{\frac{1}{2}} & x^* \\
0 & 0 
\end{bmatrix}
\begin{bmatrix}
a_-^{\frac{1}{2}} & 0 \\
x & 0 
\end{bmatrix} =\begin{bmatrix}
a_-+x^*x & 0 \\
0 & 0 
\end{bmatrix} =0.\end{equation*}
Since ${}^2\!\!\cA$ is a C*-algebroid, this implies $A=0$, hence $a_-=0$. 
\end{proof}

The following result shows that a positive answer to these questions may depend on 
how rich the fibres $\cA_s^t$ are.

\begin{proposition} \label{p:rich}
Let $\cA$ be a $C^*$-algebroid with unit and let $s,t\in S_\cA$ be such that the unit 
$\epsilon_t$ of the isotropy $C^*$-algebra $\cA_t^t$
belongs to the closure of the convex cone generated by the set 
$\{uu^*\mid u\in \cA_s^t\}$. Then, for any $b\in \cA_s^t$ the element $b^*b$ is positive
in the isotropy $C^*$-algebra $\cA_s^s$.
\end{proposition}

\begin{proof} The convex cone generated by the set $\{uu^*\mid u\in \cA_s^t\}$ is the 
subset $\cC_s^t$ of the isotropy $C^*$-algebra $\cA_t^t$ described by
\begin{equation*}
\cC_s^t=\{\sum_{j=1}^n u_ju_j^*\mid u_j\in\cA_s^t\mbox{ for all }
j=1,\ldots,n\mbox{ and }n\in \NN\}.
\end{equation*}
If $b\in \cA_s^t$ is an arbitrary element then, for any $v\in \cC_s^t$, letting 
$v=\sum_{j=1}^n u_ju_j^*$, where 
$u_j\in \cA_s^t$, $j=1,\ldots,n$ and $n\in\NN$, since $b^*u_j\in \cA_s^s$, it follows
that the element 
\begin{equation*}
b^*vb=b^*\bigl(\sum_{j=1}^nu_ju_j^*\bigr)b=\sum_{j=1}^n b^*u_ju_j^*b
=\sum_{j=1}^n (b^*u_j)(b^*u_j)^*
\end{equation*}
is positive in the isotropy $C^*$-algebra $\cA_s^s$. By assumption, the unit $\epsilon_t$
can be approximated by a sequence $(v_n)_n$ with $v_n\in \cC_s^t$ and then
\begin{equation*} 
b^*b=b^*\epsilon_t b=\lim_{n\ra\infty} b^*v_nb,
\end{equation*}
is a positive element in the isotropy $C^*$-algebra $\cA_s^s$, since the cone of 
its positive elements is closed.
\end{proof}

\end{document}